\documentclass[11pt,reqno]{amsart}
\usepackage{amsmath,amssymb,amsthm,mathtools}
\usepackage{mathrsfs,bm,color}
\usepackage{fancyhdr,lastpage}
\usepackage{enumitem,calc}
\usepackage{bbm}
\usepackage{graphicx}

\usepackage[utf8]{inputenc}
\usepackage[T1]{fontenc}
\usepackage{lmodern}
\usepackage[normalem]{ulem}
\usepackage{verbatim}
\usepackage{bbm}
\usepackage{stmaryrd}
\usepackage{amsmath}
\usepackage{amssymb}
\usepackage{dsfont}
\usepackage{amsthm}
\usepackage{hyperref}
\hypersetup{
	colorlinks,
	linkcolor={red!50!black},
	citecolor={green!50!black},
	urlcolor={blue!80!black}
}

\usepackage{xcolor}

\def\Tr{{\rm Tr}}
\usepackage{amsmath}%
\usepackage{amsfonts}%
\usepackage{amssymb}%
\usepackage{graphicx}
\usepackage{dsfont}
\newcommand{\Var}{\operatorname{Var}}

\theoremstyle{plain} %plain, definition, remark
\newtheorem{theorem}{Theorem}[section]
\newtheorem{theo}[theorem]{Theorem}
\newtheorem{lem}[theorem]{Lemma}
\newtheorem*{lemma*}{Lemma}
\newtheorem{coro}[theorem]{Corollary}
\newtheorem*{corollary*}{Corollary}
\newtheorem{prop}[theorem]{Proposition}
\newtheorem*{proposition*}{Proposition}

\newtheorem*{assumption*}{Assumption}

\newtheorem*{definition*}{Definition}

\newtheorem*{example*}{Example}

\newtheorem*{remark*}{Remark}
\newtheorem{rem}[theorem]{Remark}

\newtheorem{hyp}[theorem]{Hypothesis}
\def\circt2{{\small{\circ} \atop t }}
\def\circt{{\stackrel{\tiny\circ}{t}}}

\numberwithin{equation}{section}

\newcommand{\E}{\mathbb{E}}
\newcommand{\N}{\mathbb{N}}

\newcommand{\R}{\mathbb{R}}

\renewcommand{\P}{\mathbb{P}}
\newcommand\given[1][]{\:#1\vert\:}
\DeclareMathOperator{\1}{\mathbbm{1}}

\newcommand{\ba}{{\beta}}
\newcommand{\bb}{{\beta_{SNR}}}
\newcommand{\bc}{{\beta_{S}}}

\newcommand{\pert}{\mathrm{pert}}

\newcommand{\cL}{\mathcal{L}}

\newcommand{\cB}{\mathcal{B}}
\newcommand{\cS}{\mathcal{S}}

\newcommand{\sP}{\mathscr{P}}
\newcommand{\cM}{\mathcal{M}}

\newcommand{\cC}{\mathcal{C}}

\newcommand{\cI}{{\mathcal I}}

\renewcommand{\P}{\mathbb{P}}
\newcommand{\mmm}{\mathrel{}\mid\mathrel{}}

\newcommand{\bR}{\bm{R}}

\newcommand{\bQ}{\bm{Q}}
\newcommand{\bM}{\bm{M}}

\newcommand{\bx}{{\bm{x}}}
\newcommand{\bu}{{\bm{u}}}
\newcommand{\bt}{{\bm{t}}}

\newcommand\bg{{\bm{g}}}
\newcommand{\out}{\mathrm{out}}

\renewcommand{\leq}{\leqslant}
\renewcommand{\geq}{\geqslant}
\renewcommand{\epsilon}{\varepsilon}

\def\supp{{\mbox{supp}}}
\def\pP{{\mathbb P}}
\def\pQ{{\mathbb Q}}
\def\pG{{\mathbb G}}

\newcommand{\iii}{i_1,\dots,i_p}
\def\by{{\mathbf{y}}}
\def\bs{{\mathbf{s}}}

\def\R{{\mathbb R}}

\title[Estimating rank-one matrices with mismatched prior]{Estimating rank-one matrices with mismatched prior \\ and noise: Universality and Large Deviations}

\author{Alice Guionnet, Justin Ko, Florent Krzakala, Lenka Zdeborov\'a}

\email{aguionnet@ens-lyon.fr}
\email{justin.ko@ens-lyon.fr}
\email{florent.krzakala@epfl.ch}
\email{lenka.zdeborova@epfl.ch}

\thanks{This project   has received funding from the European Research Council (ERC) under the European Union
	Horizon 2020 research and innovation program (grant agreement No. 884584), as well as from the Swiss National Science Foundation grant SNFS OperaGOST, $200021\_200390$.}

\begin{document}
	\maketitle
\begin{abstract}
We prove a universality result that reduces the free energy of rank-one matrix estimation problems in the setting of mismatched prior and noise to the computation of the free energy for a modified Sherrington-Kirkpatrick spin glass. Our main result is an almost sure large deviation principle for the overlaps between the truth signal and the estimator for both the Bayes-optimal and mismatched settings. Through the large deviations principle, we recover the limit of the free energy in mismatched inference problems and the universality of the overlaps.
\end{abstract}
 
	\section{Introduction}
Estimating factors of noisy low-rank matrices is a fundamental problem with many applications in machine learning and statistics. Consider the following probabilistic rank-one matrix estimation problem: one has access to noisy observations $Y_{ij}$ of a $N \times N$ rank-one matrix $w_{ij}= \frac{x^0_ix^0_j}{\sqrt{N}}$, and the goal is to estimate the vector $\bx^0 \in \R^N$ using either a Bayesian or a maximum likelihood approach. Many important problems in statistics and machine learning can be expressed in this way, such as sparse PCA \cite{zou2006sparse}, the Wigner spiked model \cite{johnstone2009consistency}, community detection \cite{deshpande2017asymptotic}, matrix completion \cite{candes2012exact}, submatrix localization \cite{banks2018information}, or synchronization \cite{javanmard2016phase}.

There have been a number of results for such problems in the {\it Bayes-optimal} case, where the statistician knows both the prior information on ${\bx}^0$ (the prior $\pP_0$), and the statistics of the noise (the likelihood $\pP_\out(Y\given w$)), and is able to characterize the information-theoretically optimal performance \cite{deshpande2017asymptotic,krzakala2016mutual,dia2016mutual,lelarge2017fundamental,el2018estimation,mourrat2021hamilton}. In this paper, we consider the more difficult task of characterizing the asymptotic performance for estimators that  mismatch the prior and the noise distribution, including Bayesian ones where one assumes a prior or a noise distribution different from the ones that were used to generate the data, as well as a risk minimization approaches where one optimized the assumed likelihoods.

This generality comes with increased technical difficulty. While the Bayes optimal approach can be asymptotically characterized by a {\it replica symmetric} (to use spin glass theory terminology \cite{mezard2009information,abbe2013conditional,COJAOGHLAN2018694}) formula providing the asymptotic mutual information that can be proven rather simply \cite{el2018estimation}, the general situation we consider here, and discussed in the physics literature in \cite{lenkarmt}, requires an approach reminiscent of the Parisi formula for the Sherrington-Kirkpatrick model \cite{talagrand2006parisi}. To overcome these difficulties, we prove a finer result where we estimate constrained free energies and obtain a
quenched large deviation principle. Our main objectives are three-fold: First, we aim to establish a general replica symmetry-breaking formula for these models, irrespective of non-matching prior, different channels, or mismatching noise. Second, we seek to derive a formula for the large deviation of the overlap, also known as the Franz-Parisi potential in statistical physics. Lastly, we aim to demonstrate the universality of these formulas across various types of noises, and in particular, the universality of Gaussian noise.
	
We summarize our results and their application below:	
	\begin{itemize}
            \item In Theorem~\ref{mainfree}, we provide an asymptotic Parisi-type formula for the free energy of the rank-one matrix estimation problem for any mismatching separable prior and noise. This allows characterizing the asymptotic performances of  empirical Bayesian setting when one does not know the parameters and to study the maximum a posteriori error (MAP) (minimum of the loss) under various hypotheses, which is the classical statistics approach. In particular, this proves the conjecture for the free energy from  \cite{lenkarmt}.
		\item In Theorem~\ref{quenchedldp} we also provide a large deviation principle for the order parameters of the problems, particularly for the overlap between the reconstructed signal and the original ground truth one. This quantity is called the Parisi-Franz potential in statistical physics \cite{franz1997phase}, and such a result has its own interest as it generalizes the results of \cite{PVS}. The Parisi-Franz potential is thought to be fundamental in understanding the computationally easy-hard (or information-algorithmic gaps) transition \cite{zdeborova2016statistical,bandeira2022franz}.
  
  \item We further show that the large deviation as a function of the overlap is universal and depends on the actual and assumed likelihood only through their (generalized) Fisher information so that the Gaussian noise case can capture any separable likelihood. This leads to a very strong universality principle, conjectured in \cite{lesieur2015mmse} and generalizing the more limited Bayes optimal one \cite{krzakala2016mutual}. Concretely, this means 
  that an entire set of problems and noisy output (including community detection, Laplace noise, submatrix localization and others...) have universal Bayes-optimal error corresponding to those of a Gaussian problem (note that, previously, only the universality of the free energy was proven, not one of the  overlaps). This  also shows that the maximum a posteriori error (so-called MAP) is also universal, given the zero temperature large deviation is also universal. 
	\end{itemize}

%\subsection{Technical Contributions and Open Problems}

Mismatched inference problems have been the focus of several recent works \cite{Camillimismatch,barbier2021performance,barbier2022price,farzadmismatch}. In \cite{farzadmismatch}, the authors considered Gaussian prior and Gaussian additive noise and mismatched the variances of both, the technical difficulty is then solved using rotation invariance of the priors and spherical integrals. Authors of \cite{barbier2022price} studied a similarly specific rotationally invariant estimator through approximate message passing. Authors of \cite{Camillimismatch} also considered a Gaussian mismatched problem analogue to a spin glass model with a Mattis interaction, as a proof technique, they used adaptive interpolation with Rademacher priors. The authors of \cite{barbier2021performance} study a different mismatched problem by relating it to the Shcherbina-Tirozzi spin glass model which is not directly related to the matrix estimation problems we consider in this work. Compared to these works, we consider a much more generic mismatch and prove a strong universality result that reduces generic mismatched estimation problems to a Gaussian framework, which was not studied in the above works. We analyze these Gaussian models and the respective overlaps from a large deviations point of view. These overlaps are central objects and often encode the behaviors of optimal estimators. This framework also allows us to consider general factorized priors. Through a covering argument we are able to use the large deviations principle to recover a formula for the free energy expressed as an iterated variational formula. Similar variational formulas for the free energy have been previously obtained for generic versions of finite rank estimation problems  in the Bayesian optimal setting \cite{  Reeves_tensor, MourratXia-tensor, MourratXiaChen-tensor,barbiertensor,lelarge2017fundamental,JCThomasSparse} and in the Bayesian optimal setting with generic noises in \cite{alberici2021multi,alberici2021multi_boltzmann,alberici2022statistical, BehneReeves-hetero,AJFL_inhomo}. Our work includes the Bayesian optimal setting but also explores the general setting where the Nishimori identity may not hold.

The large deviations principle studied in this work is closely connected with the limit of the Franz--Parisi potential of a general Hamiltonian consisting of the sum of the usual SK Hamiltonian, a magnetization term, and a self-overlap term. This large deviations principle encompasses several previously studied spin glass models \cite{PVS,chenferro,Camillimismatch, lelarge2017fundamental,erbasubmatrix} (see Section~\ref{sec:example} for a detailed discussion).  The main technical contribution of this paper is a unifying formula for the limit of the Franz--Parisi potential. Similar large deviation principles or Franz--Parisi potentials for classical spin glass models, without the magnetization term and self-overlap terms have also been computed and applied in several areas \cite{PanchenkoChaos,ChenChaos2,JagSpectralGap,JagSphericalSpectralGap,JagMetastable,FlorentAhmedFranzParisi,Franzlargedeviations,JagStatisticalThresholds,BFK_TAP, kosphere, kocs,ThomasVector,PVS}.

The large deviations rate function we compute is a powerful tool to understand phase diagrams of such inference problems. For instance, one can hope to obtain a characterization of the replica symmetric regime outside of the Nishimori line in inference problems, by studying the minimizers of such functionals. The phase diagram in a subset of models in our class of free energies were previously studied in \cite{ChenTucaParisi,JagganathTobascoPhasediagram}. Generalizations to higher rank models and spiked tensor estimation problems can also be proved in the future.

In contrast to classical spin glass models, the main technical difficulty is the localization around the overlap $R_{10} = \frac{\bx \cdot \bx^0}{N}$ between the signal and estimator. Standard techniques to control this overlap rely on the Nishimori identity and the concentration of overlaps \cite{Barbier_overlapconc}, which are not applicable in the mismatched setting. Instead, we combine techniques from large deviations, and spin glasses to study a localized version of the constrained optimization problem. When localized around configurations with finite entropy, we are able to do a smooth approximation of the free energy with respect to the overlaps, allowing us to regularize the laws of the overlaps through a perturbation enforcing a localized version of the Ghirlanda--Guerra identities \cite{GGI} and in particular ultrametricity of overlaps \cite{PUltra} is adapted to our setting. As such, we are able to obtain a Parisi type variational formula \cite{parisi1979infinite,talagrand2006parisi}. We then use uniform bounds on our estimates to remove the localization and recover an almost sure large deviations principle, solving the original problem.  The uniform control along the boundary points of configurations with finite entropy were particularly difficult, but this was resolved using tools from large deviations such as exposed hyperplanes and Rockafellar's theorem. 

The paper is structured as follows. In the next Section~\ref{sec:setting} we discuss the setting and the main theorems. We then present the proofs of theorems on universality in Section~\ref{secuniversality}, of the large deviation upper bound in Section~\ref{sec:upbd} and lower bound in Section~\ref{sec:ldlb}, and of the expression of the  free energy in Section~\ref{proofthmquenched}.
	
\section{Setting and main theorems}
\label{sec:setting}
We now formally describe the problem and our results. We consider  non-Bayes optimal inference for rank one statistical inference problems, where we want to recover a rank one signal observed via some arbitrary separable noise in the presence of an arbitrary separable prior information. In this setting, the statistician does not have perfect information, so the posterior distribution may not be the optimal one in these models. Our main goal is to prove the general replica symmetry-breaking formula for the free energy of these models. 
	
	In the inference problems, we want to study  Boltzmann-Gibbs measures of the form {
		\begin{equation}\label{BG}
			d\pG_N^{Y}(\bx) = \frac{1}{Z_X(Y)}  \prod_{1 \leq i< j \leq N} e^{g( Y_{ij}, \frac{x_i x_j}{\sqrt{N}}  )} \prod_{1 \leq i \leq N} d\pP_X(x_i).
		\end{equation}
		In the sequel, we will assume that $\pP_{X}$  is a  probability measure supported in the compact set $[-C,C]$ of  the real line. }
	We consider that the ground truth signal was generated using the distribution $\pP_0$ and the observed data $Y$ were generated from some output channel $\pP_\out\big(Y \given \frac{x_i x_j}{\sqrt{N}}\big)$. In the Bayes optimal case, $\pP_X$ and  $\pP_\out\big( Y_{ij}\given \frac{x_ix_j}{N}\big)$ are known to the statistician, i.e.
	\[
	d\pP_X = d\pP_0 \quad \mbox{ and }\, \, d\pP_\out(Y\given w) = e^{g(Y,w)} dY
	\]
	so that in this case
	\[
	d \pG_N^{Y}(\bx) = d\pP(X = \bx \given Y) = \frac{1}{Z_X(Y)}  \prod_{1 \leq i < j \leq N} \frac {d\pP_\out\big( Y_{ij}\given[\big] \frac{x_ix_j}{\sqrt{N}}\big)}{dY} \prod_{1 \leq i \leq N}d \pP_0(x_i).
	\]
	
	The main consequence is that in the Bayes optimal setting, the average with respect to the Gibbs measure are generated from the true posterior distribution. As an important consequence, 
	the Nishimori property holds, allowing 
	to replace the signal with a uniform sample from the posterior and vice versa, a key step to use classical spin glass theory to estimate the free energy $\frac{1}{N}  \E_Y \log Z_X(Y)$. We want to compute the free energy $\frac{1}{N}  \E_Y \log Z_X(Y)$ and study the Boltzmann-Gibbs measures $\pG_{N}^{Y}$
	in the general case when
	\[
	\pP_X \neq \pP_0 \quad\text{or}\quad g(Y,w) \neq \ln \frac{d\pP_\out(Y\given w)}{dY}=:g^{0}(Y,w) .
	\]
	The main technical consequence is that for this Gibbs measures the Nishimori property may not hold. As we will see, the standard overlap concentration proofs fail, so we will have to invoke the general Ghirlanda--Guerra identities to observe replica symmetry breaking in these models.

	\subsection{Main results}
	We define$$
	Z_N^Y = \int e^{\sum_{ij} g(Y_{ij} | \frac{x_i x_j}{\sqrt{N}}) } \, d\pP^{\otimes N}_X(\bx) 
	$$
	and given {a sequence of measurable sets  $A = A(\bx^0) \subset \R^N$ that may depend on $\bx^0$ but not on $W$,} we define the corresponding constrained partition function 
	\[
	Z_N^Y(A) = \int \1(\bx\in A) e^{\sum_{ij} g(Y_{ij} | \frac{x_i x_j}{\sqrt{N}}) } \, d\pP^{\otimes N}_X(\bx) .
	\]
	In this article we study the Boltzmann--Gibbs measure 
	
	%$$
	%d\pG_{N}^{Y} \left( \bx\right)=\frac{1}{ Z_{N}^{Y}} e^{\sum_{ij} g(Y_{ij} | \frac{x_i x_j}{\sqrt{N}}) } \, d\pP^{\otimes N}_X(\bx),$$
	\[
	\pG_N^Y(A) = \frac{Z_N^Y(A)}{Z_N}
	\]
	so that $\pG_{N}^{Y} (A)=Z_N^Y(A) /Z_{N}^{Y}$. We 
	prove an almost sure large deviation principle for the law of the overlaps 
	$$R_{1,1}=\frac{1}{N}\sum_{i=1}^{N} x_{i}^{2}, \quad R_{1,0}=\frac{1}{N}\sum_{i=1}^{N} x_{i}x_{i}^{0}$$
	under $\pG_{N}^{Y}$ as well as universality of these large deviations. 
	The main steps is to compute the free energy
	\begin{equation}\label{eq:FEgrowingrank}
		F_N(g) = \frac{1}{N} \bigg( \E_Y  \Big(\log Z_N^Y - \sum_{i < j} g(Y_{ij},0) \Big) \bigg)
	\end{equation}
	and  the constrained free energy
	\begin{equation}\label{eq:FEgrowingrank}
		F_N(g: A) = \frac{1}{N} \bigg( \E_Y  \Big(\log Z_{N}^{Y}(A) - \sum_{i < j} g(Y_{ij},0) \Big) \bigg)
	\end{equation}
	where $A=\{ \bx: (R_{1,1}, R_{1,0})\in F \}$ where $F$ is a measurable set of $\mathbb R^{2}$. 
	We need to subtract the $ g(Y_{ij},0)$ terms otherwise the free energy will  not grow on the order $N$. The second summation is trivial to compute. We begin by describing our technical hypotheses. We first need to assume that the signal is compactly supported.
	\begin{hyp}[Compact Support]\label{hypcompact}
		$\pP_0$ and $\pP_X$ are compactly supported probability measures on the real line so that $x_0$ and $x$ take values in $[-C,C]$ for some finite $C$.
	\end{hyp}
	This hypothesis  implies that, uniformly, we have 
	\begin{equation}
		|w_{ij}| = \Big| \frac{ x_i x_j}{\sqrt{N}} \Big| \leq \frac{C^2}{\sqrt{N}} .
	\end{equation}
	This uniform bound will allow to expand the function $g$ in the variables $w_{ij}$. To do so, we need to assume sufficient regularity of the function $g$, namely that, if $\|\cdot \|$ denotes the supremum norm: 
	\begin{hyp}[Regularity]\label{hypg}
		The function $g(Y,w)$ and $g^0(Y,w)$ is three times differentiable in the $w$ coordinate and twice differentiable respectively and
		\[
		\E_{\pP_\out(Y \given 0)} [ (\partial_w g(Y ,0) )^3 ],~ \|\partial_w^{(2)}g(\cdot ,0)\|,~ \|\partial_w^{(3)}g\|, ~\|\partial_w g(\cdot ,0)\|, ~\|\partial_w^{(2)}g(\cdot ,0)\| 
		\]
		are bounded.
	\end{hyp}
	Our last hypothesis is a requirement for our function $g$ to be a consistent estimator of $\pP_0$, namely
	\begin{hyp}[Consistent Estimator]\label{hypderiv}
		For $Y \sim \pP_\out(Y\given 0)$,
		\[
		\E_{\pP_{\out} (Y\given 0)} \big[ \partial_w g(Y,0) \big]  = \int \partial_{w}g(y,0) d\pP_{\out}(y \given 0) = \int \partial_w g(y,0)  e^{g_{w}^0 (y,0)} dy = 0 .
		\]
	\end{hyp}
	For example, if $g$ corresponds to the classical rank 1 Gaussian estimation problem, then this requirement is equivalent to assume that our guess of the output distribution $\pP_{\out}(\cdot\given0)$ in the absence of a signal is centered.   Without this hypothesis,
	the normalized free energy diverges. If $\E_{\pP_\out(Y|0)} [  \partial_w g(Y,0) ] \neq 0$, then we would need to normalize the free energy by $N^{3/2}$ instead of $N$. In the Bayes optimal case, this condition is automatically satisfied.

	Under these technical restrictions on $g$ and $g^0$, we are able to reduce the non-Bayes optimal problem to the appropriate Gaussian estimation problem with generalized covariance. Consider the Hamiltonian given for  3 real numbers $\bar\ba=(\ba, \bb,\bc)$ by: 
	\begin{align}\label{eq:nonbayesoptimalHamiltonian1}
		H_N^{{\bar\ba}}(\bx) &= \sum_{i<j} \ba \frac{W_{ij}}{ \sqrt{  N}} x_ix_j + \frac{\bb}{ N} (x_ix_j)(x_i^{0} x_j^{0}) +  \frac{\bc}{2N}  (x_i x_j )^2
	\end{align}
	where $W_{ij}$ are iid standard Gaussians and the covariance parameters are given by \eqref{eq:delta1},\eqref{eq:delta2} and \eqref{eq:delta3}. The Gibbs measure associated with this Hamiltonian is denoted by
	\begin{equation}\label{BGbeta}
		\pG^{\bar\beta}_N(x) = \frac{1}{ Z^{N} _{Y}} e^{H_N^{{\bar\ba}}(\bx)} \, d\pP^{\otimes N}_X(\bx)
	\end{equation}
	
	The corresponding free energy associated with \eqref{eq:nonbayesoptimalHamiltonian1} is given by
	\begin{equation}\label{deff}
		F_N({\bar\ba}) = \frac{1}{N} \E_Y  \Big(\log \int e^{H_N^{{\bar\ba}}(\bx)} \, d \pP_X^{\otimes N}(\bx)  \Big).
	\end{equation}
	To state a large deviations principle, given any set $A = A_N(\bx^0) \subset \R^N$ that can depend on $x^0$, we also define the constrained free energy in the spirit of $F_{N}(g:A)$ defined in \eqref{eq:FEgrowingrank}
	
	\[
	F_N( {\bar\ba} : A) = \frac{1}{N} \bigg( \E_Y  \Big(\log \int \1(\bx\in A) e^{H_N^{{\bar\ba}}(\bx)} \, d \pP_X^{\otimes N}(\bx)   \Big) \bigg). 
	\]
	
	We first state a universality result that will imply that the free energy of general inference models are equivalent to the free energy of the Gaussian estimation problem.
	
	\begin{prop}[Universality] \label{prop:universality1}
		If Hypothesis \ref{hypcompact}, \ref{hypg}, and \ref{hypderiv} hold,
		then the free energy of the vector spin models satisfy for $N$ large enough 
		\[
		\big| F_N(g) - F_N(\bar\beta) \big| =  O( N^{-1/2} )
		\]
		where $\bar\beta=(\ba,\bb,\bc)$ is given by 
		\begin{itemize}
			\item 
			\begin{equation}\label{eq:delta1}
				\ba = \left[ \E_{\pP_\out(Y|0)} \bigg[ (\partial_w g(Y,0))^2 \bigg] \right]^{\frac{1}{2}}
			\end{equation}
			\item 
			\begin{equation}\label{eq:delta2}
				\bb =  \E_{\pP_\out(Y|0)} \bigg[ \partial_w g(Y,0) \partial_w \ln\pP_\out(Y\given 0) \bigg]
			\end{equation}
			\item
			\begin{equation}\label{eq:delta3}
				\bc = \E_{\pP_\out(Y|0)} \bigg[  \partial_{w}^{2} g(Y,0) \bigg].
			\end{equation}
		\end{itemize}
		
		More generally, for any sequence of measurable sets  $A = A_N(\bx^0) \subset \R^N$ such that $
		\liminf_{N \to \infty} F_{N}(0:A)>-\infty$,
		\[
		\big| F_N(g:A) - F_N(\bar\beta:A) \big| =  O( N^{-1/2} ).
		\]
	\end{prop}
	
	\begin{rem}
		In the Bayes optimal case when $g(Y,w) = \ln \pP_{\out} (Y \given w)$, these parameters simplify greatly
		\[
		\ba^{2} = \bb= - \bc.
		\]
		%And in particular, the covariance in \eqref{eq:nonbayesoptimalHamiltonian} only depends on one parameter instead of three.
	\end{rem}

	Our main goal is to compute the limit of the free energy
	$ F_N(\bar\beta:A).
	$
	To state our main theorem, let us describe its limit which is defined in the spirit of Parisi formula.  It will depend   on the functional order parameters which  are increasing sequences such that for some integer number $r$
	\begin{equation}\label{eq:zetaseq1}
		\zeta_{-1} = 0 < \zeta_0< \dots < \zeta_{r-1} < 1
	\end{equation}
	and
	\begin{equation}\label{eq:Qseq1}
		0 = Q_0 \leq Q_1 \leq \dots \leq Q_{r-1} \leq Q_r =  S.
	\end{equation}
	For good choices of sequences \eqref{eq:zetaseq1} and \eqref{eq:Qseq1}, these sequences can be interpreted as a discrete approximation of the limiting  distribution of the overlap
	\[
	R_{1,2} = \frac{1}{N}\sum_{i = 1}^N x_i^1 x_i^2
	\]
	of two replicas $\bx^1$ and $\bx^2$ from the limiting constrained Gibbs measure,  given for any measurable subset $A$ of $\mathbb R^{N}$. 
	To define our limits, we first recursively define a random variable coming from Ruelle probability cascades \cite[Chapter 2]{PBook}. We start by defining recursively the random variables $X_r, X_{r - 1}, \dots, X_0$ that depend on $x^{0}$, the sequences \eqref{eq:zetaseq1} and \eqref{eq:Qseq1}, and real parameters $\lambda,\mu$.  
	Let $X_r$ be the random variable
	\[
	X_r = \log  \int e^{\ba \sum_{j = 1}^r z_i x  + \lambda x^2 + \mu x x^0  } \, d \pP_X (x)
	\]
	where $z_j$ are Gaussian random variables with covariance
	\[
	\Var(z_j) = Q_j - Q_{j-1}
	\]
	and $x^0$ is an independent random variable with distribution $\pP_0$.
	We define recursively for $0 \leq p \leq r-1$ the random variables 
	\begin{equation}\label{recur}
		X_j = \frac{1}{\zeta_j} \log \E_{z_{j + 1}} e^{\zeta_j X_{j + 1}}.
	\end{equation}
	We let $X_{0}=X_{0}(\lambda,\mu,Q,\zeta)[x^{0}]$ be the resulting function of $x^{0}$ at $j=0$. 
	We finally define the   function $\varphi_{\bar\ba}$ on  $[0,C]^{+}\times [-C,C]$  given by
	\begin{equation}\label{defvarphi}
		\varphi_{\bar\beta}(S,M)=\inf_{\mu,\lambda,\zeta,Q} \bigg( \E_{0} [X_0(\lambda,\mu,Q,\zeta)] - \mu S - \lambda M - \frac{\ba^2}{4} \sum_{k = 0}^{r - 1} \zeta_k ( Q^2_{k + 1} - Q^2_k ) + \frac{\bb M^2}{2}  + \frac{\bc S^2}{4}  \bigg),\end{equation}
	where the infimum is over $\mu,\lambda \in \R$, $r$ levels of symmetry breaking, and sequences $\zeta$ and $\bQ$ satisfying \eqref{eq:zetaseq1} and \eqref{eq:Qseq1}. 
	The average $\E_0$ is with respect to $\pP_0$ since the recursive quantity $X_0$ depends on $x_0$. 
	One of our main theorems is  the following estimates on the free energies of our model:
	\begin{theo}[Limit of the Free Energy]\label{mainfree} 
		For any real numbers ${\bar\ba}=(\ba,\bb,\bc)$,
		\[
		\lim_{N \to \infty} F_N (\bar\ba) = \sup \varphi_{\bar\ba}.	\]
		In particular, if Hypothesis \ref{hypcompact}, \ref{hypg}, and \ref{hypderiv} hold, with $\beta$ given by \eqref{eq:delta1},\eqref{eq:delta2} and \eqref{eq:delta3}, we have
		$$\lim_{N \to \infty} F_N (g)  = \sup \varphi_{\bar\ba}.$$
	\end{theo}

	The main difficulty to prove this theorem compared to, e.g. \cite{deshpande2017asymptotic,krzakala2016mutual,dia2016mutual,lelarge2017fundamental,el2018estimation} is that usual concentration of overlaps do not apply because we are outside the Nishimori line. To overcome this difficulty, we prove a finer result where we estimate constrained free energies and  obtain
	a  quenched large deviation principle.
	In fact, Theorem \ref{mainfree} extends to the restricted free energies $F_{N}(\bar\beta:A)$ and also holds almost surely. This leads us to the following quenched  large deviation principles. We recall the definition \eqref{BG} of the Boltzmann-Gibbs measure $\pG_{N}^{Y}$ as well as definition  \eqref{BGbeta} of the Boltzmann-Gibbs measure $\pG^{\bar\beta}_N$. We will consider the couple $(R_{1,1},R_{1,0})$ of the overlaps:
	$$R_{1,1}=\frac{1}{N}\sum_{i=1}^{N}(x_{i})^{2},\qquad
	R_{1,0}=\frac{1}{N}\sum_{i=1}^{N} x_{i}x_{i}^{0}\,.$$ We will see that $(R_{1,1},R_{1,0})$ asymptotically live 
	in the closed subset $\mathcal C$  of $[0,C^{2}]\times [-C^{2},C^{2}]$  given by 
	\begin{equation}\label{defC}\mathcal C=\cap_{\rho,t\in [-1,1]^{2}}\{(S,M):  \E_{x^{0}}[\mbox{essinf}_{x}\{\rho x^{2}+ tx x^{0}\}] \le \rho S+tM \le  \E_{x^{0}}[\mbox{esssup}_{x}\{\rho  x^{2}+ tx x^{0}\}]\}.\end{equation}
We note that $\varphi_{\bar\ba}$ is equal to $-\infty$ if $(S,M)$ does not belong to $\mathcal C$ by taking $Q_{1}=\cdots=Q_{r-1}=0$ and $\lambda,\mu$ going to infinity. Our main theorem is the following quenched large deviation principle:
	\begin{theo}\label{quenchedldp}  For every real numbers $\bar\ba=(\ba,\bb,\bc)$, the law of $(R_{1,1},R_{1,0})$ under $\pG^{\bar\beta}_N$ satisfies an almost sure large deviation principle with speed $N$ and good rate function $I^{FP}_{\bar\ba}$ which is infinite if $(S,M)$ do not belong to $\mathcal C$ and otherwise is given by
		$$I_{\bar\beta}^{FP}(S,M)=-\varphi_{\bar\ba}(S,M)+\sup_{(s,m)\in\cC}\varphi_{\bar\ba}(s,m)\,.$$
		In other words, 
		\begin{itemize}
			\item $I_{\bar\beta}^{FP}$ is a good rate function in the sense that its level sets 
   \[
    \{ (S,M) \mmm I_{\bar\beta}^{FP}(S,M) \leq L\}
   \]
   are compact for all $L \geq 0$. 
			\item for any closed subset $F$ of $\mathbb R^{2}$, for almost all $(W,\bx^{0})$,
			$$\limsup_{N\rightarrow\infty}\frac{1}{N}\log \pG^{\bar\beta}_{N}((R_{1,1},R_{1,0})\in F)\le-\inf_{(S,M)\in F} I^{FP}_{{\bar\beta}}(S,M)$$
			\item  for any open subset $O$ of $\mathbb R^{2}$, for almost all $(W,\bx^{0})$,
			$$\liminf_{N\rightarrow\infty}\frac{1}{N}\log \pG^{\bar\beta}_{N}((R_{1,1},R_{1,0})\in O)\ge-\inf_{(S,M)\in O} I_{\bar\beta}^{FP}(S,M)\,.$$
		\end{itemize}
		%Similarly, if Hypothesis \ref{hypcompact}, \ref{hypg}, and \ref{hypderiv} hold, and $(\ba,\bb,\bc)$ are given by \eqref{eq:delta1},\eqref{eq:delta2} and \eqref{eq:delta3},  the law of $(R_{1,1},R_{1,0}) $ under $\pG_{N}^{Y}$ satisfies an almost sure  large deviation principle with speed $N$ and good rate function $I_{\bar\beta}^{FP}$.

	\end{theo}
	
	\begin{rem}
		This theorem is new as well for the SK model where $\bb=\bc=0$, see \cite{PVS} for the large deviations of $R_{11}$.
	\end{rem}

        Combining this with the universality of the free energy in Proposition~\ref{prop:universality1}, we immediately arrive at a LDP for the overlaps under the Boltzmann--Gibbs measure $\pG_{N}^{Y}$. This large deviations principle is universal in the sense that it only depends on the model parameters $\bar \ba = (\ba,\bb,\bc)$.

        \begin{coro}
        \label{coro2.9}
            If Hypothesis \ref{hypcompact}, \ref{hypg}, and \ref{hypderiv} hold, and $(\ba,\bb,\bc)$ are given by \eqref{eq:delta1},\eqref{eq:delta2} and \eqref{eq:delta3},  the law of $(R_{1,1},R_{1,0}) $ under $\pG_{N}^{Y}$ satisfies an almost sure  large deviation principle with speed $N$ and good rate function $I_{\bar\beta}^{FP}$.
        \end{coro}

        If the rate function has a unique minimizer, then we also arrive at concentration of the overlaps under the generic $\pG_{\bar\beta}^N$ and the Boltzmann--Gibbs measure $\pG_{N}^{Y}$. This almost sure limit of the overlaps only depend on $(\ba,\bb,\bc)$. 
 
	\begin{coro}
 \label{coro2.10}
		If the rate function $I_{\bar\beta}^{FP}$ has a unique minimizer $(S_{{\bar\beta}},M_{{\bar\beta}})$, then $(R_{11},R_{10})$ converges almost surely towards $(S_{{\bar\beta}},M_{{\bar\beta}})$ under $\pG^{\bar\beta}_N$, but also under $\pG_{N}^{Y}$ for any $g$ satisfying Hypothesis \ref{hypcompact}, \ref{hypg}, and \ref{hypderiv}  and so that $(\ba,\bb,\bc)$ are given by \eqref{eq:delta1},\eqref{eq:delta2} and \eqref{eq:delta3}.
	\end{coro}
	Theorem \ref{quenchedldp} will be derived from exponential tightness,  concentration of measure and 
	an averaged convergence of the restricted free energies which reads as follows. We let 
	$$\mathcal B_{\delta}=\{\bx^{0}\in\mathbb R^{N}:
	d(\frac{1}{N}\sum_{i=1}^{N}\delta_{x_{i}},\mathbb P_{0})\le \delta\}$$
	and for $(S,M)\in\mathcal C$ we set
	$$\Sigma_{\epsilon}(S,M)=\{(\bx,\bx^{0})\in [-C,C]^{2N}:|R_{1,1}-M|\le\epsilon, |R_{1,0}-S|\le \epsilon\}\,.$$
	Our main technical result is the following:
	
	\begin{theo}\label{technicalldp}For every real numbers $\bar\ba=(\ba,\bb,\bc)$, 
		every $(S,M)\in\mathcal C$,
		$$\varphi_{\bar\beta}(S,M)\le \lim_{\epsilon\downarrow 0}\lim_{\delta\downarrow 0}\liminf_{N\rightarrow\infty}\frac{1}{N}  \E_Y 1_{\mathcal B_{\delta}}  \Big(\log \int \1(|R_{1,1}-S|\le \epsilon, |R_{1,0}-M|\le\epsilon) e^{H_N^{{\bar\ba}}(\bx)} \, d \pP_X^{\otimes N}(\bx)   \Big) \qquad\qquad$$
		$$\qquad \le \lim_{\epsilon\downarrow 0}\limsup_{N\rightarrow\infty}\frac{1}{N}  \E_Y  \Big(\log \int \1(|R_{1,1}-S|\le \epsilon, |R_{1,0}-M|\le\epsilon) e^{H_N^{{\bar\ba}}(\bx)} \, d \pP_X^{\otimes N}(\bx)   \Big) \le\varphi_{\bar\ba}(S,M)\,.$$
	\end{theo}
	
	The upper bound of this theorem is proven in  section \ref{sec:upbd}, whereas the lower bound is proven in section  \ref{sec:ldlb}.
 \iffalse
 \begin{rem}
		If $\pP_X$ is standard Gaussian, then we can get a slightly cleaner closed form of this formula. If we write
		\begin{align*}
			X_r = \frac{1}{\sqrt{2 \pi}} \log  \int e^{\big( \ba \sum_{j = 1}^r z_j  + \mu x^0 \big)x - \frac{1}{2} \big(  - 2\lambda + 1 \big) x^2  } \, dx
		\end{align*}
		then it can be explicitly computed following the same computations as the spherical spin glass models. For $\lambda$ such that the coefficient in front of the $x^2$ is negative, that is
		\[
		- 2\lambda + 1 > 0
		\]
		we can use the usual moment generating function to see that
		\[
		X_r = -\frac{1}{2} \log \Big(  - 2\lambda + 1 \Big) + \frac{1}{2}  \Big(  - 2\lambda + 1 \Big)^{-1} \Big( \ba \sum_{j = 1}^r z_j + \mu x^0 \Big)^2.
		\]
		If we define
		\[
		D_\ell = \Big( - 2\lambda + 1 \Big) - \ba^2 \sum_{k = \ell}^{r-1} \zeta_k( Q_{k + 1} - Q_k )
		\]
		then we can explicitly compute the recursion \cite{TSPHERE} to discover that
		\[
		2\E_0 X_0 =  \sum_{k = 0}^{r - 1}\frac{1}{\zeta_k} \log\frac{D_{k + 1}}{D_{k}}- \log D_r  + \frac{1}{D_0} \E_{0} \Big(  \mu x^0 \Big)^2
		\]
		where we take $\lambda$ so that $D_0 > 0$.
	\end{rem}
 \fi
	
	The replica symmetric case happens if the maximizing sequences are attained at the point when $r = 1$ and $Q_1 = Q$ and $\zeta_0 \to 0$ and $\mu = \lambda = 0$. In this case, the replica symmetric functional is
	\begin{align*}
		\varphi_{RS}(Q) &=  - \frac{\ba^2 Q^2}{4} + \bb\frac{M^{2}}{2} + \bc \frac{ S^2 }{4} 
		+ \E_{z,x_0} \log \int \exp \bigg( \ba \sqrt{Q } zx   \bigg) \, d\pP_X(x)
	\end{align*}
	where $z \sim N(0,1)$, $x_0 \sim \pP_0$. The replica symmetric solution gives an upper bound of the free energy, i.e. for any $Q$
	\[
	\lim_{\epsilon \to 0} \limsup_{N \to \infty} F_{N}({\bar\beta}:  \Sigma_\epsilon(S,M)) \leq \phi_{RS}(Q).
	\]
	However, it is not expected that this bound is sharp.
	
	\begin{rem}
		In the notation of \cite{lenkarmt}, we have $\bar R = \bc$ and $Q + \Sigma = S$ . The quantity  $\bc = - \ba^2$ in the Bayes optimal case. It is also not expected that the replica free energy will be necessarily replica symmetric.
	\end{rem}

    \subsection{Examples}\label{sec:example}

    The quenched LDP in Theorem~\ref{quenchedldp} covers a wide range of previously studied spin glass models. We briefly mention some examples in this section.

    \subsubsection{Sherrington-Kirkpatrick model with soft spins \cite{PVS}} This corresponds to the case when $\bb = 0$ and $\bc = 0$. This case is considerably easier because the overlap $R_{10}$ does not play a role. However, the norms of soft spin configurations are not fixed.
    
    \subsubsection{Sherrington-Kirkpatrick model with ferromagnetic interaction \cite{chenferro}} This corresponds to the case when $\pP_0 = \delta_1$, $\pP_X = \frac{1}{2} \delta_1 + \frac{1}{2} \delta_{-1}$. Unlike the previous model, $R_{10}$ appears in this model is simpler because $x^0$ is non-random.

    This model corresponds to a mismatched inference problem where the data is generated from a spiked matrix model with a deterministic rank-$1$ spike, but the statistician has no information on the signal distribution, so he naively assumes a balanced Rademacher prior. 
    
    \subsubsection{Sherrington-Kirkpatrick model with Mattis interaction \cite{Camillimismatch}} This corresponds to the case when $\pP_X = \frac{1}{2} \delta_1 + \frac{1}{2} \delta_{-1}$. In this model, all configurations are on the unit sphere, so we may take $\bc = 0$ without loss of generality since overlaps $R_{11} = 1$ everywhere.

    This model corresponds to a mismatched inference problem where the data is generated from a spiked matrix model
    \[
    Y_{ij} = W_{ij} + \frac{1}{\sqrt{N}} x^0_i x^0_j 
    \]
    but the statistician has no information on the signal distribution, so he naively assumes a balanced Rademacher prior. 

    \subsubsection{Symmetric Rank~1 Matrix Estimation \cite{lelarge2017fundamental}} The Hamiltonian in this Bayes-optimal inference problem occurs when $\ba = \bb$ and $\bc = -\frac{1}{2} \bb$. The overlaps concentrate in this model. This will suggest that the minimizer of the rate function Corollary~\ref{coro2.10} will concentrate on the maximizer of the rank~1 replica symmetric formula \cite[Equation~(3)]{lelarge2017fundamental}.

    \subsubsection{Maximum-average Submatrix Problem \cite{erbasubmatrix} } The Hamiltonian for this model is the classical SK Hamiltonian defined on the uniform configruation space of Boolean spins $\bx \in \{0,1 \}^N$. Computing the large deviations for the magnetization $m = \frac{1}{N}\sum_{i = 1}^N \sigma_i$, is a direct consequence of Theorem~\ref{quenchedldp} when our signal is non-random and concentrated on $1$ $\pP_0 = \delta_1$, and we take $S = M = m$, becausse of the Boolean nature of the spins.
    
    \subsubsection{The BBP Transition \cite{BBP} }
    If our prior $\pP^{\otimes N}_X$ is rotationally invariant, then the ground state free energies of $F_N(\bar \ba)$ is of particular interest, because it relates to the BBP transition of random matrices. In particular, if we consider $\ba = L \ba'$, $\bb = L \bb'$, $\bc = L \bc'$ for some $L \gg 0$ then
    \begin{align*}
      \frac{1}{L} F_N(\bar \ba) &=   \frac{1}{LN} \E_Y  \Big(\log \int e^{L H_N^{{\bar\ba}}(\bx)} \, d \pP_X^{\otimes N}(\bx)  \Big)
      \\&\simeq \frac{1}{2} \mathrm{esssup}_{\bx\in \supp(\pP^{\otimes N}_X)} \bigg[ \langle \bx, (\ba' W + \bb' \bx_0 \bx_0^\intercal) \bx \rangle + \frac{\bc'}{2} \|x\|_2^4 \bigg]
    \end{align*}
    which equals the top eigenvalue of the matrix $\ba' W + \bb' \bx_0 \bx_0^\intercal$ if  $\pP^{\otimes N}_X$ is rotationally invariant. 
    
    %\subsubsection{Semidefinite Relaxations \cite{javanmard2016phase}} 

    \begin{rem}
    The result in Theorem~\ref{quenchedldp} also holds if $H_N^{{\bar\ba}}(\bx)$ includes an external field term $\sum_{i \leq N} h_i x_i$. This term decouples during the cavity computations, so it does not introduce an additional technical challenge. 
    \end{rem}
    
    \subsubsection{Outline of the Paper} This article is organized as follows. 
    
    We first prove the universality stated in  Proposition \ref{prop:universality1} in Section \ref{secuniversality}. This observation connects the free energy of all rank $1$ inference problems to the general Hamiltonian $H_N^{\bar \beta}(\bx)$ defined in \eqref{eq:nonbayesoptimalHamiltonian1}. This follows from classical results for the universality of spin glasses and random matrix theory, to approximate general likelihood functions with its second order Taylor expansion.

 We then prove the large deviation upper bound in average as stated in Theorem \ref{technicalldp} in Section \ref{sec:upbd}, based on the usual tilting argument in the proof of Cramer's theorem and the famous interpolation trick introduced by Guerra \cite{guerra2003broken}, see \cite[Chapter~3]{PBook}. 
 
The proof of the complementary lower bound of Theorem \ref{technicalldp} takes the entirety of Section~\ref{sec:ldlb}. This bound is proved using the cavity approach and a regularizing perturbation of the Gibbs measure. However, unlike classical spin glass models the argument is much more delicate in this setting. The main reason is that the constraint on the overlap $\1(\bR_{10} \approx M)$ depends on a external source of randomness $\bx^0$ and the indicator function is not smooth with respect to this random variable. We dealt with this challenge in Section~\ref{sec:indicator}  by localizing the free energy around the empirical measure of $\bx^0$ and smoothing out the indicator on sets with finite entropy. This restriction to sets with finite entropy and a smooth approximation of the indicator is critical because small deviations of $\bx^0$ will lead to very large deviations of the free energy otherwise, so concentration of the (non-localized) free energy will not be possible. The Ghirlanda--Guerra identities (Section~\ref{sec:GGI}) and the cavity method (Section~\ref{sec:cavityI}) are adapted to these localized free energies.

The final step to proving our lower bound is asymptotically sharp requires an additional argument when our overlaps are restricted to values on the boundary of the  $\mathcal{C}$, which encodes the set of values with finite entropy.  We only have a large deviations lower bound on the so called set of exposed points, which does not include these boundary terms apriori. We adapt the proof of the Gartner--Ellis Theorem and use the large deviations bound for tilted measures combined with Rockafellar's Theorem to extend the lower bound to all points. This is explained in the proof of Lemma~\ref{lem:sharpupbd}. This final large deviations result combined with the cavity computations finishes the proof of the lower bound in Section~\ref{sec:cavityII}.

 Lastly, we remove this localization and use the fact that all of our estimates are uniform over $\bx^0$ to deduce the almost sure LDP in Theorem \ref{quenchedldp} in Section~\ref{proofthmquenched}. This part of the proof relies heavily on the fact that the signal is a product measure, so the empirical law of $\bx^0$ converges almost surely to $\pP_0$, so the localized free energy is a good approximation of the total free energy in the limit.

	\section{Universality}\label{secuniversality}

	Just like in the Bayes optimal case \cite[Section~3]{AJFL_inhomo}, we will show that these models will reduce to a Gaussian estimation problem under some mild conditions on $g$. Consider the Hamiltonian
	\begin{align}\label{eq:nonbayesoptimalHamiltonian}
		H_N^{{\bar\beta}}(\bx) &= \sum_{i<j}\left( \frac{\ba W_{ij}}{ \sqrt{  N}} x_ix_j + \frac{\bb}{ N} (x_ix_j)(x_i^{0} x_j^{0}) +  \frac{\bc}{2N}   (x_i x_j )^2\right)
	\end{align}
	where $W_{ij}$ are iid standard Gaussians and the covariance parameters are given in \eqref{eq:delta1}, \eqref{eq:delta2} and \eqref{eq:delta3}.
	
	Given any subset $A = A(\bx^0)$ that may depend on $\bx^0$, we define the restricted  free energy by
	\[
	F_N({\bar\beta}: A) = \frac{1}{N} \bigg( \E_Y  \Big(\log \int  \1(\bx\in A) e^{H_N^{{\bar\beta}}(\bx)} \, d \pP_X^{\otimes N}(\bx) \Big) \bigg). 
	\]
	We first prove a universality result that will imply that the free energy of general inference models are equivalent to the free energy of the Gaussian estimation problem.

	\begin{prop}[Universality] \label{prop:universality}
		If Hypothesis~\ref{hypcompact} and Hypothesis~\ref{hypderiv} holds then the free energy satisfies
		\[
		\big| F_N(g) - F_N({\bar\beta}) \big| =  O( N^{-1/2} ).
		\]
		More generally, for any sequence of measurable sets   $A = A_N(\bx^0)\subset \R^{N}$ such that  for $N$ large enough $F_N(0:A)>-\infty$, we have
		\[
		\big| F_N(g: A) - F_N({\bar\beta}: A) \big| =  O( N^{-1/2} ).
		\]
	\end{prop}
	%Repeated from earlier on
	%\begin{rem}
	%	In the Bayes optimal case,
	%	\[
	%	\ba^2 = \bb \quad \text{and}\quad \bc = -\beta^{2}
	%	\]
	%\end{rem}
	Notice that we may take $A = \R^{N}$, so it suffices to prove the universality of $F_N(g: A)$. 
	We begin by showing only the second order Taylor expansion of $g$ matters in the computation of the free energy. 
	
	\begin{lem}[Independence of Third Order Expansions] \label{lem:univ1}
		If $\|\partial_w^3 g\|_\infty < \infty$, then for any sequence of measurable sets   $A = A_N(\bx^0)\subset \R^{N}$ such that $F_N(0:A)>-\infty$ for $N$ large enough, we have
		
		\[
		F_N(g:A) = F_N(\tilde g:A) + O\Big( \frac{1}{\sqrt{N}}\Big) 
		\]
		where
		\[
		\tilde g(Y,w) = g( Y,0) + \partial_w g( Y,0)  w + \frac{1}{2} \partial_w^{(2)}g( Y,  0) w^2.
		\]
	\end{lem}
	
	\begin{proof}
		By Taylor's theorem, for all $i,j$,
		$$(g(  Y_{ij},w_{ij})-g( Y_{ij},0))=  \partial_w g( Y_{ij},0) w_{ij} + \frac{1}{2} \partial_w^{(2)}g(  Y_{ij},0)  w_{ij}^2  +\frac{w_{ij}^3}{3!} \partial_w^{(3)}g(Y_{ij},\theta_{ij}w_{ij})$$
		for some $\theta_{ij}\in [0,1]$. Since our hypothesis implies that $|w_{ij}|_\infty \le C^2/\sqrt{N}$, our assumption that $\|\partial_w^3 g\|_\infty < \infty$ implies that uniformly
		\[
		\frac{1}{N} \sum_{i < j} \frac{w_{ij}^3}{3!}\partial_{w}g (Y_{ij},\theta_{ij}w_{ij} ) =  O\Big( \frac{1}{N^{1/2}}\Big) .
		\]
		The claim follows. 
	\end{proof}
	
	The next step in the reduction is to prove that the coefficient of the second derivative term can be replaced by its conditional average.
	
	\begin{lem}[Concentration of Second Order Terms] \label{lem:univ2}
		Assume the $Y_{ij}$ are independent, $\sup_{i,j} \|\partial_w^{(2)}g(\cdot ,0)\|_\infty < \infty$ and $\pP_X$ is compactly supported.  Then,   for any sequence of measurable sets   $A = A_N(\bx^0)\subset \R^{N}$ such that $F_{N}(0:A)>-\infty$ for $N$ large enough, we have
		$$F_N(\tilde g:A)=  F_N(\bar g:A) +O\Big(\frac{1}{\sqrt{N}}\Big)$$
		with \[
		\bar  
		g(Y,w) = g( Y_{ij},0) + \partial_w g( Y,0)  w + \frac{1}{2} \E_Y[ \partial_w^{(2)}g(Y,0)\given B] w^2 \,
		\] 
		where $\E_Y[ \cdot \given B]$ is any conditional expectation so that $\E_Y ( \partial_w^{(2)} g(Y,0)- \E_Y[ \partial_w^{(2)} g(Y,0)\given B] )=0$.
	\end{lem}
	\begin{proof}
		Notice that since $F_{N}(0;A)$ is finite the following difference is well defined:
		$$F_N(\tilde g:A)-F_N(\bar g:A)=\E_Y\frac{1}{N}\ln \Big\langle e^{\frac{1}{2\sqrt N}\sum_{i< j} \frac{1}{\sqrt{N}}(\partial_w^{(2)}g(Y_{ij},0)-\E_Y[ \partial_w^{(2)}g(Y_{ij},0)|B])(x_i x_j)^2)} \Big\rangle$$
		where $$\langle f\rangle= \frac{\int  \1(\bx\in A) f (\bx) e^{\sum_{i< j} \bar g(Y_{i,j},w_{ij})} d\pP_X^{\otimes N}(\bx)}{ \int  \1(\bx\in A) e^{ \sum_{i< j} \bar g(Y_{i,j},w_{ij})} d\pP_X^{\otimes N}(\bx)}\,.$$
		Let $Z$ be the $N\times N$ symmetric  matrix with entries $\frac{1}{2\sqrt{N}}(\partial_w^{(2)}g(Y_{ij},0)-\E_Y[ \partial_w^{(2)}g(Y_{ij},0)|B])$ so that 
		$$\sum_{i< j} \frac{1}{2\sqrt{N}}(\partial_{w}^{(2)}g (Y_{ij},0)-\E_Y[ \partial_w^{(2)}g(Y_{ij},0)|B])(\bx_i^{\mathrm T} \bx_j)^2=\Tr\left (Z (\bx^{\mathrm{T}} \bx)^2\right).$$
		$Z$ is a random  matrix under $P_B$, it has centered independent entries with covariance bounded by $C/N$ and $(\bx^{\mathrm{T}} \bx)^2$ is the matrix with entries $(x_i x_j)^2$.
		Because these entries are bounded, we can use concentration inequalities  (\cite[ Theorem 2.3.5]{AGZ}  or \cite{GZ00})  and \cite[Lemma 5.6]{HuGu} to see that there exists some finite $L_0$ such that
		\begin{equation}\label{conc2}\pP_B\left(\|Z\|_\infty\ge L\right)\le e^{-N(L-L_0)}\,.\end{equation}
		On $\{\|Z\|_\infty\le L\}$,
		$$\left|\Tr\bigl(Z (\bx^{\mathrm{T}} \bx)^2\big)\right|=\left| \sum_{i,j} Z_{ij} x_i^2 x_j^2\right| \le L \sum_{i=1}^N x_i^4\le C L N $$
		for some finite constant $C$ depending only the bound on the support of $\pP_X$.
		Hence 
		$$F_N(\tilde g:A)-F_N(\bar g:A)=\E_Y 1_{\|Z\||\ge L} 
		\frac{1}{N}\ln \bigg\langle e^{\frac{1}{\sqrt N}\sum_{i\le j} \frac{1}{2\sqrt{N}}(\partial_w^{(2)}g (Y_{ij},0)-\E_Y[ \partial_w^{(2)}g(Y_{ij},0)|B])(x_i x_j)^2)}\bigg\rangle + O\Big( \frac{1}{\sqrt{N}}\Big).$$
		Moreover as $\partial_w^{(2)}g(Y,0)$ is assumed uniformly bounded over $i,j$, the term in the above expectation is uniformly bounded and therefore the first term is going to zero exponentially fast by \eqref{conc2}. 
	\end{proof}
	
	Later on we take $B=\sigma\{(x_i^0)_{i \leq N}\}$ and use that conditionally on $\bx^0 = (x_i^0)_{i \leq N}$ the $Y_{ij}$ are independent. We finally compare our free energy to those of a spin glass model. It will depend on three matrices:
	\[ \gamma_{ij}= \E_Y[ \partial_w^{(2)} g(Y_{ij},0)\given \bx^0], \quad
	\mu_{ij} =\E_Y [\partial_w g(Y_{ij}, 0) \given \bx^0] ,\quad \sigma^2_{ij} = \E_Y[ (\partial_w g(Y_{ij}, 0)  -\mu_{ij})^2 \given \bx^0].
	\]
	By universality, we will prove that we can replace $\partial_w g( Y_{ij},0)$ by $\sigma_{ij} W_{ij} + \mu_{ij}$ where $W_{ij}$ are iid standard Gaussian variables (under the assumption that $\sqrt{N} \mu_{ij} = O(1)$). 
	\begin{lem}[Universality in Disorder] \label{lem:univ3} 	Assume that 
		\[
		\sup_{i,j} \|\mu_{ij} \|_\infty= O(N^{-1/2}),~\sup_{ij} \|\sigma_{ij}^2 \|_\infty< \infty,~\sup_{i,j}  \bigg\| \frac{ \E_{Y}[ |\partial_w g( Y_{ij},0) - \mu_{ij} |^3|\bx^0] }{ \sigma_{ij}^3 } \bigg\|_\infty < \infty
		\]
		%\begin{itemize}
		%	\item $\sup_{i,j} \|\mu_{ij} \|_\infty= O(N^{-1/2})$
		%	\item $\sup_{ij} \|\sigma_{ij}^2 \|_\infty< \infty$
		%	\item  $\sup_{i,j}  \bigg\| \frac{ \E_{Y}[ |\partial_w g( Y_{ij},0) - \mu_{ij} |^3|\bx^0] }{ \sigma_{ij}^3 } \bigg\|_\infty < \infty$
		%\end{itemize}
		then for any sequence of measurable sets   $A = A_N(\bx^0)\subset \R^{N}$ such that $F_N(0:A)>-\infty$ for $N$ large enough, 
		
		$$F_N(\bar g:A)=F_N(\sigma,\mu,\gamma:A)+O\Big( \frac{ 1}{N^{1/2}} \Big) 
		$$
		where
		$$F_N(\sigma,\mu,\gamma:A)=\E_{W,x^0}\Big[\frac{1}{N}\ln\E_{x}[ 1_{\bx\in A}\exp( H_N(\bx))]\Big]$$
		with 
		\begin{align*}
			H_N(\bx) &=  \sqrt{\frac{1}{N}} \sum_{i<j}\left( \sigma_{ij} W_{ij} x_i x_j + \mu_{ij} x_i x_j \right) + \frac{1}{2N} \sum_{i<j} \gamma_{ij} x^2_i  x^2_j .
		\end{align*}
	\end{lem}
	
	The proof follows from an approximate integration by parts lemma \cite[Lemma 3.7]{PBook} 
	\begin{lem}\label{lem:approxIBP}
		Suppose $x$ is a random variable that satisfies $\E x = 0$, $\E |x|^3 < \infty$. If $f: \R \to \R$ is twice continuously differentiable and $\|f''\|_\infty < \infty$, then
		\[
		| \E x f(x) - \E x^2 \E f'(x)| \leq \frac{3}{2} \| f''\|_\infty \E |x|^3.
		\]
	\end{lem}
	
	\begin{proof}[Proof of Lemma \ref{lem:univ3}]
		We follow the proof of Carmona--Hu \cite{UnivCARMONAHU} presented in \cite[Theorem 3.9]{PBook}. To compare the free energies $F_N(\bar g)$ and $F_N(\sigma, \mu, \kappa)$ we use an interpolation argument. Conditionally on $\bx^0$, consider the interpolating Hamiltonian
		\begin{align*}
			H_{N}(\bx,t) &= \frac{1}{\sqrt{N} } \sum_{i < j} \Big( \sqrt{t} ( \partial_{w}g(Y_{ij},0) - \mu_{ij}) + \sqrt{1 - t} \sigma_{ij} W_{ij} \Big) x_i x_j + \frac{1}{\sqrt{N}} \sum_{i < j} \mu_{ij}  x_i x_j  + \frac{1}{2N} \sum_{i < j} \gamma_{ij}  (x_i x_j)^2
			\\&= \frac{1}{\sqrt{N} } \sum_{i < j} \sigma_{ij} \Big( \sqrt{t} \tilde W_{ij}+ \sqrt{1 - t} W_{ij} \Big)  x_i x_j + \frac{1}{\sqrt{N}} \sum_{i < j} \mu_{ij}  x_i x_j  + \frac{1}{2N} \sum_{i < j} \gamma_{ij}  (x_i x_j)^2
		\end{align*}
		where we defined $\tilde W_{ij} = \sigma_{ij}^{-1} ( \partial_w g( Y_{ij},0) - \mu_{ij})$ to simplify notation. Notice that 
		\[
		\E_{Y}[ \tilde W_{ij}^2\given \bx^0] = \sigma_{ij}^{-2} \E_{Y}[ ( \partial_w g( Y_{ij}, 0) - \mu_{ij})^2|\bx^0] = 1
		\]
		and
		\[
		\E_{Y} [\tilde W_{ij} \given \bx^0] = \sigma_{ij}^{-1} \E_{Y} [ ( \partial_wg(Y_{ij}, 0) - \mu_{ij} )  \given \bx^0] = 0
		\]
		so both $W$ and $\tilde W$ have mean zero and variance 1. We define the interpolating Hamiltonian,
		\[
		\phi(t) = \frac{1}{N} \E_{Y}[ \ln \E_{X} \1(\bx\in A) \exp(  H_N(\bx,t) )|\bx^0], \qquad \langle f\rangle_t=\frac{\int  \1(\bx\in A) f(\bx)  \exp(  H_N(\bx,t) )d\pP^{\otimes N}_X(\bx)}{\int \1(\bx\in A)  \exp(  H_N(\bx,t) )d\pP^{\otimes N}_X(\bx)}
		\]
		and notice
		\begin{equation}\label{eq:univdisorder1}
			\phi'(t) =\E_{W,Y}\left[ \frac{1}{2 \sqrt{t} N^{3/2}} \sum_{i < j} \sigma_{ij}\tilde W_{ij} \langle x_i x_j  \rangle_t - \frac{1}{2 \sqrt{1-t} N^{3/2}} \sum_{i < j} \sigma_{ij} W_{ij} \langle x_i x_j  \rangle_t \given[\bigg]\bx^0\right]
		\end{equation}
		Let $f(\tilde W_{ij}) = \langle x_i x_j  \rangle_t$ 
		(the dependence on $\tilde W$ is in the numerator and denominator in the Gibbs measure). We find that
		\[
		\frac{\partial f}{ \partial \tilde W_{ij} } = \frac{ \sqrt{t} \sigma_{ij}}{ \sqrt{N} } \bigg( \langle (x^{1}_i  x^1_j)^2 \rangle_t - \langle ( x^{1}_i  x^1_j )(x^{2}_i  x^2_j) \rangle_t \bigg)
		\]
		and
		\[
		\frac{\partial^2 f}{ \partial \tilde W^2_{ij} } =  \frac{t \sigma^2_{ij}}{N}   \bigg(\langle (x^{1}_i x^1_j)^3 \rangle_t -2 \langle (x^{1}_i x^1_j )^2(x^{2}_i  x^2_j) \rangle_t - \langle (x^{1}_i  x^1_j)(x^{2}_i  x^2_j)^2 \rangle_t + 2\langle( x^{1}_i  x^1_j )(x^{2}_i x^2_j)(x^{3}_i x^3_j ) \rangle_t \bigg)
		\]
		so the second derivative is bounded by 
		\[
		\frac{\sup_{ij} \|\E_{Y}[ |\partial_w g(Y_{ij},0) - \mu_{ij} |^2 |\bx^0]\|_\infty 6 C^6}{N}
		\]
		where $C$ is such that   $x\in [-C,C]$ almost surely. Applying the approximate integration by parts lemma to $\tilde W$ stated in Lemma~\ref{lem:approxIBP} applied conditionally on $\bx^0$ implies
		\begin{align}
			&\bigg| \E_Y\Big[\frac{1}{2 \sqrt{t} N^{3/2}} \sigma_{ij}\tilde W_{ij} \langle x_i x_j  \rangle_t\given[\Big] \bx^0 \Big] - \frac{\sigma_{ij}^2}{2 N^2} \bigg( \E_Y[ \langle (x^{1}_i  x^1_j)^2 \rangle_t\given \bx^0] - \E_Y [\langle( x^{1}_i  x^1_j )(\bx^{2}_i \bx^2_j) \rangle_t\given\bx^0] \bigg) \bigg| \notag
			\\&\le  \frac{\sup_{ij} \|\E_{Y}[ |\partial_w g(Y_{ij},0) - \mu_{ij} |^2 |\bx^0]\|_\infty 6 C^6 \kappa^3}{N} \cdot \frac{3 \sup_{i,j} \E [ |\tilde W_{i,j}|^3 \given \bx^0 ] }{4N^{3/2}} = O\Big( \frac{ 1}{N^{5/2}} \Big) \label{eq:disorderintbyparts1}
		\end{align}
		by our assumption on the uniform bounds on the conditional expectation of $\tilde W$. The classical integration by parts lemma for Gaussians implies
		\begin{equation}\label{eq:disorderintbyparts2}
			\E_W \bigg[\frac{1}{2 \sqrt{1-t} N^{3/2}} \sigma_{ij} W_{ij} \langle x_i x_j  \rangle_t \bigg]=\frac{ \sigma_{ij}^2}{2 N^2}\E_W \bigg[ \bigg(\langle (\bx^{1}_i \cdot \bx^1_j)^2 \rangle_t - \langle( x^{1}_i \cdot x^1_j )(x^{2}_i  x^2_j) \rangle_t \bigg) \bigg].
		\end{equation}
		
		Summing over $i < j$ in \eqref{eq:univdisorder1} using the computations \eqref{eq:disorderintbyparts1} and \eqref{eq:disorderintbyparts2} gives us the bound
		\[
		|\phi'(t)| \leq O\Big( \frac{ 1}{N^{1/2}} \Big)
		\]
		so that
		\[
		|\phi(1) - \phi(0)| = |  F_N(\bar g :A)-F_N(\sigma,\mu,\kappa:A) | \leq O\Big( \frac{ 1}{N^{1/2}} \Big) .
		\]
	\end{proof}
	
	We now simplify the coefficients $\sigma_{ij}, \mu_{ij}$ and $\gamma_{ij}$ in terms of the constants $\ba, \bb$, and $\bc$ defined in \eqref{eq:delta1}, \eqref{eq:delta2}, \eqref{eq:delta3}. We denote
	\[
	\E_{P_\out(Y\given w)} f(Y) = \int f(Y) d \pP_\out(Y \given w) =  \int f(Y)  e^{g^{0}(Y,w) }dY, \quad  \E_{P_\out(Y\given 0)} f(Y)  =  \int f(Y)  e^{g^{0}(Y,0) }dY\,.
	\] 
	By Taylor's theorem, we see that with $w$ bounded uniformly by  $L^{2}/\sqrt N$
	$$\ln \frac{d P_\out(Y\given w)}{dY}= g^{0}(Y,w)= g^{0}(Y,0)+\partial_{w}g^{0}(Y,0)w+\frac{1}{2}\partial^{(2)}_{w} g^{0}(Y,0)w^{2}+O(\frac{1}{N^{3/2}})\,.$$
	\begin{enumerate}
		\medskip
		\item {\bf The mean is of order $\frac{1}{\sqrt{N}}$  under Hypothesis \ref{hypderiv}:} With $w^{0}_{ij}= x_{i}x_{j}/\sqrt{N}$, and recalling the fact that $\|\partial_w g^0 \|_\infty, \partial_w g^0 \|_\infty < \infty$ by Hypothesis~\ref{hypg}  we have
		\begin{align*}
			\mu_{ij} &= \E_{Y}[ \partial_w g(Y_{ij},0) |\bx^0 ] = \int \partial_w g(y,0) e^{\ln P_\out(y\given w^0_{ij}) } \, dy 
			\\&= \int \partial_w g(y,0) \Big(1 + \partial_w g^0(y,0) w^0_{ij}  + \big( (\partial_w g^0(y,0))^2 + \partial_w^{(2)}g^0(y,0) \big) \frac{(w^0_{ij})^2}{2} + O(N^{-1})\Big) e^{g^0(Y,0)}\, dy
			\\&= \E_{P_\out(Y\given  0)} \partial_w g(Y,0) + \frac{x_i^0 x_j^0}{\sqrt{N}} \E_{P_\out(Y\given  0)} \partial_w g(Y,0)g^0_w(Y,0) + O(N^{-1}) 
			\\&= \E_{P_\out(Y\given  0)} \partial_w g(Y,0) + \frac{x_i^0 x_j^0}{\sqrt{N}} \bb + O(N^{-1})
		\end{align*}
		where the first term vanishes under Hypothesis \ref{hypderiv}.
		\medskip
		
		\item {\bf The variance is of order $1$:}
		\begin{align*}
			\sigma^2_{ij} &= \E_{Y} [(\partial_w g(Y,0))^2 - \mu_{ij}^2 \given \bx^0]
			\\&= \int (\partial_w g(Y,0))^2 e^{g^0(Y,0)} (1 + O(N^{-1/2}) \, dy - \mu_{ij}^2
			\\&= \E_{P_\out(Y \given 0)} (\partial_w g(Y,0))^2 - (\E_{P_\out(Y\given 0)} \partial_w g(Y_{ij},0) )^2 + O(N^{-1/2})
			\\&=  \ba^2 - (\E_{P_\out(Y\given 0)} \partial_w g(Y,0) )^2 + O(N^{-1/2}) )
		\end{align*}
		where the second term vanishes under Hypothesis \ref{hypderiv}.
		\medskip
		
		\item {\bf The $\gamma$ coefficient is of order $1$:} 
		\begin{align*}
			\gamma_{ij} &= \E_{Y}[ \partial_w^{(2)}g(Y,0) \given \bx^0] 
			%\\&= \int \partial_w^{(2)}g(Y,0) e^{g^0(Y,0)}  + O(N^{-1/2}) \, dy
			\\&= \E_{P_\out(Y \given  0)} \partial_w^{(2)}g(Y,0)  + O(N^{-1/2})
			\\&= \bc  + O(N^{-1/2})
		\end{align*}
		by the definition of the coefficients \eqref{eq:delta1} and \eqref{eq:delta3}.
	\end{enumerate}
	When  the term $\E_{P_\out(Y \given 0)} \partial_w g(Y,0) = 0$ as assumed in Hypothesis \ref{hypderiv}, then we conclude that 
	\begin{equation}\label{eq:covariancessimplification}
		\mu_{ij} =  \frac{x_i^0 x_j^0}{\sqrt{N}} \bb  + O(N^{-1}), \quad \sigma_{ij} = \ba  + O(N^{-1/2}), \quad \gamma_{ij} =  \bc   + O(N^{-1/2}).
	\end{equation}
	
	With this in mind, an interpolation argument and Gaussian integration by parts will prove that the Hamiltonian associated with the free energy $F(\sigma,\mu,\gamma)$ given by
	\[
	H_{N,\sigma,\mu,\gamma}(\bx) = \sum_{i<j} \left(\frac{\sigma_{ij} W_{ij}}{\sqrt{N}} (x_i x_j) + \frac{\mu_{ij}}{\sqrt{N}} (x_i x_j) + \frac{1}{2N} \gamma_{ij} (x_i x_j )^2 \right)
	\]
	can be replaced with  
	\begin{align*}
		H_N^\beta(\bx) &= \sum_{i<j} \bigg( \ba \frac{W_{ij}}{ \sqrt{  N}} x_ix_j + \frac{\bb}{ N} (x_ix_j)(x_i^{0} x_j^{0}) +  \frac{1}{2N} \bc  (x_i x_j )^2 \bigg)
	\end{align*}
	defined in \eqref{eq:nonbayesoptimalHamiltonian} without changing the limit of the free energy. 
	\begin{lem}[Reduction to Low Rank Hamiltonian] \label{lem:univ4}
		If Hypothesis~\eqref{hypg},Hypothesis~\eqref{hypcompact} and Hypothesis~\eqref{hypderiv} hold with $\bar\beta=(\beta,\bb,\bc)$, then  for any sequence of measurable sets   $A = A_N(\bx^0)\subset \R^{N}$ such that $F_N(0:A)>-\infty$ for $N$ large enough, 
		
		$$F_N(\sigma,\mu,\kappa:A)=F_N(\bar \beta:A) +O( N^{-1/2})\,.
		$$
		
	\end{lem}
	
	\begin{proof}
		Consider the interpolating Hamiltonian,
		\begin{align*}
			H_N(t,\bx) &= \sum_{i<j} \frac{\sqrt{t} \sigma_{ij} W_{ij}}{\sqrt{N}} (x_i  x_j) + \frac{t\mu_{ij}}{\sqrt{N}} (x_i  x_j) + \frac{t}{2N} \gamma_{ij} ( x_i  x_j )^2
			\\&\quad + \sum_{i<j} \frac{\sqrt{1-t} \ba \tilde W_{ij}}{ \sqrt{ N}} (x_i x_j) + \frac{t\bb }{ N}(x_i^0 x_j^0) (x_i x_j) + \frac{t}{2N} \bc  (x_i x_j )^2.
		\end{align*}
		where $W$ and $\tilde W$ are independent standard Gaussians. If we define
		\[
		\phi(t) = \frac{1}{N} \E_{W,\tilde W, x^0} \ln \E_x  \1(\bx\in A) e^{H_N(t,\bx)}
		\]
		then 
		\begin{align*}
			\phi'(t) &= \frac{1}{N} \E \bigg\langle \partial_t H_N(t,\bx) \bigg\rangle_t
			\\&=\frac{1}{N}\E \bigg(\sum_{i<j} \frac{\sigma_{ij} W_{ij}}{2 \sqrt{t} \sqrt{N}} \langle x_i x_j \rangle_t  + \frac{\mu_{ij}}{\sqrt{N}} \langle x_i x_j \rangle_t + \frac{1}{2N} \gamma_{ij} \langle (x_i x_j)^2 \rangle_t \bigg)
			\\&\quad - \frac{1}{N} \E\bigg( \sum_{i<j} \frac{\ba \tilde W_{ij}}{2 \sqrt{1 - t} \sqrt{  N}} \langle x_i  x_j\rangle_{t} + \frac{\bb (x_i^0 x_j^0)}{ N} \langle x_i x_j \rangle_t + \frac{1}{2N} \bc \langle (x_i x_j )^2 \rangle_t \bigg)
		\end{align*}
		where $\langle \cdot \rangle_t$ is the average with respect to the Gibbs measure $G_t \propto  \1(\bx\in A) e^{H_N(t,\bx)}$. Recall that \eqref{eq:covariancessimplification}.	This implies that clearly the $\mu_{ij}$ and $\gamma_{ij}$ cancel with the non-Gaussian terms in the summation up to some $O( N^{-1/2})$ error. If we integrate by parts, then
		\[
		\frac{1}{N}\E \sum_{i<j} \frac{\sigma_{ij} W_{ij}}{2 \sqrt{t} \sqrt{N}} \langle x_i x_j \rangle_t = \frac{1}{2N^2} \sum_{i < j} \sigma_{ij}^2\big( \langle (x_i^1 x_j^1)^2 \rangle_t - \langle (x_i^1 x_j^1)(x_i^2  x_j^2) \rangle_t\big)
		\]
		and
		\[
		\frac{1}{N}\E \sum_{i<j} \frac{\ba \tilde W_{ij}}{2 \sqrt{1 - t} \sqrt{ N}} \langle x_i x_j \rangle_t = \frac{\ba^2}{2 N^2} \sum_{i < j} \big( \langle (x_i^1 x_j^1)^2 \rangle_t - \langle (x_i^1 x_j^1)(x_i^2 x_j^2) \rangle_t\big)
		\]
		so the difference of the Gaussian terms are also $O( N^{-1/2})$. Therefore,
		\[
		|\phi'(t)| = O( N^{-1/2}),
		\]
		which completes the proof. 
	\end{proof}
	
	\begin{rem}
		Notice that $\mu_{ij} = O( \frac{1}{\sqrt{N}} )$, so one of the hypothesis in the universality theorem is automatically satisfied if $\E_{P_\out(Y \given 0)} \partial_w g(Y,0) = 0$.
	\end{rem}
	
	We now have all the parts to conclude the universality result.
	
	\begin{proof}[Proof of Proposition~\ref{prop:universality}]
		Combine the results from Lemma~\ref{lem:univ1}, Lemma~\ref{lem:univ2}, Lemma~\ref{lem:univ3}, and Lemma~\ref{lem:univ4}.
	\end{proof}

	\section{Weak Large deviation Upper Bound}\label{sec:upbd}
	We will prove the weak large deviation upper bound of Theorem \ref{technicalldp}.
	In this section we consider the case where the measurable set $A$ is equal to the open ball where the overlaps $R_{1,1}=\frac{1}{N}\sum_{i=1}^{N} x_{i}^{2}, R_{1,0}=\frac{1}{N}\sum_{i=1}^{N} x_{i}x_{i}^{0}$ are close to some given values. Recall that we denoted this ball: 
	\begin{equation}\label{eq:integrationregion2}
		\Sigma_\epsilon(S,M) = \{ \bx \in \R^N \mmm |R_{1,1} - S| < \epsilon, |R_{1,0} - M| < \epsilon \}, 
	\end{equation}
	and let $F_{N}(\bar\beta:  \Sigma_\epsilon(S,M))$ be the free energy constrained to this ball:
	$$F_{N}(\bar\beta:  \Sigma_\epsilon(S,M))=
	\frac{1}{N} \bigg( \E_Y  \Big(\log \int \1(\bx\in \Sigma_\epsilon(S,M)) e^{H_N^{\bar\ba}(\bx)} \, d \pP_X^{\otimes N}(\bx)   \Big) \bigg),.$$

	%{\color{red} We can add the $\delta$ ball here to the upper bound with no difficulty. All the integration by parts is done with respect to the Gaussians, so it won't change the proof and the sum of the $X_0$}
	
	We begin by proving that the Parisi functional is an upperbound of the constrained free energy, namely the upper bound in Theorem \ref{technicalldp}.
	\begin{prop}[Large Deviation Upper Bound of the Free Energy] \label{prop:upbd}
		There exists a universal  finite constant $L$ such that for every	 $S,M\in \cC$, and every  real numbers $\mu,\lambda$, we have 
		\[
		F_{N}(\bar\beta:  \Sigma_\epsilon(S,M)) \leq 
		\bigg(- \mu S - \lambda M + \E_{0} X_{0}(\lambda,\mu,Q,\zeta)[x^{0}] - \frac{\ba^2}{4} \sum_{k = 0}^{r - 1} \zeta_k ( Q^2_{k + 1} - Q^2_k ) \bigg) + \frac{\bb}{2} M^{2} + \frac{\bc}{4} S^{2} \]
		\[ +   L\varepsilon(|\mu|+|\lambda|)+o_{N,\epsilon}(1) \,.
		\]
		where $X_0$ was defined in \eqref{recur} and $\E_0$ is the average with respect to $x^0 \sim \pP_0$. 	Moreover $o_{N,\epsilon}(1)=O(\varepsilon)+O(N^{-1})$ is independent of $\lambda,\mu$.
	\end{prop}
	\iffalse This is replica symmetric, but we haven't optimized over the parameters yet, so it is still a bit more complicated and unnecessary to include at the moment.
	\begin{rem}
		Our upper bound is a generalization of the replica symmetric functional derived via the replica trick in \cite[Appendix~C]{lenkarmt}. The replica symmetric upper bound happens when we evaluate the functional in Proposition~\ref{prop:upbd} at the points $r = 1$ and $\zeta_0 = 0$ giving
		\[
		\E \log \int_{\Sigma(S,M)} e^{ \ba \sqrt{Q} gy + \bb M x^0 y - \frac{\ba^2}{2} Q y^2  + \frac{\bc}{2} S y^2 } \, d \pP_X(y)   - \frac{M^2\bb}{2}  - \frac{S^2\bc }{4} .
		\]	
	\end{rem}
	\fi
	First notice that it is enough to consider the case where $\bb=\bc=0$ since 
	\begin{align}
		F_{N}(\bar\beta:  \Sigma_\epsilon(S,M))% &= \frac{1}{N} \E \log \int_{\Sigma_\epsilon(S,M)} e^{H_N(\bx)} \, d\pP_X^{\otimes N}(\bx) \notag\\
		&=\frac{1}{N} \E \log \int_{\Sigma_\epsilon(S,M)} e^{H^{SK}_N(\bx)} \, d\pP_X^{\otimes N}(\bx) + \bb \frac{M^2}{2} +\bc  \frac{S^2}{4}  + O(\epsilon) \label{eq:upbdsksplit}
	\end{align}
	where
	\begin{align*}
		H^{SK}_N(\bx) &= \ba \sum_{i<j}  \frac{W_{ij}}{ \sqrt{  N}} x_ix_j .
	\end{align*}
	We therefore  focus on proving an upper bound for the term
	\begin{equation}\label{eq:SKFEupbd}
		F_N^{SK} (\Sigma_{\epsilon}(S,M)): = \frac{1}{N} \E \log \int_{\Sigma_\epsilon(S,M)} e^{H^{SK}_N(\bx)} \, d\pP_X^{\otimes N}(\bx)
	\end{equation}
	because the other terms are constant. The goal of this section is to prove the following statement.
	\begin{prop}[Large Deviation Upper Bound of the SK Free Energy] \label{prop:upbd:SK}
		There exists a universal constant $L$ that is independent of $N$ such that for every	 $S,M\in \cC$, and every  real numbers $\mu,\lambda$, we have 
		\[
		F_N^{SK}(\Sigma_{\epsilon}(S,M)) \leq \bigg(- \mu S - \lambda M + \E_{0} X_{0}  - \frac{\ba^2}{4}  \sum_{k = 0}^{r - 1} \zeta_k ( Q^2_{k + 1} - Q^2_k ) \bigg)  + L \epsilon ( |\mu| + |\lambda| ) +o_{N,\epsilon}(1)
		\]
		where $X_0$ was defined in \eqref{recur} and $\E_0$ is the average with respect to $x^0 \sim \pP_0$.
	\end{prop}

	We now state the analogue of the replica symmetry breaking formula.
	Let $r > 1$ and consider parameters
	\begin{equation}\label{eq:zetaseq}
		\zeta_{-1} = 0 < \zeta_0< \dots < \zeta_{r-1}\le 1
	\end{equation}
	and sequence
	\begin{equation}\label{eq:Qseq}
		0 = Q_0 \leq Q_1 \leq \dots \leq Q_{r-1} \leq Q_r =  S.
	\end{equation}
	Let $v_\alpha$ be the weights of the Ruelle probability cascades \cite[Chapter 2]{PBook}, corresponding to \eqref{eq:zetaseq}. Recall that the Ruelle probability cascades is a random probability measure on $\N^r$, the leaves of the infinite rooted tree with depth $r$ encoded by the sequence of parameters $\zeta$. Every leaf of the tree $\alpha = (n_1, \dots, n_r) \in \N^r$ can be encoded by a path along the vertices,
	\[
	\alpha_{|1} = (n_1), ~\alpha_{|2} = (n_1,n_2),~ \dots,~ \alpha_{|r-1} = (n_1,n_2, \dots, n_{r-1}),~ \alpha = \alpha_{|r} =  (n_1, \dots, n_r)
	\]
	with the convention that $\alpha_{|0} = \emptyset$ is the root of the tree, and $k \leq r$ denotes the distance from the vertex $\alpha_{|k} \in \N^k$ to the root. Each vertex $\beta_{|k} = (n_1, \dots, n_{k-1}, n_k )$ of the tree will be associated with a random variable $u_{\beta_{|k}}$ defined as follows: Let $\beta_{|k - 1} = (n_1, \dots, n_{k-1})$ denote the parent of $\beta_{|k}$ and let
	\[
	u_{(\beta_{|k - 1} , 1)} > u_{(\beta_{|k - 1},2)} > \dots > u_{(\beta_{|k - 1},n_k)} > \dots.
	\]
	be the points from a Poisson process with mean measure $\zeta_{k-1} x^{-1 - \zeta_{k-1}}$ arranged in decreasing order, and define
	\[
	u_{\beta_{|k}} = u_{ (n_1, \dots, n_{k-1}, n_k ) } = u_{(\beta_{|k - 1},n_k)}.
	\]
	We further assume that these points are generated independently for different parent vertices. For each leaf $\alpha \in \N^r$, the weights of the Ruelle probability cascades $v_\alpha$ is the product of these points along the path from the root to the leaf:
	\[
	v_\alpha = \frac{u_{\alpha_{|1}} \cdots u_{\alpha_{|r}}}{ \sum_{\beta\in\mathbb N^{r}} u_{\beta_{|1}} \cdots u_{\beta_{|r}}  }.
	\]
	
	We consider the Gaussian processes $Z(\alpha)$ and $Y(\alpha)$ indexed by points on the infinite tree $\N^r$ with covariances
	\[
	\E Z(\alpha^1) Z(\alpha^2) = Q_{\alpha^1 \wedge \alpha^2} \quad \E Y(\alpha^1)  Y(\alpha^2)  = \frac{1}{2} Q^2_{\alpha^1 \wedge \alpha^2}.
	\]
	The notation $\alpha^1 \wedge \alpha^2$ denotes the least common ancestor of the paths leaves $\alpha^1$ and $\alpha^2$ of the infinite tree indexed by $\N^r$,
	\[
	\alpha^1 \wedge \alpha^2 = \min\Big\{0 \leq j \leq r \mmm \alpha_{|1}^1 = \alpha_{|1}^2, \dots, \alpha_{|j}^1 = \alpha_{|j}^2, \alpha_{|j + 1}^1 \neq \alpha_{|j +1}^2  \Big\}
	\]
	Notice that we are off by a factor $\frac{1}{2}$ in comparison to the usual SK models because we sum over $i < j$ in these problems. We let $Z_i(\alpha)$ be independent copies of $Z(\alpha)$ and we consider the interpolating Hamiltonian
	\begin{align*}
		H_N(t,\bx,\alpha) &=  \sum_{i<j}  \frac{\sqrt{t} \ba W_{ij}}{ \sqrt{  N}} x_ix_j  + \sum_{i \leq N}  \sqrt{(1-t)  }\ba Z_i(\alpha) x_i + \sqrt{t} \ba \sqrt{N} Y(\alpha) .
	\end{align*}
	We define the constrained interpolating free energy as
	\begin{align*}
		\phi_N^{S,M,\epsilon}(t) %&= \frac{1}{N} \E \log \sum_{\alpha} v_\alpha \int \1(R_{1,1} \approx S) \1(|R_{1,0}-M|\le\epsilon) e^{ H_N(t,\bx,\alpha) } \, d\pP_X^{\otimes N}(\bx)
		&:= \frac{1}{N} \E \log \sum_\alpha v_\alpha \int_{\Sigma_\epsilon(S,M)} e^{ H_N(t,\bx,\alpha) } \, d\pP_X^{\otimes N}(\bx) .
	\end{align*}
	A standard interpolation argument will give us an upper bound of the free energy.

	\begin{lem}[Guerra's Interpolation] \label{lem:upbdinterpolation}
		We have
		\[
		\phi_{N}^{S,M,\epsilon}(1)  \leq \phi_{N}^{S,M,\epsilon}(0) + O(\epsilon)+O(\frac{1}{N})
		\]
		where $O(\epsilon)$ is uniform in $N,S$ and $M$ and $O(\frac{1}{N})$ is uniform in $\epsilon,S$ and $M$. 
	\end{lem}
	\begin{proof}
		We denote in short $\phi$ for $\phi_N^{S,M,\epsilon}$ during the proof. 
		We have
		\begin{align*}
			\phi'(t) &= \frac{1}{N} \E \bigg\langle \frac{\partial H_N(t,\bx,\alpha)}{\partial t} \bigg\rangle_t
			\\&= \frac{1}{N} \E \bigg\langle \sum_{i<j} \ba \frac{ W_{ij}}{2 \sqrt{t} \sqrt{  N}} x_ix_j - \sum_{i \leq N} \frac{\ba}{2\sqrt{1 - t} } Z_i(\alpha) x_i + \frac{\ba\sqrt{N} }{2\sqrt{t}  } Y(\alpha)\bigg\rangle_t
		\end{align*}
		where $\langle \cdot \rangle_t$ is the average  under  $dG_N^t$  associated with the interpolating Hamiltonian : for a test function $f$
		\[
		\int f(\bx,\alpha) d G_N^t(\bx,\alpha) =\frac{\sum_{\alpha \in\N^{r}}\int_{\Sigma_\epsilon(S,M) }f(\bx,\alpha)  e^{H_N(t;\bx,\alpha) } v_\alpha d \pP(\bx) }{ \sum_\alpha \int_{\Sigma_\epsilon(S,M)} e^{H_N(t;\bx,\alpha) } v_\alpha d \pP(\bx) }.
		\]
		Integrating by parts  the Gaussian process $W,Y,Z$ (see \cite[Lemma~1.4]{PBook})  shows that $\phi'(t)$ equals 
		\begin{align*}
			& \frac{1}{N} \E \bigg\langle \sum_{i<j} \frac{\ba^2}{2  N} x^1_ix^1_j x^1_i x^1_j - \frac{\ba^2}{2  N} x^1_ix^1_j x^2_i x^2_j  \bigg\rangle_t
			- \frac{1}{N} \E \bigg\langle \sum_{i \leq N} \frac{\ba^2}{2 } Q_{\alpha^1 \wedge \alpha^1} x^1_i x^1_i - \frac{\ba^2}{2 } Q_{\alpha^1 \wedge \alpha^2} x^1_i x^2_i \bigg\rangle_t
			\\&+ \frac{1}{N} \E \bigg\langle \frac{N \ba^2 }{4 } Q_{\alpha^1 \wedge \alpha^2}^2 - \frac{N \ba^2}{4} Q_{\alpha^1 \wedge \alpha^1}^2  \bigg\rangle_t
			\\&= \E \bigg\langle \frac{\ba^2}{4 } R_{1,1}^2 - \frac{\ba^2}{4  } R_{1,2}^2 -  \frac{\ba^2}{2 } Q_{\alpha^1 \wedge \alpha^1} R_{1,1} + \frac{\ba^2}{2 } Q_{\alpha^1 \wedge \alpha^2} R_{1,2} + \frac{\ba^2 }{4 } Q_{\alpha^1 \wedge \alpha^1}^2 - \frac{\ba^2}{4} Q_{\alpha^1 \wedge \alpha^2}^2  \bigg\rangle_t + O\Big( \frac{1}{N} \Big) \,.
		\end{align*}
		The error $O(\frac{1}{N})$ comes from the diagonal terms and is uniform. 
		The self overlap terms from the integration by parts are cancelled off and the diagonals are of order $\frac{1}{N}$. We can simplify the upper bound further by completing the squares to conclude that
		\[
		\phi'(t) \leq - \frac{\ba^2}{4} \E \langle (R_{1,2} - Q_{\alpha^1 \wedge \alpha^2} )^2 \rangle_t 
+ \frac{\beta^2}{4} \E \langle (R_{1,1} - Q_{\alpha^1 \wedge \alpha^1} )^2 \rangle	+o_{N,\epsilon}(1) .	\]
The positive quadratic term is small because $\alpha^{1}\wedge\alpha^{1}=r$ and  $Q_{r}=S$ and $R_{11} \approx S$ on the set $\Sigma_{\epsilon}(S,M)$, so we can absorb it into the error term $o_{N,\epsilon}(1)=O( \frac{1}{N})+\beta^{2}C^{2}\epsilon $.
		
		We conclude that
		\[
		\phi'(t) \leq o_{\epsilon,N}(1)
		\]
		Integrating with respect to $t$ implies that $\phi(1) \leq \phi(0) + o_{\epsilon,N}(1)
		$.
	\end{proof}
	From Lemma~\ref{lem:upbdinterpolation}, we have shown that
	\begin{align*}
		F^{SK}_N(\Sigma_{\epsilon}(S,M))&+\frac{1}{N} \E \log \sum_\alpha v_\alpha  e^{ \ba \sqrt{N}  Y(\alpha)  } =\phi_{N}^{S,M,\epsilon}(1)\le \phi_{N}^{S,M,\epsilon}(0)  + o_{\epsilon,N}(1)
		\qquad\\
		&
		\le \frac{1}{N} \E \log \sum_\alpha v_\alpha \int_{\Sigma_\epsilon(S,M)} e^{\ba\sum_{i \leq N}  Z_i(\alpha) x_i } \, d \pP_X^{\otimes N} (\bx) + o_{\epsilon,N}(1)
	\end{align*}
	and therefore
	\begin{align}
		F^{SK}_N(\Sigma_{\epsilon}(S,M))& \le  \frac{1}{N} \E \log \sum_\alpha v_\alpha \int_{\Sigma_\epsilon(S,M)} e^{\ba\sum_{i \leq N}  Z_i(\alpha) x_i } \, d \pP_X^{\otimes N} (\bx)\nonumber\\
		&
		- \frac{1}{N} \E \log \sum_\alpha v_\alpha  e^{ \sqrt{N} \ba Y(\alpha)  } + o_{\epsilon,N}(1)
		\label{eq:upbdintermediate}
	\end{align}
	where the error terms are independent of our choice of $S$ and $M$. To write the upper bound in the form appearing in Proposition~\ref{prop:upbd} we have to compute the average of the terms that depend on $\alpha$. These averages with respect to the Ruelle probability cascades variable $\alpha$ can be computed using the following recursive formulation from \cite[Theorem 2.9]{PBook}.
	\begin{lem}[Averages with Respect to the Ruelle Probability Cascades ] \label{lem:RPCavg}
		Let $C: \R \to \R$ be an increasing non-negative function. Suppose that there exists a Gaussian process $g(\alpha)$ by $\alpha \in \N^r$ with covariance
		\[
		\E g(\alpha^1) g(\alpha^2) = C( Q_{\alpha^1 \wedge \alpha^2} )
		\]
		independent of $v_\alpha$. For a function $f: \R \to \R$ we define
		\[
		X_r = f\Big( \sum_{k = 1}^r ( C( Q_k ) - C( Q_{k - 1} )  )^{1/2} z_{ k}  \Big) \qquad X_{p} = \frac{1}{\zeta_p} \log \E_{z_{k + 1}} e^{\zeta_p X_{p + 1}} \quad \text{for $0 \leq p \leq r-1$}
		\]
		where $z_k$ are iid standard Gaussians. If $\E e^{\zeta_{r-1}  X_r } < \infty$ then
		\[
		\E \log \sum_{\alpha} v_\alpha e^{ f ( g(\alpha) ) } = X_0.
		\]
		The average on the outside is over the randomness in the Gaussian processes and the random measure $v_\alpha$.  
	\end{lem}
	
	\begin{proof}
		The proof can be found in \cite[Theorem 2.9]{PBook}. Essentially the special covariance structure of $g(\alpha)$ that depends only on the branching points of the rooted tree allows us to compute the expected values recursively from the leaves of $\N^r$ to its root. 
		
		Let's start with the case when $r = 1$ for simplicity. The result follows from the following invariance property of the Ruelle probability cascades: if $u_n$ are points from a Poisson process with mean measure $\mu(dx) = \zeta x^{-1-\zeta}$ arranged in decreasing order and if $X_n$ is another iid sequence of random variables independent of $u_n$, then the Poisson processes  $(u_n X_n)$ and $(\E X^\zeta)^{\frac{1}{\zeta}} u_n$ have the same mean measures. Taking the logarithms imply that
		\[
		\E \log \sum_{n \geq 1} u_n X_n = \E \log \sum_{n \geq 1} u_n + \frac{1}{\zeta} \log \E X^\zeta\implies \E \log \sum_{n \geq 1} v_n X_n = \frac{1}{\zeta} \log \E X^\zeta,
		\]
		provided that all the terms are well defined, which is explained in more detail in \cite[Lemma~2.2]{PBook}. 
		
		The case when $r > 1$ follows by induction and using the fact that the children at each level $k$ of the tree are generated indpendently from the mean measure $\mu(dx) = \zeta_k x^{-1-\zeta_k}$. The Gaussian process $g(\alpha)$ can also be defined as the sum of independent random variables along the vertices of the paths to the roots. Indeed, for every vertex $0 \leq p \leq r$ and vertex $(n_1,\dots,n_p) \in \N^p$, we can associate it with an independent standard Gaussian random variable $z_{(n_1,\dots, n_p)}$. For $\alpha \in \N^r$, we see that
		\[
		g(\alpha) \stackrel{d}{=} \sum_{k = 0}^{r-1} ( C(Q_{k + 1}) - C(Q_{k}) )^{\frac{1}{2}} z_{\alpha_{|k}}
		\]
		In particular, $g(\alpha)$ is independent of $v_\alpha$. Furthermore, if we denote $\mathscr{F}_k$ to be $\sigma$-algebra generated by the random variables on the vertices indexed by points in $\N^k, \N^{k - 1},\dots, \N$ then $z_{\alpha|k + 1}$ and $z_{\alpha|k + 1}$ is independent $\mathscr{F}_k$. The formula now follows from induction along the levels of the tree conditionally on $\mathscr{F}_k$. The details of this computation can be found in \cite[Theorem 2.9]{PBook}.
		
	\end{proof}
	
	We can now simplify the terms in \eqref{eq:upbdintermediate} to arrive at the upper bound stated in Propositions~\ref{prop:upbd:SK} and ~\ref{prop:upbd} .
	
	\begin{proof}[Proof of Proposition~\ref{prop:upbd:SK}]
		The second term in \eqref{eq:upbdintermediate} with the $Y(\alpha)$ Gaussian processes can be computed explicitly using Lemma~\ref{lem:RPCavg} applied to the process $Y(\alpha)$, $C(x)=x^2 $,
		and $f(y) = \sqrt{N} \ba y$
		\begin{equation}\label{eq:ycomputation}
			\frac{1}{N} \E \log \sum_\alpha v_\alpha  e^{ \sqrt{N}\ba Y(\alpha) } = \frac{\ba^2}{4 } \sum_{k = 0}^{r - 1} \zeta_k ( Q^2_{k + 1} - Q^2_k ) \,.
		\end{equation}
		Indeed, we then have $X_{r}= \frac{ \sqrt{N} \ba}{\sqrt{2}} \sum_{k=1}^{r} (Q_{k}^{2}-Q^{2}_{k-1})^{1/2} z_{k}$ and therefore
		$$X_{k}=\frac{1}{\zeta_{k}}\E_{z_{k+1}}[e^{\zeta_{k}X_{k+1}}]$$
		is such that
		$$X_{r-1}= \frac{ \sqrt{N}\ba}{\sqrt{ 2} }\sum_{k=1}^{r-1} (Q_{k}^{2}-Q^{2}_{k-1})^{1/2} z_{k}+\frac{N\ba^2}{4}(Q^{2}_{r}-Q^{2}_{r-1})\zeta_{r},\cdots, X_{0}=\frac{N}{4} \sum (Q_{k+1}^{2}-Q^{2}_{k})\zeta_{k+1}$$

		To explicitly compute the first term in \eqref{eq:upbdintermediate}, we will need to remove the constraint on the domain. We do this by introducing Lagrange mulitiplier terms to ensure that the upper bound is sharp after minimizing over these new parameters (see Lemma~\ref{lem:sharpupbd}). For parameters $\lambda$ and $\mu$, we have on $\Sigma_{\epsilon}(S,M)$ that
		\[
		\bigg| \frac{\lambda}{N} \sum_{i = 1}^N x_i x_i - \lambda S \bigg| \leq \epsilon | \lambda|
		\mbox{ 
			and }
		\bigg| \frac{\mu}{N} \sum_{i = 1}^N x_i x_i^0 - \mu M \bigg| \leq \epsilon | \mu|.
		\]
		By adding and subtracting $\frac{\lambda}{N} \sum_{i = 1}^N x_i x_i$ and $\frac{\mu}{N} \sum_{i = 1}^n x_i x_i^0$ from the exponents, we see that for any real numbers $\mu,\lambda$, 
		\begin{align*}
			\Lambda_{N}:=& \frac{1}{N} \E \log \sum_\alpha v_\alpha \int_{\Sigma_\epsilon(S,M)} e^{\sum_{i \leq N} \ba Z_i(\alpha) x_i
			} \, d \pP_X^{\otimes N} (\bx) 
			\\&\leq \epsilon ( |\mu| + |\lambda| ) - \mu S - \lambda M + \frac{1}{N} \E \log \sum_\alpha v_\alpha \int e^{\sum_{i \leq N}\left\lbrace \ba Z_i(\alpha) x_i + \lambda x_i^2 + \mu x_i x_i^0 \right\rbrace  } \, d \pP_X^{\otimes N} (\bx) 
		\end{align*}
		where the second integral is an unconstrained integral. This upper bound can be computed recursively using Lemma~\ref{lem:RPCavg} on $g(\alpha) = Z(\alpha)$ and $f(z) = \log \int e^{\sum_{i = 1}^N \ba z_i x_i  + \lambda x_i^2 + \mu x_i x_i^0  } \, d \pP^N_X (\bx) $ and independence of the random variables. If we define
		\[
		X_{r} = \log  \int e^{\sum_{i = 1}^N \ba \sum_{j = 1}^r z_{j,i} x_i + \lambda x_i^2 + \mu x_i x_i^0  } \, d \pP_X^N (\bx) = \sum_{i = 1}^N \log  \int e^{ \ba \sum_{j = 1}^r z_{j,i} x_i + \lambda x_i^2 + \mu x_i x_i^0  } \, d \pP_X (x_i)
		\]
		where $z_{j,i}$ are independent for $j,i$ and
		\[
		\Var(z_{j,i}) = Q_j - Q_{j-1}
		\]
		and define recursively for $0 \leq j \leq r-1$
		\[
		X_{j,i} = \frac{1}{\zeta_j} \log \E_{z_{j + 1,i}} e^{\zeta_j X_{j + 1,i}} \qquad X_{r,i} = \log  \int e^{ \ba \sum_{j = 1}^r z_{j,i} x_i + \lambda x_i^2 + \mu x_i x_i^0  } \, d \pP_X (x_i),
		\]
		and if $\E_{z_j}$ denotes the expected value with respect to $z_{j,1}, \dots, z_{j,N}$,
		\[
		X_j = \sum_{i = 1}^N X_{j,i} = \frac{1}{\zeta_j} \log \E_{z_{j + 1}} e^{\zeta_j \sum_{i = 1}^N X_{j + 1,i}}
		\]
		Lemma~\ref{lem:RPCavg} and \eqref{eq:ycomputation} applied to \eqref{eq:upbdintermediate}, and the fact that $\frac{1}{N} \E_0 \sum_{i = 1}^N X_{0,i} = \E_{0} X_{0,1} = \E_{0} X_{0}$ as defined in \eqref{recur} imply that
		\[
		F_N^{SK}(\Sigma_{\epsilon}(S,M)) \leq- \mu S - \lambda M + \E_{0} X_0 - \frac{\ba^2}{4} \sum_{k = 0}^{r - 1} \zeta_k ( Q^2_{k + 1} - Q^2_k ) +  \epsilon ( |\mu| + |\lambda| ) +o(\epsilon) + o(N^{-1}).
		\]
		Proposition \ref{prop:upbd:SK} follows.
	\end{proof}
	
	%\color{red} Check to make sure that the randomness in $x_0$ doesnt pose a problem if we average conditionally on the $x_0$? The increments are independent, and they are constrained, so perhaps we can average it at the very end. The optimization is a bit strange because we have $Q_r$ and $S$ so we have an extra parameters even in the RS case. If the nishimori property holds, then we need to take $Q = M = S$ for the problem to work out. This is because the Ruelle probability cascades will imply that the diagonal and offdiagonal terms must be the same. \color{black}

	\section{Large deviation lower bound}\label{sec:ldlb}

	We now derive the matching lower bound of the free energy, namely the lower bound of Theorem \ref{technicalldp}. In fact, we prove a slightly stronger result concerning the quenched  restricted free energy: 
	
	\[
	F_N^{Y}( \bar\ba : A) = \frac{1}{N}   \log \int \1(\bx\in A) e^{H_N^{\bar\ba}(\bx)} \, d \pP_X^{\otimes N}(\bx) 
	\]
	and,  recalling that the Hamiltonian $H_{N}^{\bar\ba}$ depends on  $Y=(W,\bx^{0})$,  denote $\E[ .|\bx^{0}]=E_{W}$  the expectation conditionnally to $\bx^{0}$, namely with respect to $W$ only. Recall $\varphi_{\bar\ba}$ is defined in \eqref{defvarphi}. 

Since atypical values of $\bx^0$ can cause infinite values of the random constrained free energy
	\[
	\frac{1}{N} \log \int_{\Sigma_{\epsilon}(S,M)} e^{H_N(\bx)} \, d \pP_X^{\otimes N}(x)
	\]
	when $\{ \bx \mmm |R_{1,0}-M|\le\epsilon \} = \emptyset$, we need to restrict our analysis to avoid these atypical values. Let
	\[
	\hat \pP_0 = \frac{1}{N} \sum_{i = 1}^N \delta_{x_i^0}
	\]
	denote the empirical measure of $\bx^0$. The Wasserstein 1 metric on $\mathcal{P}(\R)$ is given by
	\[
	d(\pP,\pQ) = \sup \bigg\{  \bigg| \int f \, d \pP - \int f \, d \pQ \bigg| \mmm \|f\|_L \leq 1 \bigg\}
	\] 
	where $\|f\|_L$ is the smallest Lipschitz constant of $f$.
	We denote the $\delta$ ball of empirical measures with
	\[
	\cB_\delta = \{ \bx^{0}\in\mathbb R^{N}:d( \hat \pP_0, \pP_0 ) \leq \delta \} .
	\]
	We may restrict $\bx^0$ to this set without changing the limit of the free energy by Lemma~\ref{lemdelta}. A specific rate of decay for $\delta$ can be chosen later.
 
	\iffalse \begin{prop}[Lower Bound of the Free Energy] Assume that $\pP_{X}$ and $\pP_{0}$ satisfy Hypothesis \ref{hypcompact}.  For any real numbers $\ba,\bb,\bc$,
		for any $S, M \in \cC$,  for any $\epsilon>0$, any $\delta>0$ small enough, uniformly on $\bx^{0}$ such that $d(\frac{1}{N}\sum_{i=1}^{N}\delta_{x_{i}},\pP_{0})\le \delta$,
		we have
		$$ \liminf_{N \to \infty}  \E[F_{N}^{Y}(\beta:\Sigma_{\epsilon}(S,M))||\bx^{0}]\ge \varphi_{\bar\ba}(S,M)+O(\delta)$$	
	\end{prop}
	\fi
	\begin{prop}[Lower Bound of the Free Energy] \label{prop:lwbd} Assume that $\pP_{X}$ and $\pP_{0}$ satisfy Hypothesis \ref{hypcompact}.  For any real numbers $\ba,\bb,\bc$,
		for any $S, M \in \cC$,  for any $\epsilon>0$, any $\delta>0$ small enough, we have
		$$ \liminf_{N \to \infty}  \E \1_{\cB_\delta} [F_{N}^{Y}(\bar\beta:\Sigma_{\epsilon}(S,M))|\bx^{0}]\ge \varphi_{\bar\ba}(S,M)+O(\delta)$$	
	\end{prop}
	Again, it is enough to concentrate on the case where $\bb=\bc=0$ since the corresponding terms are almost constants on $\Sigma_{\epsilon}(S,M)$.   We therefore in the rest of this section restrict ourselves to $\bb=\bc=0$. The partition function is  then the standard SK Hamiltonian $H_{N}^{SK}(\bx)$ with constrained self overlaps $R_{11}$ and magnetizations $R_{10}$. We will use a regularizing perturbation and the cavity computations to compute the first term. Moving forward, we will focus on proving a lower bound for
	\begin{eqnarray}\label{freeSK}
		F_{N}^{SK}(\Sigma_{\epsilon}(S,M))(\bx^{0})&:= &\E[F_{N}((\ba,0,0):\Sigma_{\epsilon}(S,M))|\bx^{0}]\nonumber\\
		&=&\frac{1}{N} \E_{W} \log\int_{\Sigma_\epsilon(S,M)} e^{ \sum_{i<j} \ba\frac{W_{ij}}{ \sqrt{  N}} x_ix_j } \, d \pP_X^{\otimes N} (\bx)
	\end{eqnarray}
	uniformly on $\bx^{0}$ in $ {\cB_\delta}$. We often denote in short $F_{N}^{SK,\epsilon}(S,M)=F_{N}^{SK}(\Sigma_{\epsilon}(S,M))(\bx^{0})$  for simplicity.
	
	We will proceed using the cavity computations on the localized free energies to discover that the lower bound of the free energy is a continuous functional of the distribution of the overlap array generated by samples from a Gibbs measure. The key intuition behind this proof is that the constrained array of configurations
	\[
	(R_{\ell,\ell'} )_{\ell,\ell' \geq 0}
	\]
	has constant diagonals (after a small change variables) so we only need to understand the distribution of the offdiagonal elements $R_{\ell,\ell'}$ for $\ell \neq \ell'$ and $\ell, \ell' > 0$. Arrays of this form are well studied and its limiting distribution can be characterized if it satisfies an invariance property called the Ghirlanda--Guerra identities. 
	
	The main difficulty in contrast to the usual spin glass models is that the restriction $\1(|R_{1,0} - M|\le\epsilon)$ depends on $\bx^0$, so extra care has to be done to verify that the crucial concentration of measure and annealed large deviations estimates hold in the setting. For technical reasons, it will be easier to work with a $C^1$ approximation of the indicator function and a restriction of the empirical measure of finite samples from $\pP_0$. These will be explained in the following subsections.

	\subsection{Large Deviations under the reference measure}
	
	To compute the lower bound, we will have to restrict ourselves to values of $(S,M)$ such that $\Sigma_{\epsilon}(S,M)$ has  finite entropy. In this section, we will explicitly compute a large deviations rate function for the reference measure. Recall the following notation
	\begin{equation}\label{eq:convexhulls}
		S \in \mathrm{conv}\{ x^2 \mmm x \in \supp(\pP_X) \} =:\cS \qquad M \in \mathrm{conv}\{ x x^0 \mmm x \in \supp(\pP_X), x^0 \in \supp(\pP_0) \} =: \cM
	\end{equation}
	where $\mathrm{conv}$ is the closed convex hull. In other words, 
	\begin{equation}\label{eq:supportconstraint}
		\inf_{x \in \supp(\pP_X)} x^2 \leq S \leq \sup_{x \in \supp(\pP_X)} x^2 \quad\text{and}\quad \inf_{\substack{x \in \supp(\pP_X) \\ x^0 \in \supp(\pP_0)}} xx^0 \leq M \leq  \sup_{\substack{x \in \supp(\pP_X) \\ x^0 \in \supp(\pP_0)}} xx^0.
	\end{equation}
	Moreover, we also know that by Cauchy-Schwarz inequality, $S,M$ must satisfy
	\begin{equation}\label{CS} 
		M\le \sup_{ x^0 \in \supp(\pP_0)} |x^0| \sqrt{S}
	\end{equation}
	%We denote by $\mathcal D$ the (closed) set of $(S,M)\in \mathcal S\times \mathcal M$ satisfying \eqref{CS}. 
	In fact, we more precisely  see that  $(S,M)$ should belong to the set $\mathcal C$ defined in \eqref{defC} since we have:
	\begin{lem} For any $
		\delta>0$, for any real numbers $r,t$ in $[-1,1]^{2}$,
		$$\E_{x^{0}}[\mbox{essinf}_{x}\{r x^{2}+ tx x^{0}\}]-C^{2}\delta\le r R_{1,1}+t R_{1,0} \le  \E_{x^{0}}[\mbox{esssup}_{x}\{r x^{2}+ tx x^{0}\}]+C^{2}\delta\,,$$
		uniformly on $x^{0}\in \cB_{\delta}$. 
	\end{lem}
	\begin{proof}For any for any real numbers $r,t$ in $[-1,1]^{2}$
		
		$$\frac{1}{N}\sum_{i=1}^{N}\mbox{essinf}_{x}\{r x^{2}+ tx x^{0}_{i}\} \le rR_{1,1}+tR_{1,0} \le  \frac{1}{N}\sum_{i=1}^{N} \mbox{esssup}_{x}\{r x^{2}+ tx x^{0}_{i}\}\,.$$
		But $x^{0}\mapsto \mbox{essinf}_{x}\{r x^{2}+ tx x^{0}\}$ and $x^{0}\mapsto \mbox{esssup}_{x}\{r x^{2}+ tx x^{0}\}$ are Lipschitz with constant bounded by $C^{2}$ and hence uniformly on $\mathcal B_{\delta}$ 
		$$\E_{x^{0}}[\mbox{essinf}_{x}\{r x^{2}+ tx x^{0}\}]-C^{2}\delta\le rR_{1,1}+tR_{1,0} \le  \E_{x^{0}}[\mbox{esssup}_{x}\{r x^{2}+ tx x^{0}\}]+C^{2}\delta\,.$$
		%We finally can let $\epsilon$ going to zero to conclude. 
	\end{proof}

	For $(\lambda,\mu)\in\mathbb R^{2}$, consider the annealed log Laplace transform
	$$\Lambda(\lambda,\mu):=\int \left(\log \int  e^{ \lambda x^2 + \mu x x^0  } \, d \pP_X(x)\right)d\pP_{0}(x^{0})$$
	and  consider the rate function on $(\mathbb R)^{2}$ given by 
	\begin{equation}\label{eq:ratefunvolume}
		\cI(S,M)=\sup_{(\lambda,\mu)\in\mathbb R^{2}}\{ I_{S,M}(\lambda,\mu)
		\}, \mbox{ with } I_{S,M}(\lambda,\mu)=\lambda S + \mu M -\Lambda(\lambda,\mu)\,.
	\end{equation}
	We have the following large deviations principle. 
	
	\begin{prop}[Large Deviations of the Entropy Term] \label{prop:ratefunction}   Assume that $\pP_{X}$ and $\pP_{0}$ satisfy Hypothesis \ref{hypcompact}. The law of the overlaps $(R_{1,1},R_{1,0})$  under $\mathbb P_{X}^{\otimes N}$ satisfies a quenched large deviations principle  with good rate function $\cI$. Moreover, we have the following quantitative weak large deviation principle:
		
		\begin{itemize}
			\item For any  $S,M\in\mathcal C$ and $\epsilon>0$, any $K > 0$, $\delta>0$, uniformly on $\bx^{0}\in\cB_{\delta}$, 
			\begin{equation}\label{wldub}
				\limsup_{N \to \infty} \frac{1}{N}  \log  \pP^{\otimes N}_X(\Sigma_{\epsilon}(S,M))  \leq -\sup_{|\mu| + |\lambda| \leq K} I_{S,M}(\lambda,\mu) 		
				+ o_K(\epsilon)+O(\delta).
			\end{equation}
			\item 	For any  $S,M$ in the interior of $ \mathcal C$ and  for any $\epsilon>0$ there exists $\delta(\epsilon)>0$ so that for $\delta\le \delta(\epsilon) $, uniformly on $\bx^{0}\in\cB_{\delta}$, 
			\[
			\liminf_{N \to \infty} \frac{1}{N} \log   \pP^{\otimes N}_X(\Sigma_{\epsilon}(S,M))  \geq - \cI(S,M) +o(\epsilon)+o(\delta)
			.
			\]
			\item 	For any  $S,M$ in the boundary of $ \mathcal C$, and  for any $\epsilon>0$ there exists $\delta(\epsilon)>0$ so that for $\delta\le \delta(\epsilon) $, 
			\begin{equation}
				\liminf_{\epsilon\rightarrow 0} \liminf_{\delta\rightarrow 0}\liminf_{N \to \infty} \inf_{\bx^{0}\in\cB_{\delta}}\frac{1}{N}  \log  \pP^{\otimes N}_X(\Sigma_{\epsilon}(S,M))  \geq - \cI({S,M}).\label{wldlb}
			\end{equation}
			Note that because these estimates are uniform on the balls $\cB_{\delta}$, they also hold if we take expectation over such $\bx^{0}$. 	
		\end{itemize}
		Here $o(\epsilon)$ and $o_{K}(\epsilon)$ go to zero uniformly for $K$ in a compact set.
	\end{prop}
	
	The large deviation result is a quenched version of Cram\`er's theorem. It can be for instance deduced from \cite[Theorem 2.2]{BDG} which gives a quenched large deviation principle $\frac{1}{N}\sum_{i=1}^{N}\delta_{x_{i},x_{i}^{0}}$ under the condition that $\frac{1}{N}\sum_{i=1}^{N}\delta_{x_{i}^{0}}$ converges towards $\pP_{0}$, which is almost surely true, and the contraction principle based on the remark that $\mu\rightarrow (\int x_{1}^{2}d\mu(x_{1},x_{0}),\int x_{1} x_{0}d\mu(x_{1},x_{0}))$ is continuous as $\mu$ is a probability measure on the bounded set $\cC$. This result is  also a special case of Lemma~\ref{lem:sharpupbd} which we will prove in Section~\ref{sec:ldlb}. 
	
	\subsection{Smoothing the Indicator}\label{sec:indicator}
	
	A critical step in the validity of the Ghirlanda--Guerra identities is the rate of concentration. The concentration of the Gaussian terms are immediate from classical Gaussian concentration inequalities, but the concentration with respect to $\bx^0$ is more technical in this setting. 
	
	\iffalse
	To ensure the validity of the Ghirlanda--Guerra identities, we need the function \color{red} the perturbation is not defined until later, so maybe it will be useful to move this section earlier \color{black}
	\[
	\phi  = \log \int \1(|R_{1,1}-S|\le\epsilon) \1(R_{1,0} \approx_\epsilon M) e^{H^\pert_N(\bx)} \, d \pP_X(\bx).
	\]
	to satisfy the concentration estimate
	\[
	\sup \Big\{  \E | \phi - \E \phi| \mmm 0 \leq t_p \leq 3, p \geq 1 \Big\} \leq v_N(s) = o(N).
	\]
	\fi
	The main technical difficulty comes from the fact that the indicator $\1(|R_{1,0} - M |\le \epsilon )$ is not differentiable in $\bx^0$ and the logarithm is unbounded if $\{|R_{1,0} - M|\le \epsilon\}$ is the empty set. Such large variations with respect to the realization of $\bx^0$ makes the verification of concentration trickier. We will do the following regularization of the indicator by using a special uniform $C^1$ approximation of the indicator function $\1(|R_{1,0} - M|\le\epsilon)$. Given any $\epsilon,\cL > 0$ and $M \in \R$ we define %(see the graph \href{https://www.desmos.com/calculator/gymdxza3w9}{here})
	\[
	\chi_{M,\epsilon}^{N} (x) = e^{- \cL N (  x - M  - \epsilon )_+^2- \cL N (  M-x - \epsilon )_+^2} =e^{- \cL N (  |x - M|  - \epsilon )_+^2}= \begin{cases}
		1 & |x - M| \leq \epsilon\\
		e^{- \cL N (  x - M  - \epsilon )^2} & x - M > \epsilon\\
		e^{- \cL N (  M-x - \epsilon )^2} & x - M <- \epsilon\\
	\end{cases}
	\]
	where $f(x)_+ = \max(f(x),0)$. The constant $\cL = \cL(S,M,\epsilon,\ba)$ is a very large constant that is independent of $N$ chosen so that it dominates the entropy and Hamiltonian. We will take
	\begin{equation}\label{eq:approxindicatorconstant}
		\cL(S,M,\epsilon,\ba) = \frac{32 K}{\epsilon^2} \bigg( \mathcal I(S,M) + KC^2 \ba + 1 \bigg)
	\end{equation}
	where $\cI$ is given in \eqref{eq:ratefunvolume}, and $K=L_{0}+1$ is a universal constant where $L_{0}$ is the constant that appears in the tail bound for the operator norm on random matrices \eqref{conc2}. Since we are considering sets with finite entropy, the constant $\cL$ is finite. 
	The function $\chi=\chi^{N}_{M,\epsilon}$ satisfies the following nice properties
	\begin{enumerate}
		%\item $\psi$ tends to $0$ sufficiently slowly.
		\item $\chi$ has bounded derivatives
		\item $\chi$ converges pointwise almost everywhere and in $L^1$ to the indicator function.
		\item $\chi > 0$ so the log partition function is never infinite if we encounter atypical values of $x^0$. 
	\end{enumerate}
	Furthermore, we have enough control over the rate of decay, and there is enough flexibility in the usual perturbations in spin glasses to account for this smoothing. 
	
	We need to show that we can replace the indicator $\1(|R_{1,0} - M| < \epsilon)$ with $\chi_{M,\epsilon/2}^{N}(R_{1,0})$ to arrive at a lower bound of the free energy when $S,M$ have finite entropy $\cI$.Recall the SK free energy $F_{N}^{SK}(\Sigma_{\epsilon}(S,M))(\bx^{0})$ averaged over $W$ only, as defined in \eqref{freeSK}.
	\begin{lem}[Smoothing the Indicator] \label{lemsmooth} Let  $\ba$  and $T$ be finite real numbers.  	For $S,M \in \cC$ such that $\{\cI(S,M)\le T\}$,  we have for every $\epsilon>0$ that there exists $\delta(\epsilon)>0$ so that for $\delta\in (0,\delta(\epsilon)]$, 
		\[
		\liminf_{\delta\rightarrow 0}\liminf_{N\to \infty} \E[1_{\bx^{0}\in \cB_{\delta}}F_{N}^{SK}(\Sigma_{\epsilon}(S,M))] 
		\ge \liminf_{\delta\rightarrow 0}\liminf_{N\to \infty}\tilde F_N^{SK}(\beta,\delta,\frac{\epsilon}{2}:S,M) 
		\]%:=\liminf_{\delta\rightarrow 0}\liminf_{N\to \infty} \frac{1}{N} \E \1_{\cB_\delta} \log Z^{SK,\epsilon}_N(S,M) \qquad\qquad$$
		%$$\qquad\qquad\geq \liminf_{\delta\rightarrow 0} \liminf_{N \to \infty} :=	\]	
		%
		where  $\tilde F_N^{SK}(\beta,\delta,\epsilon:S,M) =\frac{1}{N} \E \1_{\cB_\delta} \log \tilde Z_N^{SK}(\beta,\epsilon:S,M)$ for
		\[\tilde Z_N^{SK}(\beta,\epsilon:S,M)= \int \1(|R_{1,1}-S|\le\epsilon) \chi_{M,\epsilon}^{N}(  R_{10})  e^{H^{SK}_N(\bx)} \, d\pP_X^{\otimes N}(\bx).
		\]
		with $\cL=\cL(T,\epsilon)$ given by \eqref{eq:approxindicatorconstant} with $\cI(S,M)$ replaced by $T$. The same result holds for the quenched free energy:
		\[\liminf_{\delta\downarrow 0}
		\liminf_{N\to \infty} \inf_{\bx^{0}\in\cB_{\delta}}F_N^{SK}(\Sigma_{\epsilon}(S,M))(\bx^{0}) \geq \liminf_{\delta\downarrow 0} \liminf_{N \to \infty} \inf_{\bx^{0}\in\cB_\delta}\E_{W}\frac{1}{N}  \log \tilde Z_N^{SK}(\beta,\epsilon:S,M)
		%:=\liminf_{N\to \infty}\tilde F_N^{SK}(\beta,\delta:\Sigma_\epsilon(S,M)) 
		\]	
	\end{lem}
	\begin{rem}
		We can restrict ourselves to $S$ and $M$ with finite entropy
		because if $\mathcal I(S,M)=+\infty$, then for any $\epsilon>0$ and $\delta$ small enough, 
		\begin{equation}\label{bsup}
			\limsup_{\epsilon\rightarrow 0}\limsup_{N\to \infty} \E[1_{\bx^{0}\in \cB_{\delta}}F_{N}^{SK}(\beta:\Sigma_{\epsilon}(S,M))] =-\infty
			\,.\end{equation}
		This is because the Hamiltonian is bounded by $(L_{0}+1)C^2 \ba N$ with overwhelming probability according to \eqref{conc2}, so that  if $\mathcal I(S,M)$ is $+\infty$, we also get by Proposition \ref{prop:ratefunction}   
		\begin{align*}
			&\limsup_{N \to \infty} \E[1_{\bx^{0}\in \cB_{\delta}} F_{N}^{SK}(\beta:\Sigma_{\epsilon}(S,M))] 
			\\&\leq  \limsup_{N \to \infty} \frac{1}{N} \E 1_{\bx^{0}\in \cB_{\delta}}\log \int \1(|R_{1,0} - M| < \epsilon  ) \1(|R_{1,1} - S| < \epsilon  ) e^{N \|W\|_{op} \ba C^2 } \, d\pP_X^{\otimes N}(\bx) 
			\\&\leq  -\sup_{|\mu| + |\lambda| \leq K} I_{S,M}(\lambda,\mu) 		+(L_{0}+1)\beta C^{2}
				+ o_K(\epsilon)+O(\delta).
		\end{align*}
		where we finally assumed $\epsilon>0$ and $\delta\le \delta(\epsilon)$ and $K>0$.
		The above right hand side goes to $-\infty$ as $\delta$ and then $\epsilon$ goes to zero, and then $K$ goes to infinity.
	\end{rem}
	
	\begin{proof}
		For  any $T$ finite and $(S,M)\in\{\mathcal I\le T\}$, we will prove that uniformly on $\bx^{0}\in \cB_{\delta}$,
		\begin{align}
			&\frac{1}{N} \E_{W} \log \int \1(|R_{1,1} - S| < \epsilon  )\1(|R_{1,0} - M| < \epsilon  )e^{H^{SK}_N(\bx)} \, d \pP_X^{\otimes N}(\bx) \notag
			\\&\geq \frac{1}{N} \E_{W} \log \int \1(|R_{1,1} - S| < \epsilon  ) \chi_{M,\epsilon/2}^{N} (R_{10}) e^{H^{SK}_N(\bx)} \, d \pP_X^{\otimes N}(\bx) + o(N) .\label{eq:lowerboundindictorsummary}
		\end{align}
		In the next section, we will do the cavity computations with respect to the free energy of the approximate indicator function. Since $\chi_{M,\epsilon/2}^{N} \leq 1$, we have the obvious lower bound
		\begin{align}
			&\frac{1}{N}\E_{W} \log \int \1(|R_{1,1}-S|\le\epsilon) \1(|R_{1,0} -M|\le \epsilon)  e^{H^{SK}_N(\bx) } \, d\pP_X^{\otimes N}(\bx) \notag
			\\&\geq \frac{1}{N}\E_{W} \log \bigg( \int \1(|R_{1,1}-S|\le\epsilon)  e^{- \cL N ( | R_{1,0} - M | - \frac{\epsilon}{2} )_+^2} e^{H^{SK}_N(\bx) } \, d\pP_X^{\otimes N}(\bx) \notag
			\\& - \int \1(|R_{1,1}-S|\le\epsilon) \1( |R_{1,0} - M| \geq \epsilon)  e^{- \cL N ( | R_{1,0} - M | - \frac{\epsilon}{2} )_+^2} e^{H_N^{SK}(\bx) } \, d\pP_X^{\otimes N}(\bx) \bigg) \label{eq:lowerboundindicator}.
		\end{align}
		We need to show that the second term is negligible when compared to the first. We define the random variables
		\[
		Z_1(N) :=  \int \1(|R_{1,1}-S|\le\epsilon)  e^{-\cL N ( | R_{1,0} - M | - \frac{\epsilon}{2} )_+^2} e^{H^{SK}_N(\bx)} \, d\pP_X^{\otimes N}(\bx) (\bx)
		\]
		and
		\[
		Z_2(N) := \int \1(|R_{1,1}-S|\le\epsilon) \1( |R_{1,0} - M| \geq \epsilon)  e^{-\cL N ( | R_{1,0} - M | - \frac{\epsilon}{2} )_+^2} e^{H^{SK}_N(\bx) } \, d\pP_X(\bx).
		\]
		We have  for all $\bx^{0}\in\cB_{\delta}$
		\begin{align*}
			\frac{1}{N} \E_{W}\log\big( Z_1(N) \big) - \frac{1}{N} \E_{W} \log\big( Z_1(N) - Z_2(N) \big) &= \frac{1}{N} \E_{W} \log\bigg(1+ \frac{ Z_2(N)}{Z_1(N) - Z_2(N)} \bigg) 
			%\\&\leq \frac{1}{N} \E \1_{\cB_\delta} \bigg(\frac{ Z_2(N)}{Z_1(N) - Z_2(N)} \bigg).
		\end{align*}
		Our goal is to show that the RHS tends to $0$. Notice that
		\[
		Z_1(N) - Z_2(N) = \int \1(|R_{1,1}-S|\le\epsilon) \1( |R_{1,0} - M|\le\epsilon)  e^{-\cL N ( | R_{1,0} - M | - \frac{\epsilon}{2} )_+^2 + H^{SK}_N(\bx)}  \, d\pP_X^{\otimes N}(\bx) (\bx).
		\]
		On the set $\{ \|W\|_{op} \leq \sqrt{N} K \}$, the Hamiltonian is of order $N$ so
		\[
		Z_2(N) \leq \int \1(|R_{1,1}-S|\le\epsilon) e^{ - \cL N (\frac{\epsilon}{2})^2 + K \ba N C^2 } \, d \pP_X^{\otimes N}(\bx)  \le e^{-\frac{1}{8}\mathcal L N \epsilon^{2}}
		\]
		because $\cL$ defined in \eqref{eq:approxindicatorconstant} was chosen to dominate the term $\frac{4K \ba}{ \epsilon^{2}}C^2$. Moreover, 
		\[
		Z_1(N) - Z_2(N) \geq e^{ - NK\ba C^{2} } \int \1(|R_{1,1} - S|\le\epsilon/2) \1( |R_{1,0} - M|\le\epsilon/2)  \, d \pP_X^{\otimes N}(\bx) .
		\]
		On $\cB_\delta$ and for  $(S,M) \in \{\cI\le T\}$, we can use  Proposition~\ref{prop:ratefunction} (by looking at the lower bound) to conclude that for $N$ large enough, for $\epsilon>0$ there exists $\delta(\epsilon)>0$ so that for $\delta\le\delta(\epsilon)$,
		uniformly on $\bx^{0}\in\cB_{\delta}$,
		
		\[
		\log \int \1(|R_{1,1}-S|\le\epsilon) \1( |R_{1,0} -M|\le \epsilon)  \, d \pP_X^{\otimes N}(\bx) \geq -2NT
		\]
		This implies that if $\cL$ is chosen large enough following \eqref{eq:approxindicatorconstant}
		
		\begin{multline*}
			\frac{1}{N} \E_{W} \1(\| W\|_{op} \leq \sqrt{N} K) \log\bigg(1+ \frac{ Z_2(N)}{Z_1(N) - Z_2(N)} \bigg) \\
			\leq \frac{1}{N} \E_{W} \1(\| W\|_{op} \leq \sqrt{N} K) \frac{ Z_2(N)}{Z_1(N) - Z_2(N)} \le e^{-\frac{1}{16}\cL N \epsilon^{2}}.
		\end{multline*}
		On the set $\{ \|W\|_{op} > \sqrt{N} K \}$, the same computations as above implies that uniformly on $\bx^{0}\in\cB_{\delta}$ with $\delta<\delta(\epsilon)$
		
		\begin{align}
			&\frac{1}{N} \E_{W}\1(\| W\|_{op} > \sqrt{N} K) \log\bigg(1+ \frac{ Z_2(N)}{Z_1(N) - Z_2(N)} \bigg)  \notag
			\\&\leq \frac{1}{N} \E_{W}\1(\| W\|_{op} > \sqrt{N} K) \log\bigg( 1 + e^{ +\cL N \frac{\epsilon^2}{2}  + 2 N\|W\|_{op} \ba C^2 +2T N }\bigg) . \label{eq:upboundlog}
			%\\&\leq o(N) + K \E \1(\| W\|_{op} > \sqrt{N} L) \frac{\|W\|_{op}}{\sqrt{N}}.
		\end{align}
		Clearly the logarithmic term grows at most linearly in $N$ whereas the probability that $\| W\|_{op} > \sqrt{N} K$ decays exponentially fast by \eqref{conc2}. Therefore this term is neglectable. We conclude that there exists $c(\epsilon,\cL)>0$ such that
		
		$$0\le \frac{1}{N} \E_{W} \log\bigg(1+ \frac{ Z_2(N)}{Z_1(N) - Z_2(N)} \bigg)\le e^{- c(\epsilon,\cL)N}$$
		which permits to show with \eqref{eq:lowerboundindicator}
		that for $\cL$ large enough , uniformly on $\bx^{0}\in\cB_{\delta}$ with $\delta<\delta(\epsilon)$,  and for $N$ large enough
		\begin{align*}
			&\frac{1}{N}\E_{W} \log \int \1(|R_{1,1}-S|\le\epsilon) \1(|R_{1,0} -M|\le \epsilon)  e^{H^{SK}_N(\bx) } \, d\pP_X^{\otimes N}(\bx) \notag
			\\&\geq \frac{1}{N}\E_{W} \log  \int \1(|R_{1,1}-S|\le\epsilon)  e^{- \cL N ( | R_{1,0} - M | - \frac{\epsilon}{2} )_+^2} e^{H^{SK}_N(\bx) } \, d\pP_X^{\otimes N}(\bx) +e^{- c(\epsilon,\cL)N}
		\end{align*}
		This concludes the proof of Lemma \ref{lemsmooth}.

	\end{proof}

	\subsection{Perturbing the Hamiltonian and the Ghirlanda--Guerra Identities} \label{sec:GGI}
	We now explain in detail how to construct a perturbation of the Gibbs measure that will regularize the distribution of the overlaps. The usual perturbation and the Ghirlanda--Guerra identities of mixed $p$-spin models is sufficient in this setting. In the Bayes optimal setting, we can add some extra correction terms to force this perturbation to be of the form of a Gaussian estimation problem to preserve the Nishimori property, but such a step is not necessary here because the Nishimori property doesn't hold in our setting. The main challenge is ensuring that the localization of the empirical measure and the approximate indicator term do not spoil the essential concentration of the free energy.  Hereafter $(S,M)$ are fixed in $\{\cI<\infty\}$.  Notice that this implies that $S$ does not vanish as $\pP_{X}$ is not a Dirac mass at the origin. 
	To introduce the perturbed Hamiltonian let us first fix the self-overlap by setting
	
	\begin{equation}\label{eq:modifiedcoords}
		\hat \bx = \frac{\sqrt{SN}}{ \| \bx \|_2 } \bx
	\end{equation}
	The entries of $\hat \bx$ are still uniformly bounded for $\bx$ so that $R_{1,1}=\frac{1}{N}\|x\|_{2}^{2}$ is at $\epsilon$ distance of $S$, provided $\epsilon<S/2$. We will denote throughout $D$ such a uniform bound (which depends on $S$ and $C$). 
	For $p \geq 1$, consider
	\[
	g_{p}(\hat\bx) = \frac{1}{N^{p/2}} \sum_{\iii} g_{\iii}  \hat x_{i_1} \cdots  \hat x_{i_p}
	\]
	and the Gaussian process
	\begin{equation}\label{eq:pertg}
		g(\hat \bx) = \sum_{p \geq 1} 2^{-p} D^{- p } t_p  g_p( \bx)
	\end{equation}
	where the $g_{\iii}$ are independent standard Gaussians and $(t_p)_{p \geq 1}$ is a sequence of parameters such that $t_p \in [0,3]$ for all $p \geq 1$.  Notice that the covariance is bounded
	\begin{equation}\label{boundcov}
		\E g(\hat\bx^1) g(\hat\bx^2) = \sum_{p \geq 1} 4^{-p} D^{-2p} t_p^2 (\frac{1}{N}\sum_{i=1}^{N}\hat x_{i}^{1}\hat x_{i}^{2})^{p}\leq \sum_{p \geq 1} 4^{-p} D^{-2p} t_p^2 D^{2p}\le 3
	\end{equation}
	since $R_{1,2}= \frac{1}{N}\sum \hat x^{1}_{i}\hat x^{2}_{i} \leq C^2$. For $s>0$, we define the interpolating Hamiltonian as
	\begin{equation}\label{eq:parthamil}
		H^\pert_{N}(\bx) = H_{N}^{SK}( \bx) + s g(\hat \bx).
	\end{equation}
	\begin{lem}[Validity of the Perturbation] \label{lem:sizepert}
		If $s = N^\gamma$ for $1/4 < \gamma < 1/2$, then 
		\begin{itemize}
			\item 
			For every $\epsilon>0$, there exists $\delta(\epsilon)>0$ such that for $\delta\in (0,\delta(\epsilon)]$, 
			\[
			\liminf_{\delta\rightarrow 0}\liminf_{N\to \infty}  \tilde F_N^{SK}(\beta,\delta,\epsilon:S,M) 
			\ge 
			\liminf_{\delta\rightarrow 0}\liminf_{N\to \infty} \tilde  F_N^{\pert}(\beta,\delta,\epsilon/2:S,M)  \]
			with
			$ \tilde  F_N^{\pert}(\beta,\delta,\epsilon:S,M)   :=\frac{1}{N}\E[1_{\cB_{\delta}}\phi]$ if 
			\begin{equation}\label{defphi}
				\phi:=
				\log \int \1(|R_{1,1}-S|\le\epsilon) \chi_{M,\epsilon}^{N} (R_{1,0}) e^{H^\pert_N(\bx)} \, d \pP_X^{\otimes N}(\bx)\,.\end{equation}
			\item Moreover, if $\E$ denotes the expectation with respect to $W,\bx^{0}$ and the Gaussian variables $\bg$ of the perturbed Hamiltonian, 
			$$
			v_N (s):=\sup \Big\{  \E \1_{\cB_\delta} | \phi - \E \1_{\cB_\delta} \phi| \mmm 0 \leq t_p \leq 3, p \geq 1 \Big\} 
			$$
			satisfies  for any $\epsilon,\delta>0$, 
			
			\begin{equation}\label{eq:srequirement2}
				\lim_{N \to \infty} \frac{v_N(s)}{s^2} = 0.
			\end{equation}
			
		\end{itemize}
	\end{lem}
	
	\begin{proof}   The first result shows that $s_{N}$ is small enough so that the free energy is not changed. Indeed by an interpolation argument as in section \ref{secuniversality}, one can 
		check that the perturbation does not change the limit of the free energy if
		\begin{equation}\label{eq:srequirement1}
			\lim_{N \to \infty} \frac{s^2}{N} \to 0,
		\end{equation}
		because the covariance of the perturbation term satisfies
		\[
		\E s g(\hat \bx^1) s g(\hat \bx^2) \leq 3 s^2.
		\]
		We leave the details to this point to the reader, see e.g. \cite[Section~3.2]{PBook}. The second point is important as it will imply that any limit points of the limiting array of overlaps satisfies the Ghirlanda--Guerra identities on average. 
		We  therefore focus on the sufficient lower bound on the growth of $s$ to satisfy \eqref{eq:srequirement2}.

		The main difficulty in this computation is the indicators that were introduced for convenience earlier becomes a problem in this step. We fix an arbitrary sequence $t_p$. By independence, we can split the expected values into a statement about the concentration of $\bx^0$ and the Gaussian terms,
		\begin{equation}\label{eq:decompositionindicator}
			\E \1_{\cB_\delta} | \phi - \E \1_{\cB_\delta} \phi| \leq \E \1_{\cB_\delta} | \phi - \E_{W,g}  \phi| + \E_{\bx^0} \1_{\cB_\delta} | \E_{W,g} \phi - \E \1_{\cB_\delta} \phi|
		\end{equation}
		The average $\E_W$ is with respect to the `$W$' Gaussian terms in $H_N$, the average $\E_{\bx^0}$ is with respect to `$\bx^0$' terms in the approximate indicator and $\E_g$ is with respect to the `$g$' Gaussian terms $g_N(\bx)$, and $\E$ is the average with respect to all sources of randomness. The challenge with controlling these terms is that on the set $\cB_\delta$, we loose independence of the coordinates of the constrained variables. The upside is that $\cB_\delta$ is a set that occurs with high probability, so we can remove the indicator with a bit of work. We will control each term in the upper bound separately. 
		\\\\
		\textit{First Term:} Our goal is to show that
		\begin{equation}\label{eq:firsttermGGIconc}
			\E \1_{\cB_\delta} | \phi - \E_{W,g} \phi| \leq \E_{\bx^0} [ \E_{W,g} | \phi - \E_{W,g} \phi| ] \leq O( \sqrt{N + s_{N}^2} ).
		\end{equation}
		By independence, we can compute this upper bound conditionally on $\bx^0$. The inner expected value can be controlled using classical Gaussian concentration because the Gaussian terms have uniformly bounded covariance. By Gaussian concentration \cite[Theorem 1.2]{PBook},
		\begin{equation}\label{eq:varboundgaussian}
			\E_{W,g} ( \phi - \E_{W,g} \phi )^2 \leq 8 \sup_{\bx: |R_{1,1}-S|\le \epsilon} \E_{W,g} (H^\pert_{N}(\bx))^2.
		\end{equation}
		Since the entries of $\hat \bx$ are uniformly bounded by $D$  when $\bx$ is such that $ |R_{1,1}-S|\le \epsilon$
		and according to \eqref{boundcov}, we see that uniformly on $\bx$ and $\bx^{0}$, we have 
		\[
		\E_{W,g} ( H_{N}^{SK}(\hat \bx) + s g_N(\hat \bx) )^2  \le N D^4 \ba^2  + 3 s_{N}^2.
		\]
		
		From the bound on the variance \eqref{eq:varboundgaussian} and Jensen's inequality, we deduce     
		\begin{align*}
			\E_{W,g} | \phi - \E_{W,g} \phi| &\leq ( \E_{W,g} ( \phi - \E_{W,g} \phi )^2 )^{1/2} \\
			&\leq \sqrt{ 8 \sup_{\bx: |R_{1,1}-S|\le \epsilon} \E_{W,g} (H_{N}^{SK}(\hat \bx) + s g_N(\hat \bx))^2 } = O( \sqrt{N + s_{N}^2}).
		\end{align*}
		This upper bound is independent of $\bx^0$ and $\bx$, so \eqref{eq:firsttermGGIconc} follows immediately. 
		\\\\\textit{Second Term:} Our goal is to show that
		\begin{equation}\label{eq:secondtermGGIconc}
			\E_{\bx^0} \1_{\cB_\delta} | \E_{W,g}  \phi - \E \1_{\cB_\delta} \phi|  \leq  O( N^{\frac{1}{2}} ).
		\end{equation}
		We will use the bounded difference inequality, and this step is where the approximate indicator $\chi$ is used. The restriction to $\1_{\cB_\delta}$ is a nuisance in this section but is essential to prove the sharpness of the lower bound. We proceed like the first term and use the following decomposition
		\begin{align}
			\E_{\bx^0} \1_{\cB_\delta} | \E_{W,g} \phi - \E \1_{\cB_\delta} \phi| &\leq \E_{\bx^0} \1_{\cB_\delta} | \E_{W,g} \phi - \E  \phi| +  \E_{\bx^0} \1_{\cB_\delta} |\E  \phi - \E \1_{\cB_\delta} \phi|  \notag
			\\&\leq \E_{\bx^0} | \E_{W,g} \phi - \E  \phi| + | \E \1_{\cB^c_\delta} \phi|. \label{eq:vn2}
		\end{align}
		
		To control the first term in \eqref{eq:vn2}, we use the bounded difference property and a consequence of the Efron--Stein inequality. This step is where we use the smoothing of the indicator, because it gives us sufficient control over the variation of $\phi$  we do a small perturbation of $\bx^0$. We indeed show that
		\begin{equation}\label{eq:ggisecondtermx0conc}
			\E_{\bx^0} | \E_{W,g} \phi - \E  \phi| \leq O( N^{\frac{1}{2}} ).
		\end{equation}
		To see this, observe that $x\rightarrow x_{+}^{2}$ is continuously differentiable, with derivative $2x_{+}$, 
		so that 
		
		\begin{align*}
			|\E_{W,g} \partial_{x_i^0} \phi | &\leq 2 \cL \Big| \E_{W,g} \Big\langle  N ( R_{1,0} - (M - \epsilon) )_-  \frac{x_i}{N} \Big\rangle \Big|  +  2 \cL \Big| \E_{W,g} \Big\langle  N ( R_{1,0} - (M + \epsilon) )_+  \frac{x_i}{N} \Big\rangle \Big|
			\leq 16 C^3 \cL
		\end{align*}
		where $\langle\cdot \rangle$ is the average with respect to the measure 
		\[
		d\pG^{\pert}_N(\bx) = \frac{ \1(|R_{1,1}-S|\le\epsilon) e^{H^\pert_{N}(\hat \bx)} \chi_{M,\epsilon}^{N } ( R_{1,0} ) \, d \pP_X^{\otimes N}(\bx) }{ \int \1(|R_{1,1}-S|\le\epsilon) e^{H^\pert_{N}(\hat \bx)}
			\chi_{M,\epsilon}^{N } ( R_{1,0} ) d \pP_X^{\otimes N}(\bx)
		}
		\]
		Hence, 
		$\E_{W,g} \phi$ has a bounded derivative at each coordinate $x_i^0$ and $x_i^0$ is almost surely bounded by $C$, it satisfies the bounded difference inequality, 
		\[
		|\E_{W,g} \phi(x^0_1,\dots,x^0_i,\dots,x_N^0 ) - \E_{W,g} \phi(x^0_1,\dots,\tilde x^0_i,\dots,x_N^0 ) | \leq  32 C^4 \cL \,. 
		\]
		Therefore, Azuma Hoefding's inequality, see e.g  \cite[Corollary~3.2]{boucheronConc}, implies that
		\[
		\E_{x^0} ( \E_{W,g} \phi - \E \phi )^2 \leq  64 C^8 \cL^2 N
		\]
		which proves \eqref{eq:ggisecondtermx0conc} after applying Jensen's inequality.
		To control the second term in \eqref{eq:vn2} we use the fact that ${\cB_\delta}^{c}$ is an exponentially rare event to prove that
		\begin{equation}\label{eq:ggisecondtermx0conc2}
			| \E \1_{{\cB_\delta}^c} \phi| \leq O(N e^{-k \delta N })
		\end{equation}
		To prove this upper bound, first we  use the following upper bound by Jensen's inequality and monotonicity
		\begin{align*}
			& \E \1_{\cB^c_\delta} \log \int \1(|R_{1,1}-S|\le\epsilon) \chi_{M,\epsilon}^{N} (R_{1,0})  e^{H^\pert_{N}(\hat \bx)} \, d \pP^{\otimes N}_X(\bx) 
			\\
			%&\leq  \E_{x^0,W} \1_{\cB^c_\delta} \log \E_g \int e^{H^\pert_{N}(\hat \bx)} \, d \pP^{\otimes N}_X(\bx) \\
			&\leq  \E_{x^0} \1_{\cB_\delta^{c} } \log \E_{W,g} \int e^{H^\pert_{N}(\hat \bx)} \, d \pP^{\otimes N}_X(\bx)
			\\&\leq  \E_{x^0} \1_{\cB_\delta^{c} }  \bigg(\ba^{2}C^{2} N + \frac{3}{2} s^2 \bigg) .
		\end{align*}
		By Sanov's Theorem, we see that for every $\delta>0$ there exists $c_{\delta}>0$ such that 
		\[
		 \E_{x^0} \1_{\cB_\delta^{c} } \bigg( \frac{C^2}{2} N + \frac{3}{2} s^2 \bigg) \leq O( (N + s^2) e^{-c_{\delta} N} )
		\]
		since the empirical measures of iid samples concentrate. Therefore,
		\begin{equation}\label{eq:concupperbound11}
			\E \1_{\cB_\delta^{c}} \phi  \leq O ( (N + s^2) e^{- c_\delta N} ).
		\end{equation}
		We next prove an analoguous lower bound. The idea is to bound uniformly the $x_{i}$ and $x_{i}^{0}$, $i\le N$.

		\begin{align*}
			\E \1_{\cB^c_\delta}\phi &\geq \E \1_{\cB_\delta^c}  \log \big( e^{  - \ba \|W\|_{\infty} N^{3/2}-\cL C^{2}N}
			\big)\pP^{\otimes N}_X(|R_{1,1}-S|\le \epsilon)
		\end{align*}
		The term $\log ( \pP_X^{\otimes N}(|R_{1,1}-S|\le \epsilon) ) $ is of order $N$  since $S$ has finite entropy and $\E\|W\|_{\infty}$ is of order $\sqrt{N}$ at most (note that this is independent of $\bx^{0}$ and therefore of $\delta$). 
		Therefore the conclusion follows by  Sanov's theorem. Hence, we conclude that 	\begin{equation}\label{eq:concupperbound2}
			\E_{W,g} \1_{\cB^c_\delta}\phi \geq - L N e^{- c_{\delta} N}
		\end{equation}
		for some constant $L$ that only depends on the choice of the fixed model parameters. The upper bound \eqref{eq:concupperbound11} and lower bound \eqref{eq:concupperbound2} implies \eqref{eq:ggisecondtermx0conc2}.
		Since $\eqref{eq:ggisecondtermx0conc2}$ is of lower order than \eqref{eq:ggisecondtermx0conc}, the decomposition \eqref{eq:vn2} proves \eqref{eq:secondtermGGIconc}.
		To conclude, starting from \eqref{eq:decompositionindicator}, the bounds \eqref{eq:firsttermGGIconc} and \eqref{eq:secondtermGGIconc} imply
		\[
		\sup \Big\{ \E \1_{\cB_\delta} | \phi - \E \1_{\cB_\delta} \phi| \mmm 0 \leq t_p \leq 3, p \geq 1 \Big\}  \leq O( N+ s_{N}^2 )^{1/2} 
		\]
		which yields the Lemma. 
	\end{proof}
	
	We now evaluate the perturbed Hamiltonian with the $t_n$ coordinates replaced by $\bu=(u_{n})_{n\ge 0}$, iid uniform random variables on $[1,2]$, namely $g( \bx,\bu) := \sum_{p \geq 1} 2^{-p} D^{- p } u_p  g_p(\bx)$ and 
	$H^\pert_{N}(\bx,\bu) = H_{N}^{SK}(\hat \bx) + s g(\hat \bx,\bu)$  now depends on the additional random variables $u$.

	In this section, we  denote by  $\langle \cdot \rangle$  the average with respect to the perturbed Gibbs measure
	\begin{equation}
		\langle f \rangle = \frac{\int 1_{|R_{1,1}-S|\le\epsilon }\chi_{M,\epsilon}^{N}(R_{1,0}) f (\bx) e^{H_{N}^\pert(\hat \bx,\bu)} \, d\pP^{\otimes N}_X(\bx)}{\int 1_{|R_{1,1}-S|\le\epsilon }\chi_{M,\epsilon}^{N}(R_{1,0}) e^{H_{N}^\pert(\hat\bx,\bu)} \, d\pP^{\otimes N}_X(\bx)}
	\end{equation} 
	which depends on $S, M$ and $\epsilon$. By the convexity of the free energy functions, we have the following concentration estimate of the deviation of $g_p$ under the Gibbs measure from \cite[Theorem~3.3]{PBook}.
	\begin{lem}[Concentration of the Perturbed Hamiltonian]\label{lem:concpert}
		For any $p \geq 1$, if $s > 0$  is such that $s^{-2} v_N(s) \leq \frac{7}{4^{-p}}$ then
		\[
		\E_{u} \E \1_{\cB_\delta} \langle| g_p(\hat \bx) - \E \1_{\cB_\delta} \langle g_p(\hat \bx) \rangle |  \rangle \leq  C^p ( 2 + 18 \sqrt{v_N(s)} )
		\]
		where $\E$ is the average with respect to the Gaussian random variables  $W,\bg$ and $\bx^0$ and $\E_u$ is the average with respect to the uniform random variables $(u_n) \sim U[1,2]$.
	\end{lem}
	\begin{proof}
		Fix $p \geq 1$. By the triangle inequality, we have
		\begin{equation}\label{eq:GGIconc}
			\E_u \E \1_{\cB_\delta} \langle| g_p(\hat \bx) - \E \langle g_p(\hat \bx) \rangle | \rangle \leq \E_u \E \1_{\cB_\delta} \langle| g_p(\hat \bx) - \langle g_p(\hat \bx)  \rangle | + \E_u \E \1_{\cB_\delta} | \langle g_p(\hat \bx) \rangle - \E \langle g_p(\hat \bx) \rangle | 
		\end{equation}
		we will control each of these terms separately.
		\\\\	
		\textit{First Term:} We begin with the first term in \eqref{eq:GGIconc}. We fix  $\bx^0 \in \cB_\delta$. Consider
		\begin{equation}\label{defphi2}
			\phi(\bt) = \log \int 1_{|R_{1,1}-S|\le\epsilon }\chi_{M,\epsilon}^{N}(R_{1,0}) e^{H^\pert_{N}(\hat \bx,\bt)} \, d \pP^{\otimes N}_X(\bx) 
		\end{equation}
		as a function of $t = t_p$. Recall that in the definition of $g(\bx)$ in \eqref{eq:pertg} and \eqref{eq:parthamil}, the variable $t_{p}$ only appears in the term 
		$
		s 2^{-p} D^{-p} t_{p}g_p(\hat \bx) = s_p t_{p}g_p(\hat \bx)
		$
		where we defined $s_p = s 2^{-p} D^{-p} $ to simplify notation. Differentiating the free energy and integrating by parts implies that
		\begin{equation}\label{eq:boundderiv}
			\E_W \partial_{t_{p}}\phi(t) =s_p \E_W \langle g_p(\hat \bx) \rangle =s_p^2 t_{p} \E_W \langle R_{1,1}^p - R_{1,2}^p \rangle = s_p^2 t_{p} \E_W \langle S^p - R_{1,2}^p  \rangle \in [0, 2 C^{2p} s_p^2 t]. 
		\end{equation}
		since $S \leq C^2$.	The second derivative of the free energy gives the variance
		\[
		\E_W\partial_{t_{p}}^{2} \phi(\bt) =  s_p^2 \E_W ( \langle g_p(\hat \bx)^2 \rangle  - \langle g_p(\hat \bx) \rangle^2  ) = s_p^2 \E_W \langle ( g_p(\hat \bx) - \langle g_p(\hat \bx) \rangle  )^2 \rangle .
		\]
		We can integrate $t_{p}$ from $[1,2]$ to arrive  with \eqref{eq:boundderiv} at the bound
		\[
		s_p^2  \E_u \E_{W} \langle ( g_p(\hat \bx) - \langle g_p(\hat\bx) \rangle  )^2 \rangle =  \int_1^2 \E_{W}\partial_{t_{p}}^{2} \phi(\bt) \, dt_{p} = ( \E_{W} \partial_{t_{p}}\phi|_{t_{p}=2} - \E_{W}\partial_{t_{p}}\phi|_{t_{p}=1} ) \leq 4 C^{2p} s_p^2 .
		\]
		Jensen's inequality  implies that
		\begin{equation}\label{eq:bound1}
			\E_u \E \langle | g_p(\hat \bx) - \langle g_p(\hat \bx) \rangle  | \rangle \leq 2 C^p.
		\end{equation}
		This bound is uniform for $\bx^0 \in \cB_\delta$, so we can now integrate over $\E_{x^0} \1_{\cB_\delta}$.
		\\\\
		\textit{Second Term:} We first fix  $\bx^0 \in \cB_\delta$ and all other random processes other than $g_p$. We use convexity of the free energies to bound the second term in \eqref{eq:GGIconc}. We  recall $\phi$ defined in \eqref{defphi2}
		and let $\psi(\bt)=  \E  \1_{\cB_\delta} \phi(\bt)$. As in  \eqref{eq:boundderiv}, we find 
		
		\[ \partial_{t_{p}}\phi(\bt) =
		s_p \langle g_p(\hat \bx)\rangle \quad \text{and} \quad \partial_{t_{p}}\psi(\bt) = s_p \E \1_{\cB_\delta} \langle g_p(\hat  \bx) \rangle.  
		\]
		For $y \in [0,1]$, recall  from  \cite[Lemma~3.2]{PBook} that if $\phi,\psi$ are two differentiable convex functions on the real line, for any $y>0$
		$$|\phi'(x)-\psi'(x)|\le \psi'(x+y)-\psi'(x-y)+\frac{\Delta}{y}$$
		with $\Delta=|\psi(x+y)-\psi(x+y)|+|\phi(x-y)-\psi(x-y)|+|\phi(x)-\psi(x)|.$ Choosing $\phi$,$\psi$ the above functions of $t_{p}$, we next take the expectation $ \E  \1_{\cB_\delta}$.
		Recalling that $x=t_{p}$ belongs to $[1,2]$ and  taking $y\in [0,1]$ so that $u_{p},u_{p}+y$ and $u_{p}-y$ all belong to $[0,3]$ we see that $ \E  \1_{\cB_\delta}\Delta\le 3 v_{N}(s)$. 
		yielding
		if $1_{p}$ vanishes except at the $p$th coordinate where it is equal to one, uniformly on $t_{p}\in [1,2]$ and $y\in [0,1]$,
		
		\begin{equation}\label{lkj}
			\E  \1_{\cB_\delta}\
			| \partial_{t_{p}}\phi(\bt) -\partial_{t_{p}}\psi(\bt) | \leq \partial_{t_{p}}\psi(\bt+ y1_{p}) - \partial_{t_{p}}\psi(\bt - y1_{p}) + \frac{3 v_N(s)}{\delta}.
		\end{equation}
		If we integrate by parts with respect to the Gaussian terms first, we see that $|\partial_{t_{p}}\psi(\bt)| \leq 6       C^{2p} s_p^2$ for $t_{p} \in [0,3]$ by the argument in \eqref{eq:boundderiv}. The mean value theorem implies with \eqref{eq:boundderiv}
		that
		\begin{eqnarray*}
			\E_{t_{p}}[ \partial_{t_{p}}\psi(\bt+ y1_{p}) - \partial_{t_{p}}\psi(\bt - y1_{p}) ]&=&
			\int_1^2\left( \psi'(t + y) - \psi'(t - y)\right) \, dt  \\
			&=& \psi(2 + y) - \psi(2 - y) - \psi(1 + y) + \psi(1 -y) \leq 24 C^{2p} s_p^2  y.
		\end{eqnarray*}
		Therefore, if we take $t \sim U[1,2]$ and average on both sides then \eqref{lkj} yields
		\[
		s_p \E_u\E  \1_{\cB_\delta}  | \langle g_p(\hat \bx) \rangle - \E_g \1_{\cB_\delta} \langle g(\hat \bx) \rangle | = E  \1_{\cB_\delta}\
		| \partial_{t_{p}}\phi(\bt) -\partial_{t_{p}}\psi(\bt) | 
		\leq 24 C^{2p} s_p^2 y+ \frac{3v_N(s)}{\delta}.
		\]
		Recalling that $s_p = s 2^{-p} C^{-p} $, we can take the minimizing $y = \frac{v_N^{1/2}}{\sqrt{7}C^p s_p}$ which is in $[0,1]$ if $s^{-2} v_N(s) \leq \frac{7}{4^{-p}}$, then we get the bound
		\begin{equation}\label{eq:bound2}
			\E_u \E  \1_{\cB_\delta}  | \langle g_p(\hat \bx) \rangle - \E_g \1_{\cB_\delta} \langle g(\hat \bx) \rangle
			|  \leq 18 C^p \sqrt{v_N(s)}.
		\end{equation}
		Combining the inequalities \eqref{eq:bound1} and \eqref{eq:bound2} to bound \eqref{eq:GGIconc} finishes the proof.	By independence, we can also integrate with respect to the other random processes and $\bx^0 \in \cB_\delta$.
	\end{proof}
	Since we constrained the self overlaps $R_{1,1}$ to be constant then the general proof of the Ghirlanda--Guerra identities \cite[Theorem 3.2]{PBook} holds without modification. 
	
	\begin{theo}[Ghirlanda--Guerra Identities]\label{thm:GGI}Let $\hat R_{k,\ell}=\frac{1}{N}\sum_{i=1}^{N}\hat x_{i}^{k}\hat x^{\ell}_{i}$. 
		If $s = N^\gamma$ for $1/4 < \gamma < 1/2$, then
		\[
		\lim_{N \to \infty} \E_u \bigg| \E \1_{\cB_\delta} \langle f  \hat R_{1,n+1}^p \rangle - \frac{1}{n} \E\1_{\cB_\delta}  \langle f \rangle \E \1_{\cB_\delta} \langle \hat R_{1,2}^p \rangle - \frac{1}{n} \sum_{\ell = 2}^n \E \1_{\cB_\delta} \langle f\hat R^p_{1,\ell} \rangle \bigg| = 0
		\]
		for any $p \geq 1$, $n \geq 2$ and bounded measurable function $f$ of the $n \times n$ sub array of the overlaps. 
	\end{theo}
	
	\begin{proof}
		Let us fix $n \geq 2$ and consider a bounded function $f = f(R^n)$ of $n\times n$ overlaps from the array. By a scaling argument, we can assume that $\|f\|_{\infty} = 1$. We start with the inequality
		\begin{equation}\label{eq:GGIbound}
			| \E \1_{\cB_\delta} \langle f g_p(\hat\bx^1) \rangle - \E \1_{\cB_\delta} \langle f \rangle \E \1_{\cB_\delta} \langle g_p(\hat \bx) \rangle| \leq \E \1_{\cB_\delta} \langle |g_p(\hat \bx) - \E \1_{\cB_\delta} \langle g_p(\hat \bx) \rangle| \rangle.
		\end{equation}
		 To simplify notation, we set $s_p = s 2^{-p} D^{-p} $. Conditionally on $u_p$, Gaussian integration by parts implies that the left hand side simplifies to
		\[
		\left|s_pu_p \E \1_{\cB_\delta} \bigg\langle f \bigg( \sum_{\ell = 1}^n \hat R_{1,\ell}^p - n  \hat R^p_{1,n+1} \bigg) \bigg\rangle - s_pu_p \E \1_{\cB_\delta} \langle f \rangle ( \E \1_{\cB_\delta} \langle \hat R^p_{1,1} - \hat R^p_{1,2} \rangle )\right|.
		\]
		To see this, recall that the covariance of $\E g_p(\hat \bx^1)g_p(\hat \bx^2) = R_{1,2}^p$ from \eqref{boundcov} and the factor $s_p u_p$ appearing in front of $g_p(\hat\bx)$ in the perturbation Hamiltonian \eqref{eq:parthamil}. Since $f(R^n) = f(\hat\bx^1, \dots, \hat\bx^n)$
		\begin{align*}
		&\E \1_{\cB_\delta} \langle f(R^n) g_p(\hat\bx^1) \rangle = \E \1_{\cB_\delta} \frac{\int f(R^n) g_p(\hat\bx^1) e^{ \sum_{\ell = 1}^n H_N^\pert(\hat \bx^\ell) } \prod_{\ell = 1}^n d\pP^{\otimes N}_X(\bx^\ell)}{ \left( \int  g_p(\hat\bx^1) e^{ \sum_{\ell = 1}^n H_N^\pert(\hat \bx^\ell) }  d\pP^{\otimes N}_X(\bx) \right)^n } 
		\\&= s_pu_p \E \1_{\cB_\delta} \bigg\langle f(R^n) \bigg( \sum_{\ell = 1}^n \hat R_{1,\ell}^p - n  \hat R^p_{1,n+1} \bigg) \bigg\rangle
		\end{align*}
		where we treat the terms appearing in the denominator as a separate replica (see for example \cite[Exercise~1.2.1]{PBook}). The second term follows from a similar argument, but only one replica appears.
		
		Since we constrained the self overlaps $\hat R_{1,1}$ to be equal to $S$ by \eqref{eq:modifiedcoords}, we get that the left hand side of \eqref{eq:GGIbound} is equal to
		\[
		s_pu_p n \bigg| \frac{1}{n} \E \1_{\cB_\delta} \langle f \rangle \E \1_{\cB_\delta} \langle \hat R^p_{1,2} \rangle +\frac{1}{n} \sum_{\ell = 2}^n  \E \1_{\cB_\delta} \langle f \hat R_{1,\ell}^p \rangle  - \E \1_{\cB_\delta} \langle f \hat R^p_{1,n+1} \rangle \bigg| + \underbrace{ s_pu_p n S \E \1_{\cB_\delta} \langle f \rangle (1 - \E \1_{\cB_\delta} )}_{o_N(1)}.
		\]
		Since $u_p \geq 1$, we can remove the $u_p$ to arrive at a lower bound. By the concentration of the averages with respect to $g_p$ in Lemma~\ref{lem:concpert}, see \eqref{eq:bound2},  if $s^{-2}v_N(s) \leq 4^{-p} C^{-2p}$ we have the upper bound 
		\[
		\E_{u} \E \1_{\cB_\delta} \langle| g_p(\hat \bx) - \E \1_{\cB_\delta} \langle g_{p}(\hat \bx)  \rangle |  \rangle \leq C^p( 2 + 48 \sqrt{v_N(s)} ).
		\]
		We can now take the expected value of both sides of \eqref{eq:GGIbound} with respect to $u$ to conclude that
		\[
		s_p n \E_u \bigg| \frac{1}{n} \E\1_{\cB_\delta}  \langle f \rangle \E \1_{\cB_\delta} \langle \hat R^p_{1,2} \rangle +\frac{1}{n} \sum_{\ell = 2}^n  \E \1_{\cB_\delta} \langle f \hat R_{1,\ell}^p \rangle  -\E \1_{\cB_\delta} \langle f \hat R^p_{1,n+1} \rangle \bigg| \leq C^p( 2 + 48 \sqrt{v_N(s)} ).
		\]
		Rearranging, we see for $s$ sufficiently large so that $s^{-2}v_N(s) \leq 4^{-p} C^{-p}$
		\[
		\E_u \bigg| \frac{1}{n} \E \1_{\cB_\delta} \langle f \rangle \E \1_{\cB_\delta} \langle \hat R_{1,2} \rangle +\frac{1}{n} \sum_{\ell = 2}^n  \E\1_{\cB_\delta} \langle f \hat R_{1,\ell}^p \rangle  - \E \1_{\cB_\delta} 
		\langle f  \hat R^p_{1,n+1} \rangle \bigg| \leq \frac{ (2C)^p (2 + 48 \sqrt{v_N(s)} ) }{s n}.
		\]
		The first term in the upper bound clearly goes to $0$ if \eqref{eq:srequirement2} is satisfied, which is precisely when $ s = N^\gamma$ for $1/4 < \gamma < 1/2$ by Lemma~\ref{lem:sizepert}. % \color{red} check the order of the expected values in \eqref{eq:GGIbound}. This might mean we will have to average with respect to $\bx^0$ on the inside. Of course, we can do everything conditionally on $\bx^0$ and average over it at the very end, but this needs to be explained. The limit points will depend on $\bx^0$, which is kind of tricky because the dimension of $\bx^0$ also tends to $\infty$. It will be easier to adapt the concentration estimate to hold on average for the $\bx^0$ on the inside as well. *in fact, it appears that there is no dependence with the $\bx^0$ in the Hamiltonians since it was constrained to be $M$ in these computations. However, there is a dependence on $\bx^0$ in the constraint which makes it harder to work with. 
		
		%Also check to see if the limit depends on $\bx^0$ or the restriction. \color{black}
	\end{proof}
	
	This means that we can approximate the limiting distribution of the overlap array with one generated from the Ruelle probability cascades. 
	
	\begin{rem}\label{rem:deterministic}
		Since Theorem~\ref{thm:GGI} holds on average with respects to the random variables $u_{i}, i\ge 1$, there exists a non-random sequence $(t_{p,N})$ of parameters such that the Ghirlanda--Guerra identities hold in the limit by an application of the probabilistic method \cite[Lemma~3.3]{PBook}.
	\end{rem} 
	
	\subsection{Cavity Computations} \label{sec:cavityI}
	
	We can now do the standard cavity computations on the constrained perturbed log partition function with approximate indicator,  $\tilde Z^{\pert}_{N}(\beta,\epsilon:S,M)=e^{\phi}$ with $\phi$ defined in \eqref{defphi}:
	\[
	\tilde Z^{\pert}_{N}(\beta,\epsilon:S,M)= \int \1(|R_{1,1}-S|\le\epsilon)  e^{- \cL N ( | R_{1,0} - M | - \epsilon )^2_+} e^{H^\pert_N( \bx)} \, d\pP_X^{\otimes N}(\bx)\,.
	\]
	By Lemma~\ref{lem:sizepert}, it suffices to study the perturbed free energy. Consider the following cavity fields defined with respect to the modified coordinates $\hat x_{i}=\sqrt{(N+n)S} x_{i}/\|x\|_{2}$ ( see \eqref{eq:modifiedcoords}):
	\begin{equation}\label{eq:cav1}
		H_{N,n}^\pert( \bx)  := \sum_{1\le i<j\le N} \ba \frac{W_{ij}}{ \sqrt{ (N + n)}}  x_i x_j + s g_N(\hat \bx) ,
	\end{equation}
	\begin{equation}\label{eq:cav2}
		z_i(\hat \bx) =\frac{\ba}{ \sqrt{ N}}   \sum_{j = 1}^N  W_{j,N + i} \hat x_j ,
	\end{equation}
	\begin{equation}\label{eq:cav3}
		y(\hat \bx) =\frac{ \sqrt{n} \ba }{N}  \sum_{1\le i<j\le N} W_{ij} \hat x_i \hat x_j
	\end{equation}
	
	Let $R^+$ denote the overlaps of configurations $(\bx,\by) \in \R^{N + n}$, $R$ denote the overlaps  of configurations  $\bx \in \R^N$ and $R^y$  denote the overlaps  of configurations  $\by\in \R^n$.  The main goal of this section is to prove the following lower bound.
	\begin{prop}[The Cavity Computations]\label{Aiz} For  any $S$ with finite entropy, there exists a finite constant $c$, such that for   any $\epsilon>0$,
	 for 
any large enough  integer number $n$,  $\pP_{X}^{\otimes n}(|R_{1,1}-S|\le \epsilon)\ge e^{-cn}$. For such $S$, $\epsilon>0$ and  integer number $n$,  for any $\delta>0$,  the functional 
		$$\Delta_{N,n}(\beta,\epsilon,\delta: S,M):= \frac{1}{n}\left (\E 1_{\cB_{\delta}} \log \tilde Z_{N+n}^{\pert}(\beta,\epsilon:S,M)- \E 1_{\cB_{\delta}} \log \tilde Z_{N}^{\pert}(\beta,\epsilon:S,M)\right)$$
		is bounded below by
		\[
		\frac{1}{n} \bigg( \E \1_{\bx^{0}\in \cB_\delta} \1_{\by^{0}\in \cB_\delta} \log \bigg\langle \int_{|R_{1,1}(\by)-S|\le \epsilon}\chi_{M,\frac{\epsilon}{2}}(R_{1,0}(\by))  e^{\sum_{i = 1}^n z_i(\hat \bx) y_i}   \, d \pP^{\otimes n}_X(\by) \bigg\rangle_{N,n}^\pert - \E \1_{\cB_\delta} \log \bigg\langle e^{y(\hat \bx)} \bigg\rangle_{N,n}^\pert \bigg) + o(1)	\]
		where
		\[
		\langle f(\bx) \rangle_{N,n}^\pert = \frac{\int  \1(|R_{1,1}(\bx)-S|\le\epsilon)  f(\bx) e^{H^\pert_{N,n}( \hat \bx)- \cL N ( | R_{1,0} (\bx)- M | - \epsilon )^2_+} \, d\pP^{\otimes N}_X(\bx) }{\int  \1(|R_{1,1}(\bx)-S|\le\epsilon) e^{H^\pert_{N,n}(\hat \bx) -\cL N ( | R_{1,0}(\bx) - M | - \epsilon )^2_+ } \, d\pP^{\otimes N}_X(\bx)},
		\]
		with $R_{1,\ell}(y)=\frac{1}{n} \sum y_{i}y_{i}^{\ell}$
		and $\by^{0}\in \cB_\delta$ means that the empirical measure of $y^{0}\in \mathbb R^{n}$ is $\delta$-close to $\mathbb P_{0}$. Finally $o(1)
		$ goes to zero when $N$ goes to infinity,  then $n$ goes to infinity,  then $\delta$ goes to zero, and finally  $\epsilon$ goes to zero.
	\end{prop}
	The main application of this proposition follows from the simple fact about sequences that  for any integer number $n\ge 1$
	
	\begin{equation}\label{lb2}
		\liminf_{N\rightarrow \infty}\frac{1}{N}  \E 1_{\cB_{\delta}} \log \tilde Z_{N}^{\pert}(\beta,\epsilon,\delta:S,M)\ge  \liminf_{N\rightarrow \infty} \Delta_{N,n}(\beta,\epsilon,\delta:S,M)\,.\end{equation}
	We will let at the end $n$ going to infinity to get the desired lower bound. 
	\begin{proof} We follow the standard procedure of the Aizenman--Sims--Starr scheme. 
		\hfill\\\\
		\textit{Decoupling the Constraints on the Self Overlaps and Magnetizations:} In contrast to the classical spin glass models, we have to deal with the approximate indicator function in the cavity computations and the restriction on the empirical measures. Our goal is to prove that	 $\tilde F^{\pert}_{N+n}(\beta,\epsilon,\delta:S,M)=
		\frac{1}{n} \E \1_{\cB_\delta^+} \log \tilde  Z^{\pert}_{N + n} (\beta,\epsilon,\delta:S,M)$ is bounded below by 
		
		\begin{align}
			\tilde F^{\pert}_{N+n}(\beta,\epsilon,\delta:S,M)\geq \frac{1}{N+n} \E \1_{\cB_\delta^+}\log \int \chi_{S,M,\epsilon}(\bx)  \chi_{S,M,\epsilon}(\by) e^{  H_{N + n}^\pert(\bx,\by)  } \,d\pP^{\otimes n}_X(\by) d\pP^{\otimes N}_X(\bx)\label{ineqf}
		\end{align}
		where $\chi_{S,M,\epsilon}(\bx)= 1_{\{|R_{1,1}(\bx)-S|\le \epsilon\}} \chi^{N}_{M,\epsilon/2}(R_{1,0}(\bx)) $, $ \cB_\delta^+=\{ d(\frac{1}{N+n}(\sum_{i=1}^{N} \delta_{x_{i}^{0}}+\sum_{i=1}^{n}\delta_{y^{i}_{0}}),\mathbb P_{0})<\delta\}$.	 $H_{N + n}^\pert(\bx,\by)$ is defined as in \eqref{eq:parthamil} but in dimension $N+n$.
		We start by decoupling the approximate indicator function. We can write the overlaps of the enlarged system as the convex combination of overlaps in the bulk and cavity coordinates,
		\begin{align*}
			R_{1,0}^{+} := \frac{1}{N + n} \bigg( \sum_{i = 1}^N x_i x^0_i  + \sum_{i = 1}^n y_i y_i^0 \bigg) 
			&= \frac{N}{N + n} \bigg( \frac{1}{N}\sum_{i = 1}^N x_i x^0_i \bigg)  + \frac{n}{N + n} \bigg(\frac{1}{n} \sum_{i = 1}^n y_i y_i^0 \bigg) 
			\\&=: \frac{N}{N + n} R_{1,0}(\bx) + \frac{n}{N + n} R_{1,0}(\by)
		\end{align*}
		where $y_i \sim \pP_0$ for $1 \leq i \leq n$. Observe that
		\begin{align}
			e^{- \cL (N + n) ( | R^+_{1,0} - M | - \frac{\epsilon}{2} )_+^2}
			&\geq  e^{- \cL N ( |R_{10}(\bx) - M | - \frac{\epsilon}{2} )^{2}_+ - \cL n ( |R_{10}(\by) - M | - \frac{\epsilon}{2} )^{2}_+} \label{eq:decouple}
		\end{align}
		because the function
		\[
		x \mapsto (N + n) \Big( | x - M | - \frac{\epsilon}{2} \Big)^2_+
		\]
		is  convex in $x$, so the decomposition $R_{1,0}^{+} =  \frac{N}{N + n} R_{1,0}(\bx) + \frac{n}{N + n} R_{1,0}(\by)$ implies
		\[
		(N + n) \Big( |R_{10}^+ - M | - \frac{\epsilon}{2} \Big)^2_+  \leq N \Big( |R_{10} (\bx) - M | - \frac{\epsilon}{2} \Big)^2_+ + n\Big( |R_{10}(\by) - M | - \frac{\epsilon}{2} \Big)^2_+.
		\]
		%By monotonicity of $x \mapsto x^{\frac{5}{4}}$, we can conclude that
		%\begin{align*}
		%	-(N + n) \Big( |R_{10}^+ - M | - \frac{\epsilon}{2} \Big)_+^{\frac{5}{4}} &\geq - \Big( N \Big( |R_{10} - M | - \frac{\epsilon}{2} \Big)_+ + n\Big( |R_{10}^y - M | - \frac{\epsilon}{2} \Big)_+ \Big)^{\frac{5}{4}}
		%	\\&\geq \begin{cases}
			%		-N^{\frac{5}{4}} ( ( R_{1,0} - (M - \epsilon) )_- )^2 & R_{1,0}^y > M - \epsilon
			%		\\ -\infty & R_{1,0}^y \leq M - \epsilon.
			%	\end{cases}
		%\end{align*}
		which implies \eqref{eq:decouple} .
		Furthermore, we can also decouple the self overlap constraint using the fact
		\[
		\{|R_{1,1}^+ - S| < \epsilon \} \supseteq \{ \bx:  |R_{1,1}(\bx) - S| < \epsilon \} \cup \{\by:  |R_{1,1}(\by)- S| < \epsilon \} .
		\]
		This proves \eqref{ineqf}.

		\textit{The Aizenman--Sims--Starr Scheme:} We can use the usual cavity computations to decompose the Hamiltonians $H^\pert_{N + n}(\bx,\by)$ and $H^\pert_N(\bx)$ into its cavity fields (up to some other $O(N^{-1})$ terms). Let $(\bx,\by) \in \R^{N + n}$ where $\by \in \R^n$ denotes the cavity coordinates. We claim that on the set $|R_{1,1}(\bx)-S|\le \epsilon, |R_{1,1}(\by) - s|\le \epsilon$, we can use the decomposition
		\begin{equation}\label{eq:pertzcoords}
			H^\pert_{N + n} (\bx,\by) = H^\pert_{N,n}( \bx) + \sum_{i = 1}^n y_i z_i(\hat \bx) + o_{N,n}(1) + O(\epsilon),
		\end{equation}
		where the corresponding cavity fields are defined in \eqref{eq:cav1} and \eqref{eq:cav2}, without changing the limit of the free energy. The order $\epsilon$ term comes from the error in the change of variables when we renormalize $\bx$ or $(\bx,\by)$ to get $\hat \bx$.  $ o_{N,n}(1) $ comes from the quadratic terms in the $y_{i}$'s which is small as soon as $n^{2}/\sqrt{N}$ goes to zero. 
		We can replace $H_{{N+n}}^{\pert}(\bx,\by)$ by $ H^\pert_{N + n} (\bx,\by)$
 using the standard interpolation argument for the Aizenman--Sims--Star scheme for the SK model \cite[Theorem~3.6]{PBook}. We first   show that we can replace  the Hamiltonian $H^\pert_{N + n} (\bx,\by)$ by the Hamiltonian 
		\[
		\tilde H^\pert_{N + n} (\bx,\by) = H^\pert_{N,n}( \bx) + \sum_{i = 1}^n y_i z_i(\bx) 
		\]
		without changing the limit of the free energy. Next, we can replace this Hamiltonian $\tilde  H^\pert_{N + n} (\bx,\by)$ by
		$\bar H^\pert_{N + n} (\bx,\by):=  H^\pert_{N,n}( \bx) + \sum_{i = 1}^n y_i z_i(\hat \bx)$
		through the interpolating Hamiltonian
		\[
		z_N(\bx, \by,t)= \sqrt{t} \sum_{i = 1}^n y_i z_i(\bx) +  \sqrt{1- t} \sum_{i = 1}^n y_i z_i(\hat \bx)
		\]
		and using the definition of $\hat\bx$ in \eqref{eq:modifiedcoords} implies that $|R_{11} - \hat R_{11}| \leq \epsilon$ and the coordinates of $y$ are bounded to conclude that
		$$
			\tilde F^{\pert}_{N+n}(\beta,\epsilon,\delta:S,M)\geq \frac{1}{N+n} \E \1_{\cB_\delta^+}\log \int \chi_{S,M,\epsilon}(\bx)  \chi_{S,M,\epsilon}(\by) e^{ \bar H_{N + n}^\pert(\bx,\by)  } \,d\pP^{\otimes n}_X(\by) d\pP^{\otimes N}_X(\bx)+ o_{N,n}(1) + O(\epsilon).$$
		Similarly, we can decompose the original cavity field into
		\[
		H^\pert_{N} (\bx) = H^\pert_{N,n}(\bx) + y(\hat \bx) + o_N(1)  +  O(\epsilon)
		\]
		where the corresponding cavity field was defined  in \eqref{eq:cav1} and \eqref{eq:cav3}.
		It follows that for any $\epsilon>0$, with the notation of Proposition \ref{Aiz}
		\begin{align*}
			&\Delta_{N,n}(\Sigma_{\epsilon}(S,M))
			\\&=  \frac{1}{n} \bigg( \E \1_{\cB_\delta^+} \log \int \chi_{S,M,\epsilon }(\bx)  \bigg( \int \chi_{S,M,\epsilon }(\by)  e^{\sum_{i = 1}^n z_i(\hat \bx) y_i }  \, d \pP^{\otimes n}_X(\by) \bigg)  e^{H^\pert_{N,n}( \bx)} d \pP^{\otimes n}_X(\bx) 
			\\&- \E \1_{\cB_\delta} \log \int {\chi_{S,M,\epsilon}(\bx)} e^{y(\hat \bx)} e^{H^\pert_{N,n}( \bx)}  d \pP^{\otimes n}_X(\bx) \bigg)+o_{N,n}(1)+O(\epsilon).
		\end{align*}
		\color{black}
		By adding and subtracting the normalization terms our lower bound becomes
		\begin{align}\label{topr}
			\frac{1}{n} \bigg( \E \1_{\cB_\delta^+} \log \bigg\langle \int \chi_{S,M,\varepsilon}(\by) e^{\sum_{i = 1}^n z_i(\hat \bx) y_i  }   \, d \pP^{\otimes n}_X(\by) \bigg\rangle_{N,n}^\pert - \E \1_{\cB_\delta} \log \bigg\langle e^{y(\bx)}  \bigg\rangle_{N,n}^\pert \bigg) + \mbox{Err}
		\end{align}
		where $\langle \cdot \rangle_{N,n}^\pert$ is the average with respect to the Gibbs measure with density  proportional to   $e^{H^\pert_{N,n}(\bx)}\chi_{S,M,\varepsilon}(\bx)$ and
		\[
		\mbox{Err} = \frac{1}{n} \E (\1_{\cB_\delta^+} - \1_{\cB_\delta})  \log \int \chi_{S,M,\varepsilon}(\bx)e^{H^\pert_{N,n}(\bx)} \, d\pP^{\otimes N}_X(\bx).
		\]
		We next show that $\mbox{Err}$ goes to zero.
		Let use denote the expected value $\langle f(\bx) \rangle_{\chi} = \frac{\int \chi_{S,M,\epsilon}(\bx)  f(\bx) \, d\pP^{\otimes N}_X(\bx)}{\int  \chi_{S,M,\epsilon}(\bx)   \, d\pP^{\otimes N}_X(\bx)}$. We first consider the region where $\1_{\cB_\delta^+} - \1_{\cB_\delta} > 0$. Applying Jensen's inequality implies that
		\begin{align}
			&\E_{x^0} \E_{W} (\1_{\cB_\delta^+} - \1_{\cB_\delta}) \1(\1_{\cB_\delta^+} - \1_{\cB_\delta} > 0) \log \int \chi_{S,M,\epsilon}(\bx)  e^{H^\pert_{N,n}(\bx)} \, d\pP^{\otimes N}_X(\bx) \notag
			\\&\geq \E_{x^0} \E_{W} (\1_{\cB_\delta^+} - \1_{\cB_\delta}) \1(\1_{\cB_\delta^+} - \1_{\cB_\delta} > 0) \log \langle e^{H^\pert_{N,n}(\bx)} \rangle_{\chi}  \notag
			\\&\qquad + \E_{x^0} \bigg[ (\1_{\cB_\delta^+} - \1_{\cB_\delta}) \1(\1_{\cB_\delta^+} - \1_{\cB_\delta} > 0) \log \int \chi_{S,M,\epsilon}(\bx)  \, d\pP^{\otimes N}_X(\bx) \bigg] \notag
			\\&\geq \E_{x^0} \E_{W} (\1_{\cB_\delta^+} - \1_{\cB_\delta})\1(\1_{\cB_\delta^+} - \1_{\cB_\delta} > 0) \langle H^\pert_{N,n}(\bx) \rangle_{\chi}  + \E_{x^0} \bigg[ | \1_{\cB_\delta^+} - \1_{\cB_\delta} |  \log \int \chi_{S,M,\epsilon}(\bx)   \, d\pP^{\otimes N}_X(\bx) \bigg] \label{eq:Errbound1}
		\end{align}
		where we finally used that $\chi_{S,M,\epsilon}$ is non-negative.
		The first term is zero because $\langle \cdot \rangle_{\chi}$ does not depend on the Gaussian terms, so 
		\begin{align*}
			&\E_{x^0} \E_{W} (\1_{\cB_\delta^+} - \1_{\cB_\delta}) \1(\1_{\cB_\delta^+} - \1_{\cB_\delta} > 0) \langle H^\pert_{N,n}(\bx) \rangle_{\chi} 
			\\&= \E_{x^0}  (\1_{\cB_\delta^+} - \1_{\cB_\delta})\1(\1_{\cB_\delta^+} - \1_{\cB_\delta} > 0) \langle \E_{W} H^\pert_{N,n}(\bx) \rangle_{\chi} = 0
		\end{align*}
		because $H^\pert_{N,n}(\bx)$ is a centered Gaussian process. The second term is of order $e^{-c_{\delta}N}$ for some $c_{\delta}>0$ when $\delta>0$
		because $(S,M)$ have finite entropy so that $\log \int \chi_{S,M,\epsilon}(\bx)   \, d\pP^{\otimes N}_X(\bx) $ is at most of order $N$
		whereas 
		\begin{equation}\label{eq:sanovbound}
			\E_{x^0}(\1_{\cB_\delta^+} - \1_{\cB_\delta})  \1(\1_{\cB_\delta^+} - \1_{\cB_\delta} > 0) \leq \E_{x^0}|  \1_{\cB_\delta^+} - 1 + 1- \1_{\cB_\delta} | \leq  \E_{x^0}  \1_{(\cB_\delta^+ )^c} + \E_{x^0} \1_{\cB^c_\delta} = e^{-c_{\delta}N}
		\end{equation}
		by Sanov's theorem. We conclude that \eqref{eq:Errbound2} is lower bounded by a term that tends to zero.
		
		Likewise, on the region where $\1_{\cB_\delta^+} - \1_{\cB_\delta} \leq 0$, we have the lower bound
		\begin{align}
			&\E_{x^0} \E_{W} (\1_{\cB_\delta^+} - \1_{\cB_\delta}) \1(\1_{\cB_\delta^+} - \1_{\cB_\delta} \leq 0) \log \int \chi_{S,M,\epsilon}(\bx)  e^{H^\pert_{N,n}(\bx)} \, d\pP^{\otimes N}_X(\bx) \notag
			\\&\geq \E_{x^0} \E_{W} (\1_{\cB_\delta^+} - \1_{\cB_\delta}) \1(\1_{\cB_\delta^+} - \1_{\cB_\delta} \leq 0) \E \Big[ \sup_{\bx} H^\pert_{N,n}(\bx) \Big] \label{eq:Errbound2}
		\end{align}
		since  $\chi_{S,M,\epsilon}$ is non-negative and bounded by $1$. Since $H_{N,n}^{\pert}(\bx)$ has variance of order $N + n$, we have $\E \Big[ \sup_{\bx} H^\pert_{N,n}(\bx) \Big] = O( (N + n )^{\frac{3}{2}}) $  by Dudley's Theorem. Then \eqref{eq:sanovbound} implies that \eqref{eq:Errbound2} is lower bounded by a term that tends to zero. Hence
		$\mathcal{\mbox{Err}} \geq o_N(1)$.
		\\\\\textit{Decoupling the Constraint on the Empirical Measure:} Finally, we decouple the constraint on the empirical measure in the first term of \eqref{topr}. Recall that
		\[
		\cB_\delta^+ = \{ (\bx^0,\by^0) \mmm d( \hat \pP^{+}_0, \pP_0)  < \delta) \}, \cB_{\delta}^{y}=\{\by^{0}: d(\hat \pP^{y},\pP_{0})<\delta\}, 
		\]
		where
		\[
		\hat \pP^{+}_0 = \frac{N}{N + n} \hat \pP_0 + \frac{n}{N + n} \hat \pP^{y}_0 \quad \text{and} \quad \hat \pP^{+}_0 = \frac{1}{N + n} \sum_{i = 1}^{N + n} \delta_{x_i^0} , \quad \hat \pP^{y}_0 = \frac{1}{n} \sum_{i = N + 1}^{N + n} \delta_{y_i^0} 
		\]
		are the empirical  measures of the $\bx^0_+ \in \R^{N + n}$. It follows that 
		\[
		\{ (\bx^0, \by^0) \mmm d( \hat \pP^{+}_0, \pP_0)  < \delta \} \supset \{\bx^0 \mmm d( \hat \pP_0, \pP_0)  < \delta \} \cap \{\by^0 \mmm d( \hat \pP^{-}_0, \pP_0)  < \delta \}
		\]
		%because by Kantorovich duality, for $(\bx^0,\by^0) \in \{\bx^0 \mmm d( \hat \pP^{N}_0, \pP_0 ) < \delta \} \cap \{\by^0 \mmm d( \hat \pP^{n}_0, \pP_0 ) < \delta \}$
		%\[
		%\sup_{f} \bigg| \int f d \hat \pP^{+}_0 -  \int f d \pP_0 \bigg| \leq \sup_{f_1,f_2} \bigg( \frac{N}{N + n} \bigg| \int f_1 d \hat \pP_0 -  \int f_1 d \pP_0 \bigg|  + \frac{M}{N + n} \bigg| \int f_2 d \hat \pP^{-}_0 -  \int f_2 d \pP_0 \bigg| \bigg) \leq \delta
		%\]
		so that $
		\1_{\cB^{+}_\delta} \geq \1_{\cB_\delta}  \1_{\cB^{y}_\delta}$. 
		Next, for any realization of $\bx^0_+$, we have by Jensen's inequality
		\begin{equation}\label{jens}
			\E_{z} \log \bigg\langle \int \chi_{S,M,\epsilon} (\by) e^{\sum_{i = 1}^n z_i(\hat \bx) y_i  }   \, d \pP^{\otimes n}_X(\by) \bigg\rangle_{N,n}^\pert \geq -nC  \E \bigg\langle \max_{i \leq n} |z_i(\hat \bx)| \bigg \rangle_{N,n}^\pert - c n
		\end{equation}
		because 
		$$\int \chi_{S,M,\epsilon}(\by)d \pP^{\otimes n}_X(\by)\ge e^{-C \cL n} \pP^{\otimes n}_X(|R_{1,1}-S|\le \epsilon)\ge e^{-c n}$$
		for some finite constant $c$ since $S$ has finite entropy. Furthermore, for all $i \leq n$ and $\hat \bx$, $z_i(\hat \bx)$ is a centered Gaussian process with covariance
		\[
		\E z_i(\hat \bx^1)  z_i(\hat \bx^2) = \hat R_{1,2} \leq C^2
		\]
		so by the Cauchy--Schwarz inequality \color{black}
		\[
		\E_{z} \bigg\langle \max_{i \leq n} |z_i(\hat \bx)| \bigg \rangle_{N,n}^\pert  \le \left(\E \bigg\langle \sum_{i \leq n} |z_i(\hat \bx)|^{2} \bigg \rangle_{N,n}^\pert \right)^{1/2}= \left(\E \bigg\langle n \hat R_{1,1}\bigg \rangle_{N,n}^\pert\right)^{1/2}\le\sqrt{n} C
		\]
		Therefore, by \eqref{jens}, there exists a finite constant $L$ such that
		\[
		\E_{z} \log \bigg\langle \int \chi_{S,M,\epsilon} (\by) e^{\sum_{i = 1}^n z_i(\hat \bx) y_i  }   \, d \pP^{\otimes n}_X(\by) \bigg\rangle_{N,n}^\pert \geq
		- Ln^{\frac{3}{2}} 
		\]
		so
		\begin{align*}
			&\E \1_{\cB_\delta^+}\E_{z} \log \bigg\langle \int \chi_{S,M,\epsilon} (\by)   e^{\sum_{i = 1}^n z_i(\hat \bx) y_i  }   \, d \pP^{\otimes n}_X(\by) \bigg\rangle_{N,n}^\pert 
			\\&\geq \E \1_{\cB_\delta}\1_{\cB^y_\delta} \E_{z}\bigg( \log \bigg\langle \int  \chi_{S,M,\epsilon} (\by) e^{\sum_{i = 1}^n z_i(\hat \bx) y_i  }   \, d \pP^{\otimes n}_X(\by) \bigg\rangle_{N,n}^\pert + Ln^{\frac{3}{2}} \bigg) - Ln^{\frac{3}{2}}
			\\&\geq  \E\left[ \1_{\cB_\delta}\1_{\cB^y_\delta}  \log \bigg\langle \int  \chi_{S,M,\epsilon} (\by)   e^{\sum_{i = 1}^n z_i(\hat  \bx) y_i  }   \, d \pP^{\otimes n}_X(\by) \bigg\rangle_{N,n}^\pert - ( \1_{\cB_\delta^c} + \1_{(\cB_\delta^y)^c} ) Ln^{\frac{3}{2}} \right].
		\end{align*}
		Clearly the last two terms go to zero by Sanov's theorem  when $N \to \infty$ and $n \to \infty$ at an exponential rate so that 
		\[
		- \E ( \1_{\cB_\delta^c} + \1_{(\cB_\delta^y)^c} ) Ln^{\frac{3}{2}} = o_n(1). 
		\]
		This completes the proof.

	\end{proof}

	\subsection{The Cavity Computations II} \label{sec:cavityII}
	We now prove the lower bound of the free energy using cavity computations. They key idea stated in the previous section, is that we are able to perturb the Gibbs measure to force the overlap array to satisfy the Ghirlanda--Guerra identities in the limit. This will allow us to characterize the limiting distribution of the overlap arrays and approximate it with an overlap array generated from the Ruelle probability cascades.

	By Gaussian concentration and Weirstrass' Theorem, it follows that the lower bound is a continuous function of the distribution of the overlap arrays \cite[Theorem~1.4]{PBook}. 
	\begin{lem}[Continuity of the Lower Bound with Respect to the Overlaps]\label{lem:continuityarray}
		Let $\langle \cdot \rangle$ be the average with respect to some non-random Gibbs measure $\pG$ on the sphere with radius $\sqrt{S}$ in some Hilbert space $H$. Consider the Gaussian processes $Z(\bs)$ and $Y(\bs)$ indexed by points $\bs$  in $H$ with covariances
		\[
		\E Z(\bs^1) Z(\bs^2) = \langle \bs^1 , \bs^2 \rangle \qquad \E Y(\bs^1) Y(\bs^2) = \frac{\langle \bs^1 , \bs^2 \rangle}{2} 
		\]
		Let $n$ be a fixed integer number and $(S,M)$ with finite entropy $\cI$ so that there exists a finite constant $c$ independent of $n$ and $\epsilon$ such that for $n$ large enough, $\pP_X^{\otimes n}(\Sigma_{\epsilon}(S,M))\ge e^{-cn}$ uniformly for all $\by_0 \in \cB_\delta$. Then the functionals
		\[
		f_n^Z(S,M) = \frac{1}{n} \E_{Z} \log \bigg\langle \int_{|R_{1,1}(\by)-S|\le \epsilon}\chi_{M,\frac{\epsilon}{2}}(R_{1,0}(\by)) e^{\sum_{i = 1}^n Z_i(\bs) y_{i} } \, d \pP_X^{\otimes n}(\by) \bigg\rangle
		\]
		where $Z_i$ are independent copies of $Z$ and
		\[
		f_n^Y = \frac{1}{n} \E_{Z}\log \Big\langle e^{ \sqrt{n} \ba Y(\bs)  } \Big\rangle
		\]
		are continuous functionals of the distribution of the overlap array $(\bx^\ell \cdot \bx^{\ell'})_{\ell,\ell' \geq 1}$ under $ \pG^{\otimes \infty}$ for any $\by^0 \in \cB_\delta$. In particular, for any $\eta>0$ there exists  a finite integer number $K(\eta)$ so that  these functionals can be approximated  by a continuous function of the finite array $(\bx^\ell \cdot \bx^{\ell'})_{1 \leq \ell,\ell' \leq K(\eta)}$ uniformly over all possible choices of Gibbs measures $\pG$ and all $\by^0 \in \cB_\delta$.
	\end{lem}
	
	\begin{proof}
		We focus on $f_n^Z(S,M)$, the case of $f_n^{Y}$ is easier. We define the truncated versions of the following functions
		\[
		f_a(x) = \begin{cases}
			a & \log(x) \geq a\\
			\log(x) & -a < \log(x) < a\\
			-a & \log(x) \leq -a
		\end{cases}
		\]

		and
		\[
		g_{a}(\bx) =	\begin{cases}
			e^{a} &\mbox{ if }\quad  \int\chi_{S,M,\epsilon}(\by)  e^{\sum_{i = 1}^n x_i y_i } \, d \pP_X^{\otimes n}(\by) \ge e^{a},\\
			\int \chi_{S,M,\epsilon}(\by) e^{\sum_{i = 1}^n x_i y_i } \, d \pP_X^{\otimes n}(\by) &\mbox{ if }\quad  \int\chi_{S,M,\epsilon}(\by)  e^{\sum_{i = 1}^n x_i y_i } \, d \pP_X^{\otimes n}(\by) \in [e^{-a}, e^{a}],\\
			e^{-a}&\mbox{ if }\quad  \int\chi_{S,M,\epsilon}(\by)  e^{\sum_{i = 1}^n x_i y_i } \, d \pP_X^{\otimes n}(\by) \le e^{-a}\,.\\
		\end{cases}
		%& -a <  \int_{\Sigma_\epsilon(S,M)} e^{\sum_{i = 1}^n x_i y_i } \, d \pP_X^{\otimes n}(\by) < a\\
		\]
		where we again used $\chi_{S,M,\epsilon}(\by)= 1_{\{|R_{1,1}(\by)-S|\le \epsilon\}} \chi^{n}_{M,\epsilon/2}(R_{1,0}(\by)) $. Furthermore, $g_a(\bx)$ when viewed as a function of $\bx$ and $\by^0$ is uniformly continuous because we $\chi^{n}_{M,\epsilon/2}(R_{1,0}(\by))$ is continuous with respect to $\by^0$ and $\by^0$ takes values on a compact set. 
		
		By standard concentration inequalities, we will show that 
		\begin{equation}\label{approx}
			\Big| f_n^Z(S,M) -  \frac{1}{n} \E_{Z} f_a\langle g_{a}(Z_1(\bs), \dots, Z_n(\bs)) \rangle  \Big| = o(a)
		\end{equation}
		where the error tends to $0$ as $a \to \infty$.  Note that for any fixed $a$, $f_{a}$ is a bounded continuous function and therefore on $[e^{-a}, e^{a}]$ we can approximate it uniformly  by a polynomial $f_{\eta}$ of degree $K(a,\eta)$ up to an error $\eta$. We hence see that 
		
		$$  \frac{1}{n} \E_{Z } f_a\langle g_{a}(Z_1(\bs), \dots, Z_n(\bs)) \rangle=  \frac{1}{n} \E_{Z } f_{\eta}(\langle g_{a}(Z_1(\bs), \dots, Z_n(\bs)) \rangle)+\eta\,.$$
		We next notice that for any integer number $r$, 
		$$\E_{Z}[ \langle g_{a}(Z_1(\bs), \dots, Z_n(\bs)) \rangle^{r}]=\pG^{\otimes r}[F_{r} ( (\bx^\ell \cdot \bx^{\ell'})_{1\le \ell,\ell' \leq r})]$$
		where $F_{r}$ is continuous since $\prod_{1\le i\le r}  g_{a}(Z_1(\bs_{i}), \dots, Z_n(\bs_{i}))$ is a bounded  continuous function of the $Z_{i}(\bs_{j})$ and the convergence of the covariance of a Gaussian process implies its weak convergence. 
		Hence, up to an error $\eta$, $   \frac{1}{n} \E_{Z } f_a\langle g_{a}(Z_1(\bs), \dots, Z_n(\bs)) \rangle$ is a continuous function of the overlap array $ (\bx^\ell \cdot \bx^{\ell'})_{1\le \ell,\ell' \leq K(\eta,a)}$. 
		We thus only need to prove \eqref{approx}. 
		Clearly we have
		\begin{eqnarray}
			&& \Big|f_n^Z(S,M) -   \E_{Z} f_a\langle g_{a}(Z_1(\bs), \dots, Z_n(\bs)) \rangle  \Big| \\
			&& \le 
			\Big|  \E_{Z} (f_a\langle g_{a}(Z_1(\bs), \dots, Z_n(\bs))\rangle- f_{a}\langle g_{\infty}(Z_1(\bs), \dots, Z_n(\bs))\rangle) \Big|\nonumber\\
			&&+\Big|  \E_{Z} \log \langle g_{\infty}(Z_1(\bs), \dots, Z_n(\bs))\rangle-  \E_{Z} f_{a}\langle g_{\infty}(Z_1(\bs), \dots, Z_n(\bs))\rangle\Big| \label{qazx}\end{eqnarray}
		To bound the first term, we notice that $f_{a}$ is Lipschitz with constant $e^{a}$ so that
		\begin{eqnarray}
			&&\Big|  \E_{Z} (f_a\langle g_{a}(Z_1(\bs), \dots, Z_n(\bs))\rangle- f_{a}\langle g_{\infty}(Z_1(\bs), \dots, Z_n(\bs))\rangle) \Big|\nonumber\\
			&\le &e^{a}\E_{Z}|\langle (g_{a}-g_{\infty})(Z_1(\bs), \dots, Z_n(\bs))\rangle|\nonumber
			\\
			&\le& e^{a}\E_{Z[}[
			\langle (g_{\infty}1_{|\log(g_{\infty})|\ge a}(Z_1(\bs), \dots, Z_n(\bs))\rangle]\nonumber\\
			&\le&  e^{a-am}\E_{Z}[\langle ( g_{\infty }(Z_1(\bs), \dots, Z_n(\bs))^{m+1}+ g_{\infty}(Z_1(\bs), \dots, Z_n(\bs))^{1-m}\rangle]\label{qaz}\end{eqnarray}
		where we finally used Chebychev's inequality.
		We finally remark that because the $y_{i}$'s are bounded by $C$ and $(S,M)$ have finite entropy, 
		\begin{equation}\label{qas} e^{-C\sum |Z_{i}(\bs)|-cn }\le   g_{\infty}(Z_1(\bs), \dots, Z_n(\bs))\le e^{C\sum|Z_{i}(\bs)| + cn}\end{equation}
		Moreover, since the covariances of the $Z_{i}$ are bounded uniformly by $S$, Gaussian concentration, see e.g  \cite[Theorem 1.2]{PBook}, implies that for each $i$
		$$P( |Z_{i}(\bs)|\ge a)\le 2\exp\{-a^{2}/4S\}$$
		so that
		\begin{equation}\label{qasd}P(\sum_{i=1}^{n}|Z_{i}(\bs)|\ge a)\le n\max_{i}P( |Z_{i}(\bs)|\ge a/n)\
			\le
			2n\exp\{-a^{2}/4Sn^{2}\}\end{equation} which implies that for any $L\ge 0$
		$$\E[e^{L\sum |Z_{i}(\bs)|}]\le  \bigg(1+\frac{2n}{L} \bigg)e^{4L^{2}Sn^{2}}\,.$$
		Plugging this estimate into \eqref{qaz} implies that there exists a finite constant $c(n)$ depending on $n$ such that
		\begin{equation}\label{okm}
			\Big|  \E_{Z} [f_a(\langle g_{a}(Z_1(\bs), \dots, Z_n(\bs))\rangle)- f_{a}(\langle g_{\infty}(Z_1(\bs), \dots, Z_n(\bs))\rangle)] \Big|\le C(n) e^{a-am}e^{4(m+1)^{2}Sn^{2}}\end{equation}
		which goes to zero as $a$ goes to infinity if $m$ is chosen greater than one. The argument to bound the second term of \eqref{qazx} is similar since the difference vanishes unless $g_{\infty}$ is too big or too small, whose probability we have just estimated above by \eqref{qas} and \eqref{qasd}. These bounds hold uniformly in $\by \in \cB_\delta$ so our proof is complete. 
		
	\end{proof}

	We next use Lemma \ref{lem:continuityarray} to show that the lower bound on $\Delta_{n,N}(\Sigma_{\epsilon}(S,M))$ obtained in Proposition \ref{Aiz} converges.  We first study the large $N$ limit point of the free energy:
	$$F_{N,n}^{1}(\epsilon,\delta: S,M):
	=\frac{1}{n}  \E \1_{\bx^{0}\in \cB_\delta} \1_{\by^{0}\in \cB_\delta} \log \bigg\langle \int_{|R_{1,1}(\by)-S|\le \epsilon}\chi^{N}_{M,\epsilon}(R_{1,0}(\by))  e^{\sum_{i = 1}^n z_i(\hat \bx) y_i}   \, d \pP^{\otimes n}_X(\by) \bigg\rangle_{N,n}^\pert $$
	We will therefore use Lemma \ref{lem:continuityarray} with $\pG$ the  perturbed Gibbs measure $\pG^\pert_{N,n}$ with Hamiltonian $H^{\pert}_{N,n}(\hat x)$ and smooth conditioning by $\chi_{S,M,\epsilon}$.
	We define the following overlap array
	\[
	(R^N_{\ell,\ell'})_{\ell,\ell' \geq 1} = \Big( \frac{\hat \bx^\ell \cdot \hat \bx^{\ell '}}{N} \Big)_{\ell,\ell' \geq 1} \text{ where }R^N_{\ell,\ell} =S \text{ for all $\ell \geq 1$}. 
	\]
	The overlap array $(R^N_{\ell,\ell'})_{\ell,\ell' \geq 1}$ has bounded entries. Moreover, we have seen in Lemma  \ref{lem:continuityarray} that up to a small error $\eta$, $F_{N,n}^{1}(\epsilon,\delta: S,M)$ is a continuous function  of finitely many overlaps $(R^N_{\ell,\ell'})_{1\le \ell,\ell' \leq K(\eta)}$ (uniformly on the Gibbs measures $\pG$). 
	The space of $K(\eta)\times K(\eta)$  arrays with bounded entries is compact, so the space of probability measures on such arrays are tight. The selection theorem implies that the distribution of $(R^N_{\ell,\ell'})_{1\le \ell,\ell' \leq K(\eta)}$ converges along a subsequence to a limiting array $(R^{\infty,\epsilon}_{\ell,\ell'})_{1\le \ell,\ell' \leq K(\eta)}$. Of course this limit point depends on $\epsilon$ as well. 
	
	Next, we can take $\epsilon \to 0$, and the finite array also converges in distribution to an array $(R^{\infty}_{\ell,\ell'})_{1\le \ell,\ell' \leq K(\eta)}$ along a subsequence again by tightness.  This array can in fact be thought as infinite if we consider projective limits. 
	Furthermore, $R^\infty_{\ell,\ell} = S$ for all $\ell \geq 1$. By construction, the subarray $(R^\infty_{\ell,\ell'})_{\ell \neq \ell' \geq  1}$ also satisfies the Ghirlanda--Guerra identities, so we can characterize the limiting distribution of this array as usual \cite[Chapter 3]{PBook}.
	
	Since $R^\infty$ satisfies the Ghirlanda--Guerra identities, the distribution of the entire array is determined by $\zeta(t) = \pP(R^\infty_{1,2} \leq t)$ \cite[Theorem~2.13 and Theorem~2.17]{PBook}. We can approximate $\zeta(t)$ in $L^1$ with a piecewise constant function $\mu(t)$, so that
	\[
	\int | \zeta(t) - \mu(t) | \, dt < \epsilon.
	\]
	The density function $\mu$ of a measure can be encoded by the parameters
	\begin{equation}\label{eq:zetaseqlwbd}
		\zeta_{-1} = 0 < \zeta_0< \dots < \zeta_{r-1}
	\end{equation}
	and sequence
	\begin{equation}
		0 = Q_0 \leq Q_1 \leq \dots \leq Q_{r-1} \leq Q_r =  S.
	\end{equation}
	That is, these sequences define the density function
	\[
	\mu(Q) = \zeta_k \qquad \text{for} \qquad Q_k \leq Q < Q_{k + 1}.
	\]
	Let $v_\alpha$ denote the weights of the Ruelle probability cascades corresponding to the sequence \eqref{eq:zetaseqlwbd}. If $(\alpha^\ell)_{\ell \geq 1}$ are samples from the Ruelle probability cascades, then $\pP( \alpha^1 \wedge \alpha^2 \leq t) = \mu(t)$ by construction. This gives us an explicit way to construct the off-diagonal entries of the overlap array in the limit.  We define Gaussian processes $Z(\alpha)$ and $Y(\alpha)$ with covariance
	\[
	\E Z(\alpha^1) Z(\alpha^2) = Q_{\alpha^1 \wedge \alpha^2} \quad \E Y(\alpha^1)  Y(\alpha^2)  = \frac{1}{2} Q^2_{\alpha^1 \wedge \alpha^2}
	\]
	and let $Z_i$ for $1 \leq i \leq n$ denote independent copies of $Z$. The functionals
	\[
	f_n^Z(\mu) = \frac{1}{n} \E \log \sum_\alpha v_\alpha \int_{|R_{1,1}(\by)-S|\le \epsilon}\chi^{N}_{M,\epsilon}(R_{1,0}(\by)) e^{\sum_{i \leq n} \ba Z_i(\alpha) y_i   } \, d \pP_X^{\otimes n} (\by)
	\]
	and
	\[
	f_n^Y(\mu) = \frac{1}{n} \E \log \sum_\alpha v_\alpha  e^{ \sqrt{n} \ba Y(\alpha)   }
	\]
	are of the same form as the functionals in Lemma~\ref{lem:continuityarray} because they depend on the overlap array in exactly the same way. Furthermore, one can show that they are Lipschitz continuous \cite[Lemma~4.1]{PBook}.
	\begin{lem}[Continuity of the Cavity Functionals]\label{lem:contfuntionals}
		For any $S,M$, there exists a  finite constant $L$ (that may depend on $S$ and $M$) such that for any measurable increasing functions $\mu_1$ and $\mu_2$ from $[0,S]$ to $[0,1]$ so that $\mu_1(\infty) = \mu_2(\infty) = 1$,
		\[
		|f_n^Z(\mu_1) - f_n^Z(\mu_2)| \leq L \int | \mu_1(x) - \mu_2(x) | \, dx
		\]
		and
		\[
		|f_n^Y(\mu_1) - f_n^Y(\mu_2)| \leq L \int | \mu_1(x) - \mu_2(x) | \, dx.
		\]
	\end{lem}
	\begin{proof}
		We prove the statement for $f_n^Z(\mu)$. Let $\mu_1$ and $\mu_2$ encode two discrete density functions on $[0,S]$ encoded by sequences $(\zeta_k)$, $(Q^1_k)$ and $(Q^2_k)$. Notice that by repeating points, we could assume that the $(\zeta_k)$ sequences are common for both measures and that the $k$ in both sequences are identical. Define the interpolating measure $\mu^{-1}_t =  t \mu^{-1}_1 + (1 -t ) \mu^{-1}_2$ where $\mu^{-1}$ denotes the quantile transform of $\mu$ : $\mu^{-1}((-\infty,t])=\int_{\mu(x)\le t}dx$. Associated with this interpolating measure $\mu_t^{-1}$ is a sequence of parameters $(\zeta_k)_{-1 \leq k \leq r}$ and $(Q_k^t)_{0 \leq k \leq r}$ where $Q_k^t = tQ_k^{1} + (1-t)Q_k^{2}$. We define the interpolating process
		\[
		Z_i(\alpha;t) = \sum_{i \leq n} \ba Z^{\mu_t}_i(\alpha) y_i  
		\]
		where $Z^{\mu_t}$ is the Gaussian process defined with respect to the measure $\mu_t$, which can be expressed in the form
		\[
		Z_i(\alpha;t) = \sum_{k = 1}^{r - 1} (Q_k^t - Q^t_{k - 1})^{1/2} z^i_{\alpha_{|k}}
		\]
		for i.i.d.  $z^i_{\alpha_{|k}}$. We define the free energy
		\[
		\phi(t) := \frac{1}{n} \E \log \sum_\alpha v_\alpha \int_{|R_{1,1}(\by)-S|\le \epsilon}\chi^{N}_{M,\epsilon}(R_{1,0}(\by)) e^{\sum_{i \leq n} \ba Z_i(\alpha;t) y_i   } \, d \pP_X^{\otimes n} (\by).
		\]
		By an integration by parts, it follows that
		\begin{align*}
			\phi'(t) &= \frac{1}{n}\sum_{i = 1}^n \E \left\langle  \ba \partial_t Z_i(\alpha;t) y_i \right\rangle_t 
			\\&= \frac{\ba^2}{2n}\sum_{i = 1}^n \sum_{k = 1}^{r - 1} \E \left\langle   \frac{  y_iz^i_{\alpha_{|k} } }{ (Q_k^t - Q^t_{k - 1})^{1/2} } - \frac{ y_iz^i_{\alpha_{|k + 1}} }{ (Q_{k + 1}^t - Q^t_{k})^{1/2} }   \right\rangle_t (Q_k^1 - Q_k^2) 
			\\&= \frac{\ba^2}{2n}\sum_{i = 1}^n \sum_{k = 1}^{r - 1} \E \left\langle  y^1_iy_i^2 \1(\alpha^1 \wedge \alpha^2 \geq k) - y^1_iy_i^2 \1(\alpha^1 \wedge \alpha^2 \geq k + 1)  \right\rangle_t (Q_k^1 - Q_k^2)
			\\&= - \frac{\ba^2}{2n}\sum_{i = 1}^n \sum_{k = 1}^{r- 1} \E \langle y_i^1 y_i^2 \1(\alpha^1 \wedge \alpha^2 = k) \rangle_t (Q_k^1 - Q_k^2).
		\end{align*}
		Therefore, recalling that $\E \langle \1(\alpha^1 \wedge \alpha^2 = k) \rangle_t = \zeta_k - \zeta_{k-1}$
		\[
		|\phi'(t)| \leq \ba^2 S C^2 \sum_{k = 1}^{r- 1}  (\zeta_k - \zeta_{k-1}) |Q_k^1 - Q_k^2| =\ba^2 S C^2 \int | \mu_1(x) - \mu_2(x) | \, dx
		\]
		and our result follows from the fact that $\phi(1) = f_n^Z(\mu_1)$ and $\phi(0)= f_n^Z(\mu_2)$.
		
		The continuity of $f_n^Y$ is trivial because 
		\begin{equation}
			\frac{1}{N} \E \log \sum_\alpha v_\alpha  e^{ \sqrt{N}\ba Y(\alpha) } = \frac{\ba^2}{4 } \sum_{k = 0}^{r - 1} \zeta_k ( Q^2_{k + 1} - Q^2_k )  = \frac{\beta^2}{2} \int x \mu(x) \, dx 
		\end{equation}
		by \eqref{eq:ycomputation}. 
	\end{proof}
	
	As a consequence of Proposition \ref{Aiz}, and \eqref{lb2}, and the fact that $\Delta_{N,n}(\beta,\epsilon,\delta:S,M)$ are continuous functions of the overlaps which limit points are described, according to the Guirlenda-Guerra identities, by $\zeta$, we deduce that
	for each $n\ge 1$ 
	
	\begin{equation}\label{lb}
		\liminf_{N\rightarrow \infty}\frac{1}{N}  \E 1_{\cB_{\delta}} \log \tilde Z_{N}^{\pert}(\beta,\epsilon,\delta:S,M)\ge  \inf_{\zeta}( f_n^Z(\zeta) - f_n^Y(\zeta))+o_{n}(1)\end{equation}
	where $f_n^Z$ and $f_n^Y$ have been continuously extended to be defined with respect to all c.d.f.s instead of discrete ones.  By continuity of $f_{n}^{Z}-f_{n}^{Y}$ and compactness, this infimum is achieved. 
 
 In the following computations, it will be convenient to work with the original indicator function instead of its smooth approximation since it matches the form computed in Section~\ref{sec:upbd}. We can use the fact that $ \chi_{M,\epsilon} (R_{1,0}) \geq  \1_{(M-\epsilon, M+\epsilon)} (R_{1,0})$ to conclude that $ \chi_{S,M,\epsilon} (\by) 
		\geq \1_{\by\in \Sigma_{\epsilon}(S,M)}$. This holds pointwise for all $\zeta$ so we conclude that the free energy is lower bounded by
	\begin{equation}
		\liminf_{N\rightarrow \infty}\frac{1}{N}  \E 1_{\cB_{\delta}} \log \tilde Z_{N}^{\pert}(\beta,\epsilon,\delta:S,M)\ge  \inf_{\zeta}( \tilde f_n^Z(\zeta) - f_n^Y(\zeta))+o_{n}(1)\end{equation}
    where we defined
    \[
    \tilde f_n^Z(\mu) = \frac{1}{n} \E \log \sum_\alpha v_\alpha \int_{\Sigma_\epsilon(S,M)} e^{\sum_{i \leq n} \ba Z_i(\alpha) y_i   } \, d \pP_X^{\otimes n} (\by).
    \]

	To compute the integral explicitly, a crucial step is the removal of the constraint $\Sigma_{\epsilon}(S,M)$ on the self overlaps. Similar constrained integrals appears in the the lower bound, so we must prove that the constrained integrals and unconstrained integrals are identical in the limit for optimal choices of $\mu$ and $\lambda$. This can be done via a large deviations argument. We first state a property of the Ruelle probability cascades that will allow us to upper bound a partition of the free energy.
	
	\begin{lem}[Upper Bound of the Ruelle Probability Cascades]\label{lem:upbdRPC} 
		Let $g(\alpha)$ be a Gaussian process indexed by $\alpha \in \N^r$ with covariance	\[
		\E g(\alpha^1) g(\alpha^2) = C( Q_{\alpha^1 \wedge \alpha^2} )
		\]
		independent of $v_\alpha$. If $A_j: \R \to \R$ are positive functions of the same Gaussian process $g(\alpha)$ for $1 \leq j \leq n$ then
		\[
		\E \log \sum_{\alpha \in \N^r} v_\alpha \sum_{j \leq n} A_j(g(\alpha)) \leq \frac{\log n}{\zeta_0} + \max_{j\leq n} \E \log \sum_{\alpha \in \N^r} v_\alpha A_j(g(\alpha)),
		\]
		where $\zeta_0 > 0$ is the smallest point in the sequence \eqref{eq:zetaseq}.
	\end{lem}
	
	\begin{proof}
		The proof can be found in \cite[Lemma~6]{PPotts}. We restate it here for convenience. For
		\begin{equation}\label{eq:recur1}
			X_r = \log \sum_{j \leq n} A_j\bigg( \sum_{k = 1}^r ( C( Q_k ) - C( Q_{k - 1} )  )^{1/2} z_i \bigg) \qquad X_{p} = \frac{1}{\zeta_p} \log \E_{z_{kp+ 1}} e^{\zeta_p X_{p + 1}} \quad \text{for $0 \leq p \leq r-1$},
		\end{equation}
		and let $X_{p,j}$ such that $e^{X_r}=\sum_{j\le n} e^{X_{r,j}}$ be given by 
		\begin{equation}\label{eq:recur2}
			X_{r,j} = \log A_j\bigg(  \sum_{k = 1}^r ( C( Q_k ) - C( Q_{k - 1} )  )^{1/2} z_i \bigg) \qquad X_{p,j} = \frac{1}{\zeta_p} \log \E_{z_{p + 1}} e^{\zeta_p X_{p + 1,j}} \quad \text{for $0 \leq p \leq r-1$},
		\end{equation}
		Lemma~\ref{lem:RPCavg} implies
		\begin{equation}\label{eq:recur3}
			\E \log \sum_{\alpha \in \N^r} v_\alpha \sum_{j \leq n} A_j(\alpha) = X_0 \qquad \E \log \sum_{\alpha \in \N^r} v_\alpha  A_j(\alpha) = X_{0,j}.
		\end{equation}
		Using the recursive definition \eqref{eq:recur1} and \eqref{eq:recur2}, since $\zeta_{r-1} < 1$, Jensen's inequality implies 
		\begin{align*}
			\exp \zeta_{r-1} X_{r-1} &= \E_{z_{r}} \exp( \zeta_{r-1 }X_r )= \E_{z_{r}} \bigg( \sum_{j \leq n} \exp( X_{r,j} ) \bigg)^{\zeta_{r-1}}
			\\&\leq  \left(  \sum_{j \leq n}  \E_{z_{r}} \exp( X_{r,j} )\right)^{\zeta_{r-1}}
			\leq  \sum_{j \leq n}  \E_{z_{r}} \exp( \zeta_{r-1} X_{r,j} )
			\\&=  \sum_{j \leq n} \exp(\zeta_{r-1} X_{r-1,j} ).
		\end{align*}
		where we finally used that $f(x)=x^{\zeta_{r-1}}$ is concave and nonnegative,  therefore sub-additive on $\mathbb R^{+}$ (to see this use concavity to show that $\frac{a}{a+b} f(a+b)\le \frac{a}{a+b}f(a+b)+\frac{b}{a+b} f(0)\le f(a)$).
		Similarly, we can iterate this bound recursively using the fact that $\zeta_p/\zeta_{p + 1} < 1$ to conclude that
		\begin{align*}
			\exp \zeta_{p} X_{p} &= \E_{z_{p + 1}} \exp( \zeta_{p}X_{p+1} )= \E_{z_{p + 1}} \bigg( \sum_{j \leq n} \exp( \zeta_{p+1}X_{p + 1,j} ) \bigg)^{\frac{\zeta_p}{\zeta_{p+1}}}
			\\&\leq  \sum_{j \leq n} \E_{z_{p + 1}}  \exp(\zeta_{p} X_{p + 1,j} )
			= \sum_{j \leq n}  \exp(\zeta_{p} X_{p,j} ).
		\end{align*}
		This allows to show   when $p=0$ that
		\[
		X_0 \leq \frac{1}{\zeta_0} \log \sum_{j \leq n}  \exp(\zeta_{0} X_{0,j} ) \leq \frac{\log(n)}{\zeta_0} + \max_{j \leq n} X_{0,j},
		\]
		so applying \eqref{eq:recur3} proves our statement.	
	\end{proof}
	
	Lemma~\ref{lem:upbdRPC} will be used to upper bound the unconstrained free energy after decomposing it as the sum of contrained free energies. We can now prove that the constrained free energy is asymptotically sharp after minimizing over $\mu$ and $\lambda$.
	
	\begin{lem}[Sharp Lower Bound] \label{lem:sharpupbd} 
		For $S ,M \in \cC$ and any $\epsilon,\delta > 0$ small enough, 
		\begin{align}
			&\liminf_{N \to \infty} \frac{1}{N} \E_{Z,x^{0}} \1_{\cB_\delta} \log \sum_\alpha v_\alpha \int_{\Sigma_\epsilon(S,M)} e^{\sum_{i \leq N} \ba Z_i(\alpha) x_i } \, d \pP_X^{\otimes N} (\bx)
			\nonumber\\&\geq \inf_{\mu,\lambda} \bigg( -\lambda S - \mu M + \E_{Z,x^{0}} \log \sum_\alpha v_\alpha  \int e^{\ba Z(\alpha) x +\lambda x^{2} +\mu x x^{0}
			} \, d \pP_X (\bx)  \bigg).\label{lbge}
		\end{align} Moreover, the right hand side is equal to $-\infty$ if $\cI(S,M)=\infty$.		Furthermore, if $S,M$ belong to the interior of $\cC$, then the minimizer is attained at a unique $\mu$ and $\lambda$, such that $|\mu| + |\lambda| \leq C(S,M)$ where the constant $C$ only depends on the distance from $(S,M)$ to the boundary. 
		
	\end{lem}

	\begin{proof}A similar result is proved in \cite[Section 7]{PVS}. The proof of \cite[Section 7]{PVS} could be adapted easily for $S,M$ in the interior of the set $\cC$, but would require additional arguments for elements of the boundary of $\cC$ in which case the infimum over $(\lambda,\mu)$ which may be attained at infinity. We therefore follow another route which mimick the proof of Gartner-Ellis theorem \cite[Theorem 2.3.6]{DZ}, taking into account the random density depending on the $Z_{i}$'s.

		We first show that we can restrict ourselves to $(S,M)$ with finite entropy because the lower bound in\eqref{lbge} is infinite otherwise. 
		Indeed,
		$$\E_{Z,x^{0}} \log \sum_\alpha v_\alpha  \int e^{\ba Z(\alpha) x +\lambda x^{2} +\mu x x^{0}
		} \, d \pP_X (\bx) \le \E_{Z} \log \sum_\alpha v_\alpha  \int e^{\ba |Z(\alpha)|C} +\E_{x^{0}}\log \int e^{\lambda x^{2} +\mu x x^{0}
		} \, d \pP_X (\bx)$$ and 
		$
		\E\log \sum_\alpha v_\alpha e^{- \ba |Z(\alpha)| C}
		$
		is bounded uniformly by Lemma~\ref{lem:RPCavg} because
		\[
		\E e^{ -| \sum_{k = 1}^r (  Q^2_k - Q^2_{k - 1}  )^{1/2} z_{ k} |  C   } < \infty
		\]
		using the moment generating function for folded normals. Therefore there exists a finite constant $L$ such that
		$$ \inf_{\mu,\lambda} \bigg( -\lambda S - \mu M + \E_{Z,x^{0}} \log \sum_\alpha v_\alpha  \int e^{\ba Z(\alpha) x +\lambda x^{2} +\mu x x^{0}
		} \, d \pP_X (\bx)  \bigg)\le -\cI(S,M)+ L.$$
		We hence  can restrict ourselves to $(S,M)$ with finite entropy. 
		We then notice that 
		$$\liminf_{N \to \infty} \frac{1}{N} \E_{Z,x^{0}} \1_{\cB_\delta} \log \sum_\alpha v_\alpha \int_{\Sigma_\epsilon(S,M)} e^{\sum_{i \leq N} \ba Z_i(\alpha) x_i } \, d \pP_X^{\otimes N} (\bx)$$
		is decreasing in $\epsilon$ and does not depend on $\delta$ because $\cB_{\delta}^{c}$ has an exponentially small probability so that if $\delta'<\delta$, we have
		\begin{eqnarray*}
			&& \frac{1}{N} \E_{Z,x^{0}} \1_{\cB_\delta\backslash \cB_{\delta'}} \log \sum_\alpha v_\alpha \int_{\Sigma_\epsilon(S,M)} e^{\sum_{i \leq N} \ba Z_i(\alpha) x_i } \, d \pP_X^{\otimes N} (\bx)\\
			&&\ge \frac{1}{N} \E_{Z,x^{0}} \1_{\cB_\delta\backslash \cB_{\delta'}} \E_{Z}\log \sum_\alpha v_\alpha e^{- C\ba \sum_{i \leq N}  |Z_i(\alpha)| } \pP_X^{\otimes N} (\Sigma_\epsilon(S,M))\\
			&&\geq  \pP_{0}^{\otimes N}({\cB_\delta\backslash \cB_{\delta'}} )(- C\ba \sum_{\alpha }v_{\alpha}   \E|Z(\alpha)|  +\inf_{\bx_{0}\in \cB_{\delta}} \frac{1}{N}\log \pP_X^{\otimes N} (\Sigma_\epsilon(S,M)))\end{eqnarray*}
		which goes to zero as $N$ goes to infinity since $(S,M)$ has finite entropy so that the last term is finite.
		We next adapt Gartner-Ellis argument to our setting. It is based on a large deviation upper bound for the tilted measures. Namely let $\lambda,\mu\in\mathbb R^{2}$. We first show that for every $(S,M)\in [0,C^{2}]\times [-C^{2},C^{2}]$,
		
		\begin{eqnarray}\label{ldubt0}
			&&\limsup_{N \to \infty} \frac{1}{N} \E_{Z,x^{0}} \1_{\cB_\delta} \log\frac{ \sum_\alpha v_\alpha \int_{\Sigma_\epsilon(S,M)} e^{\sum_{i \leq N} (\ba Z_i(\alpha) x_i +\lambda x_{i}^{2}+\mu x_{i}x_{i}^{0}) } \, d \pP_X^{\otimes N} (\bx)}{ \sum_\alpha v_\alpha \int  e^{\sum_{i \leq N} (\ba Z_i(\alpha) x_i +\lambda x_{i}^{2}+\mu x_{i}x_{i}^{0}) } \, d \pP_X^{\otimes N} (\bx)}\nonumber\\
			&&\qquad
			\le -\Lambda^{*}_{\lambda,\mu}(S,M)+O(\epsilon)+O(\delta)\end{eqnarray}
		with 
		$$\Lambda^{*}_{\lambda,\mu}(S,M)=-\lambda S-\mu M +\Lambda(\mu,\lambda)+\sup_{\lambda',\mu'}\{\lambda'S+\mu' M-\Lambda(\lambda',\mu')\} $$
		where
		$$\Lambda(\lambda,\mu)= \E_{Z,x^{0}} \log \sum_\alpha v_\alpha  \int e^{\ba Z(\alpha) x +\lambda x^{2} +\mu x x^{0}
		} \, d \pP_X (\bx)\,.$$
		We denote in short $\Lambda^{*}=\Lambda^{*}_{0,0}$.
		Indeed, \eqref{ldubt0} is a direct consequence of the fact that the $v_{\alpha}$ are non negative and  almost surely we have
		$$\int_{\Sigma_\epsilon(S,M)} e^{\sum_{i \leq N} (\ba Z_i(\alpha) x_i +\lambda x_{i}^{2}+\mu x_{i}x_{i}^{0}) } \, d \pP_X^{\otimes N} (\bx)\qquad\qquad\qquad$$
		$$\qquad\qquad\qquad  \le e^{N(\lambda-\lambda') S+N(\mu-\mu') M+NO(\epsilon)}\int  e^{\sum_{i \leq N} (\ba Z_i(\alpha) x_i +\lambda' x_{i}^{2}+\mu' x_{i}x_{i}^{0}) } \, d \pP_X^{\otimes N} (\bx)$$
		We next introduced the notion of exposed points: $(S,M)$ is said to be exposed if and only if there exists $(\lambda,\mu)$ such that for every $(S',M')\neq (S,M)$ we have 
		\begin{equation}\label{defcrit} 
		\lambda S+\mu M-\Lambda^{*}(S,M)>\lambda S'+\mu M'-\Lambda^{*}(S',M')=-\Lambda^{*}_{\lambda,\mu}(S',M')+\Lambda(0,0) \,.\end{equation}
		The set $(\lambda,\mu)$ is called an exposing hyperplane.
		We first prove \eqref{ldubt0} for an exposed point $(S,M)$ with exposing hyperplane $(\lambda,\mu)$ by showing that the associated tilted measure puts some mass on a neighborhood of $(S,M)$, see \eqref{tot}. To see this, 
		we first claim that for every $(S',M')\neq (S,M)$, 
		\begin{eqnarray*}
			\Lambda^{*}_{\lambda,\mu}(S',M')&=&\Lambda^{*}(S',M')-(\lambda S'+\mu M'-\Lambda(\mu,\lambda)+\Lambda(0,0) )\\
			&\ge &\Lambda^{*}(S',M')-(\lambda (S'-S)+\mu (M'-M)+\Lambda^{*}(S,M))
			>0\end{eqnarray*}
		Moreover, it is easy to see that $\Lambda^{*}_{\lambda,\mu}$ is a good rate function so that it achieves its minimum value on the closure $\bar\Sigma_{\epsilon}(S,M)^{c}$  of $\Sigma_{\epsilon}(S,M)^{c}$, hence 
		 $\inf_{\bar\Sigma_{\epsilon}(S,M)^{c}}\Lambda^{*}_{\lambda,\mu}\ge \kappa>0$. Moreover, 
		 we can cover  $\bar\Sigma_{\epsilon}(S,M)^{c}$   by a  union of  finitely many balls  $(B_{j})_{j\le K}$  so that for each $j\le K$
		
		\begin{eqnarray}\label{ldubt}
			&&\limsup_{N \to \infty} \frac{1}{N} \E_{Z,x^{0}} \1_{\cB_\delta} \log\frac{ \sum_\alpha v_\alpha \int_{(R_{1,1},R_{1,0})\in B_{j}} e^{\sum_{i \leq N} (\ba Z_i(\alpha) x_i +\lambda x_{i}^{2}+\mu x_{i}x_{i}^{0}) } \, d \pP_X^{\otimes N} (\bx)}{ \sum_\alpha v_\alpha \int  e^{\sum_{i \leq N} (\ba Z_i(\alpha) x_i +\lambda x_{i}^{2}+\mu x_{i}x_{i}^{0}) } \, d \pP_X^{\otimes N} (\bx)}\nonumber\\
			&&\qquad
			\le -\kappa+O(\delta).\end{eqnarray}
		We therefore deduce there exists $\kappa=\kappa_{\epsilon}>0$ such that
		\begin{eqnarray}\label{ldubt}
			&&\limsup_{N \to \infty} \frac{1}{N} \E_{Z,x^{0}} \1_{\cB_\delta} \log\frac{ \sum_\alpha v_\alpha \int_{\bar \Sigma_\epsilon(S,M)^{c}} e^{\sum_{i \leq N} (\ba Z_i(\alpha) x_i +\lambda x_{i}^{2}+\mu x_{i}x_{i}^{0}) } \, d \pP_X^{\otimes N} (\bx)}{ \sum_\alpha v_\alpha \int  e^{\sum_{i \leq N} (\ba Z_i(\alpha) x_i +\lambda x_{i}^{2}+\mu x_{i}x_{i}^{0}) } \, d \pP_X^{\otimes N} (\bx)}\nonumber\\
			&&\le \limsup_{N \to \infty} \frac{1}{N} \E_{Z,x^{0}} \1_{\cB_\delta} \log \sum_{j\le K} \frac{ \sum_\alpha v_\alpha \int_{B_{j}} e^{\sum_{i \leq N} (\ba Z_i(\alpha) x_i +\lambda x_{i}^{2}+\mu x_{i}x_{i}^{0}) } \, d \pP_X^{\otimes N} (\bx)}{ \sum_\alpha v_\alpha \int  e^{\sum_{i \leq N} (\ba Z_i(\alpha) x_i +\lambda x_{i}^{2}+\mu x_{i}x_{i}^{0}) } \, d \pP_X^{\otimes N} (\bx)}\nonumber\\
			&&\le
			-\kappa+O(\delta)\end{eqnarray}
		where we finally used Lemma \ref{lem:upbdRPC} to pull the sum outside of the logarithm.  Applying again Lemma \ref{lem:upbdRPC}, we conclude that
		\begin{eqnarray*}%\label{ldubt2}
			0&=&\liminf_{N \to \infty} \frac{1}{N} \E_{Z,x^{0}} \1_{\cB_\delta} \log\frac{ \sum_\alpha v_\alpha \int e^{\sum_{i \leq N} (\ba Z_i(\alpha) x_i +\lambda x_{i}^{2}+\mu x_{i}x_{i}^{0}) } \, d \pP_X^{\otimes N} (\bx)}{ \sum_\alpha v_\alpha \int  e^{\sum_{i \leq N} (\ba Z_i(\alpha) x_i +\lambda x_{i}^{2}+\mu x_{i}x_{i}^{0}) } \, d \pP_X^{\otimes N} (\bx)}\\
			&\le&\max\left\{\liminf_{N \to \infty} \frac{1}{N} \E_{Z,x^{0}} \1_{\cB_\delta} \log\frac{ \sum_\alpha v_\alpha \int_{\Sigma_\epsilon(S,M)} e^{\sum_{i \leq N} (\ba Z_i(\alpha) x_i +\lambda x_{i}^{2}+\mu x_{i}x_{i}^{0}) } \, d \pP_X^{\otimes N} (\bx)}{ \sum_\alpha v_\alpha \int  e^{\sum_{i \leq N} (\ba Z_i(\alpha) x_i +\lambda x_{i}^{2}+\mu x_{i}x_{i}^{0}) } \, d \pP_X^{\otimes N} (\bx)}, -\kappa+\delta\right\}%\nonumber\\
		\end{eqnarray*} 
		and therefore for $\delta$ small enough (depending on $\epsilon$)
		\begin{equation}\label{tot}
			\liminf_{N \to \infty} \frac{1}{N} \E_{Z,x^{0}} \1_{\cB_\delta} \log\frac{ \sum_\alpha v_\alpha \int_{ \Sigma_\epsilon(S,M)} e^{\sum_{i \leq N} (\ba Z_i(\alpha) x_i +\lambda x_{i}^{2}+\mu x_{i}x_{i}^{0}) } \, d \pP_X^{\otimes N} (\bx)}{ \sum_\alpha v_\alpha \int  e^{\sum_{i \leq N} (\ba Z_i(\alpha) x_i +\lambda x_{i}^{2}+\mu x_{i}x_{i}^{0}) } \, d \pP_X^{\otimes N}} \ge 0.\end{equation}
		We finally can prove \eqref{lbge}. Indeed, by H\"older's inequality
		\begin{align}
			&\frac{1}{N} \E_{Z,x^{0}} \1_{\cB_\delta} \log \sum_\alpha v_\alpha \int_{\Sigma_\epsilon(S,M)} e^{\sum_{i \leq N} \ba Z_i(\alpha) x_i } \, d \pP_X^{\otimes N} (\bx)
			\nonumber\\&\geq  -\lambda S - \mu M + \E_{Z,x^{0}} \log \sum_\alpha v_\alpha  \int e^{\ba Z(\alpha) x +\lambda x^{2} +\mu x x^{0}
			} \, d \pP_X (\bx)  \\
			& +\frac{1}{N} \E_{Z,x^{0}} \1_{\cB_\delta} \log\frac{ \sum_\alpha v_\alpha \int_{\bar \Sigma_\epsilon(S,M)} e^{\sum_{i \leq N} (\ba Z_i(\alpha) x_i +\lambda x_{i}^{2}+\mu x_{i}x_{i}^{0}) } \, d \pP_X^{\otimes N} (\bx)}{ \sum_\alpha v_\alpha \int  e^{\sum_{i \leq N} (\ba Z_i(\alpha) x_i +\lambda x_{i}^{2}+\mu x_{i}x_{i}^{0}) } \, d \pP_X^{\otimes N}} +O(\epsilon)+O(\delta).\label{lbge2}
		\end{align}	
		We can finally let $N$ going to infinity, $\delta$ to zero and then $\epsilon$ to zero to conclude. 
		
		To conclude that the lower bound holds not only for exposed points we can use Rockafellar's lemma, see \cite{DZ}[Lemma 2.3.12] which shows that it is enough to show that $\Lambda$ is essentially smooth, lower semi-continuous and convex. This is clear as $\pP_{X}$ and $\pP_{0}$ are compactly supported. This implies  that the relative interior of the set of points where $\Lambda^{*}$ is finite is included in the set of exposed points, which is enough to conclude the statement of the theorem.
	\end{proof}
	\section{Proof of Theorem \ref{quenchedldp}}\label{proofthmquenched}
	In this section we prove Theorem \ref{quenchedldp} given Theorem \ref{technicalldp}.  As usual, we first prove that the rate function is good, and  then   a quenched large deviation principle. Since $(S,M)$ live in a compact space, it is enough to prove a quenched weak large deviation principle. 
	\subsection{Study of the rate function $I^{\beta}_{FP}$}
	It is enough to show that the level sets of $I^{\beta}_{FP}$ are closed, namely that $-\varphi_{\bar \ba}(S,M)$ is lower semicontinuous, since $(S,M)$ live in the compact set $[-C^{2},C^{2}]^{2}$.  But we have 
	$$	-\varphi_{\bar\ba}(S,M)=\sup_{r,\mu,\lambda,\zeta,Q} \bigg( \mu S +\lambda M  -  \E_{0} [X_0(\lambda,\mu,Q,\zeta)]  + \frac{\ba^2}{4} \sum_{k = 0}^{r - 1} \zeta_k ( Q^2_{k + 1} - Q^2_k ) - \frac{\bb M^2}{2}  - \frac{\bc S^2}{4}  \bigg),$$
	and so it is enough to show that for any fixed $r,\mu,\lambda,\zeta, Q$,  the function
	$$ f(S,M)=\mu S +\lambda M  -  \E_{0} [X_0(\lambda,\mu,S Q,\zeta)]  + \frac{\ba^2}{4} \sum_{k = 0}^{r - 1} \zeta_k ( Q^2_{k + 1} - Q^2_k ) - \frac{\bb M^2}{2}  - \frac{\bc S^2}{4} $$
	is continuous. Here we rescaled $Q$ by $S$ in order that we may assume that $Q_r = 1$ in \eqref{eq:Qseq1}, so that the sequence does not depend on $S$ anymore. The only point that we have to check is that
	$ S\rightarrow \E_{0} [X_0(\lambda,\mu,S Q,\zeta)] $ is continuous for fixed $(\lambda,\mu,Q,\zeta)$. By Lemma \ref{lem:RPCavg}
	we can write	
	\[\E_{0} [X_0(\lambda,\mu,S Q,\zeta)]:=
	\E \log \sum_{\alpha} v_\alpha e^{  \ba\sqrt{S} Z(\alpha) x + \lambda x^2 + \mu x x^0 }d\pP_{X}(x)
	\]
	with $Z(\alpha)$ a centered Gaussian process with covariance $Q_{\alpha\wedge\alpha'}$, independent from $S$, and $v_{\alpha}$ is also independent from $S$. From this formula the continuity of $S\rightarrow \E_{0} [X_0(\lambda,\mu,S Q,\zeta)] $ is clear. 
	
	\subsection{Quenched weak large deviation principle} 
	In this subsection we prove that
	\begin{lem}\label{technicalldpq} For every real numbers $\ba,\bb,\bc$, 
		every $(S,M)\in\mathcal C$, almost all $W,\bx^{0}$, 
		$$\varphi_{\bar\ba}(S,M)\le \lim_{\epsilon\downarrow 0}\liminf_{N\rightarrow\infty}\frac{1}{N} \log \int \1(|R_{1,1}-S|\le \epsilon, |R_{1,0}-M|\le\epsilon) e^{H_N^{\ba}(\bx)} \, d \pP_X^{\otimes N}(\bx)    \qquad\qquad$$
		$$\qquad\qquad\qquad \le \lim_{\epsilon\downarrow 0}\liminf_{N\rightarrow\infty}\frac{1}{N}  \log \int \1(|R_{1,1}-S|\le \epsilon, |R_{1,0}-M|\le\epsilon) e^{H_N^{\ba}(\bx)} \, d \pP_X^{\otimes N}(\bx)   \le\varphi_{\bar\ba}(S,M)\,.$$
		Moreover, by Section~\ref{sec:prooflimit} for every real numbers $\ba,\bb,\bc$, 
		and almost all $W,\bx^{0}$, 
		
		$$\liminf_{N\rightarrow\infty}\frac{1}{N} \log \int  e^{H_N^{\ba}(\bx)} \, d \pP_X^{\otimes N}(\bx) =\limsup_{N\rightarrow\infty}\frac{1}{N} \log \int  e^{H_N^{\ba}(\bx)} \, d \pP_X^{\otimes N}(\bx)=\sup \varphi_{\bar\ba}(S,M)\,.$$
	\end{lem}
	\begin{proof}This lemma is a direct consequence of Theorem \ref{technicalldp} and concentration of measure. Indeed, first notice that we may assume without loss of generality that $(S,M)$ have a finite entropy $\cI$ since otherwise the left hand side is $-\infty$ as $H_N^{\ba}(\bx)\le \beta \sqrt{N}\|W\|_{\infty} +(\bb+\bc)C^{4}$ and it is well known, see e.g \cite{AGZ}[Section 2.6.2], that there exists a positive constant $q$ so that
		$$\P(\|W\|_{\infty}\ge 3\sqrt{N})\le e ^{-qN}$$ so that
		$\|W\|_{\infty }$ is almost surely bounded by $3$. Next, we can follow the proof of Lemma \ref{lemsmooth} to see that  there exists $o(\delta)$ going to zero as $\delta$ goes to infinity (independently of the other parameters) so that on $\{\|W\|_{\infty}\le 3\sqrt{N}\}\cap\{\bx^{0}\in \cB_{\delta}\}$ we have
		$$ e^{-o(\delta)N} \int \1(|R_{1,1}-S|\le \epsilon)\chi_{M, \epsilon/2} (R_{1,0}(\bx)) e^{H_N^{\ba}(\bx)} \, d \pP_X^{\otimes N}(\bx)\qquad$$
		\begin{equation}\label{mid}\qquad \le 
			\int \1(|R_{1,1}-S|\le \epsilon, |R_{1,0}-M|\le \epsilon) e^{H_N^{\ba}(\bx)} \, d \pP_X^{\otimes N}(\bx) \qquad\end{equation}
		$$\qquad\qquad\qquad \le e^{o(\delta) N}\int \1(|R_{1,1}-S|\le \epsilon)\chi_{M,2\epsilon} (R_{1,0}(\bx)) e^{H_N^{\ba}(\bx)} \, d \pP_X^{\otimes N}(\bx)\,.$$
		We recall here that the  need for the restriction of $\bx^{0}$ to $\cB_{\delta}$ is due to the fact that we use that $(S,M)$ have finite entropy, a condition related to the fact that the empirical measure of the $\bx^{0}$ is close to $\pP_{0}$. 
		Now, for any $\epsilon,\kappa>0$
		the function 
		$$F_{N}(W,\bx^{0}):=\frac{1}{N} \log \int \1(|R_{1,1}-S|\le \epsilon)\chi^N_{M,\kappa} (R_{1,0}(\bx)) e^{H_N^{\ba}(\bx)} \, d \pP_X^{\otimes N}(\bx) $$
		is differentiable  in $W$ with derivative
		$$\partial_{W_{i,j}} F_{N}(W,\bx^{0})=\frac{\beta}{N^{3/2}}\frac{\int x_{i}x_{j } \1(|R_{1,1}-S|\le \epsilon)\chi^N_{M,\kappa} (R_{1,0}(\bx)) e^{H_N^{\ba}(\bx)} \, d \pP_X^{\otimes N}(\bx) }{\int  \1(|R_{1,1}-S|\le \epsilon)\chi^N_{M,\kappa} (R_{1,0}(\bx)) e^{H_N^{\ba}(\bx)} \, d \pP_X^{\otimes N}(\bx) }$$
		which is uniformly bounded by $\ba C^{2} N^{-3/2}$. Consequently, $W\rightarrow F_{N}(W,\by^{0})$ is Lipschitz with Lipschitz constant bounded by $\beta C^{2} N^{-1/2}$. Therefore, as a consequence of Gaussian concentration of measure we have that
		\begin{equation}\label{sa1}\pP_{W}\left(| F_{N}(W,\bx^{0})-\mathbb E_{W}[F_{N}(W,\bx^{0})]|\ge \delta \beta C^{2} \right)\le 2 e^{-\delta^{2}N}\,.\end{equation}
		Similarly as in the proof of \eqref{eq:ggisecondtermx0conc}, we see that the derivative of $x^{0}_{i}\rightarrow \mathbb E_{W}[F_{N}(W,\bx^{0})]$ is bounded for all $i\in \{1,\ldots,N\}$ (with a bound depending on $\mathcal L$ which is large but independent of $N$) so that the Azuma Hoefding's inequality insures that there exists a finite constant $B$ such that
		\begin{equation}\label{sa2}\pP_0^{\otimes N}\left(| \mathbb E_{W}[F_{N}(W,\bx^{0})]-\mathbb E_{\bx^{0}}\mathbb E_{W}[F_{N}(W,\bx^{0})]|\ge \delta \beta C^{2} \right)\le 2 e^{-\delta^{2}N}\,.\end{equation}
		Note also that $  \mathbb E_{W}[F_{N}(W,\bx^{0})]$ is uniformly bounded as $N$ goes to infinity because $S$ has finite entropy and 
		$N^{-1}\log\chi_{M,\epsilon}^{N}$ is bounded. Hence, 
		since by Sanov's theorem $\{ \bx^{0}\in\cB_{\delta}\}$ has probability greater than $1-e^{-c_{\delta}N}$ with some $c(\delta)>0$, we deduce that
		$$\E_{\bx^{0}} \mathbb E_{W}[F_{N}(W,\bx^{0})]=\E_{\bx^{0}} \1_{\cB_{\delta}}\mathbb E_{W}[F_{N}(W,\bx^{0})]+o(N)\,.$$
		Moreover, together with
		\eqref{mid}, \eqref{sa1} and \eqref{sa2} we deduce  that almost surely
		\begin{align*}
			&\lim_{\delta\downarrow 0}  \liminf_{N\rightarrow \infty}\frac{1}{N}\E_{W,\bx^{0}}\1_{\cB_{\delta}}\log  \int \1(|R_{1,1}-S|\le \epsilon)\chi_{M, \epsilon/2} (R_{1,0}(\bx)) e^{H_N^{\ba}(\bx)} \, d \pP_X^{\otimes N}(\bx)
			\\&\le \liminf_{N\rightarrow \infty}\frac{1}{N}\log \int \1(|R_{1,1}-S|\le \epsilon, |R_{1,0}-M|\le \epsilon) e^{H_N^{\ba}(\bx)} \, d \pP_X^{\otimes N}(\bx) 
			\\&\le  \limsup_{N\rightarrow\infty}\frac{1}{N}\E_{W,\bx^{0}}\log 
			\int \1(|R_{1,1}-S|\le \epsilon)\chi_{M,2\epsilon} (R_{1,0}(\bx)) e^{H_N^{\ba}(\bx)} \, d \pP_X^{\otimes N}(\bx)\,.
		\end{align*}
		We finally need to prove that for some $k$ large enough 
		\begin{multline*}
			\limsup_{N\rightarrow\infty}\frac{1}{N}\E_{W,\bx^{0}}\log 
			\int \1(|R_{1,1}-S|\le \epsilon)\chi_{M,2\epsilon} (R_{1,0}(\bx)) e^{H_N^{\ba}(\bx)} \, d \pP_X^{\otimes N}(\bx)
			\\\le \limsup_{N\rightarrow\infty} \frac{1}{N}\E_{W,\bx^{0}}\log 
			\int_{\Sigma_{k\epsilon}(S,M)} e^{H_N^{\ba}(\bx)} \, d \pP_X^{\otimes N}(\bx)
		\end{multline*}
		Indeed, we see that the left hand side is bounded by
		$$\E_{W,\bx^{0}}[F_{N}(W,\bx^{0})]\le \max\{  \frac{1}{N}\E_{W,\bx^{0}}\log 
		\int_{\Sigma_{k\epsilon}(S,M)} e^{H_N^{\ba}(\bx)} \, d \pP_X^{\otimes N}(\bx), - \cL (k\epsilon)^{2} +\mbox{const.}\}$$
		and we know by taking $N$ going to $\infty$ and then $\epsilon$ go to zero that the first term is upper bounded by some finite constant $\varphi_{\bar\ba}(S,M)$ whereas the second is of order $-k^{2}$. The conclusion follows by taking $k$ large enough. We finally can use the upper bound in Theorem \ref{technicalldp} to obtain the quenched large deviation upper bound. 
		The convergence of the free energy is a direct consequence of Theorem \ref{mainfree} and concentration of measure. 
	\end{proof}
	\section{Proof of Theorem \ref{mainfree} - The Limit of the Free Energy}\label{sec:prooflimit}
	We now prove that the annealed  weak large deviation theorem of Theorem \ref{technicalldp} allows us to control the 
	free energy $F_{N}(\beta)$ defined in \eqref{deff} and the free energy of the SK model:
	$$F_N^{SK}(\ba)=\frac{1}{N}\E[\log \int e^{\frac{\ba}{\sqrt{N}}\sum_{i<j} W_{ij} x_{i}x_{j}} d\mathbb P_{x}^{\otimes N}(\bx)]\,.$$
	We first prove that it gives the upper bound of the free energy.
	\subsection{Upper bound on the free energy} To this end  we let for $t\in [0,1]$, $\varphi^{t}$ be the function on $\mathcal C$ defined by: 
	\begin{equation}\label{defvarphi2}
		\varphi^{t}(S,M)=\inf_{r,\mu,\lambda,\zeta,Q} \bigg(- \mu S - \lambda M + \E_{0} [X_0(\lambda,\mu,Q,\zeta)[x^{0}]] - \frac{\ba^2}{4} \sum_{k = 0}^{r - 1} \zeta_k ( Q^2_{k + 1} - Q^2_k )  \bigg) + t \bigg( \frac{\bb M^2}{2}  + \frac{\bc S^2}{4} \bigg)\,,\end{equation}
	with $\varphi^{1}=\varphi$ the function defined in \eqref{defvarphi}. 
	\begin{lem}[Upper Bound of the Free Energy]\label{lem:upperboundcovering}
		For any $\bar \beta = (\ba,\bb\bc) \in \R^3$
		\[
		\limsup_{N \to \infty} F_N(\bar\ba) \leq \sup_{S,M \in \cC} \varphi(S,M)\qquad\mbox{ and }\qquad   \limsup_{N \to \infty} F_N^{SK}(\ba) \leq \sup_{S,M \in \cC} \varphi^{0}(S,M)
		\]
		
	\end{lem}
	\begin{proof} It is enough to prove one of the two cases as the proofs are identical and so we show how to bound  $F^{SK}_N(\ba) $. 
		We define the functional
		\[
		\sP_{S,M}(\mu, \lambda, \zeta, Q ) = \bigg(- \mu S - \lambda M + \E_{0} X_0 - \frac{\ba^2}{4} \sum_{k = 0}^{r - 1} \zeta_k ( Q^2_{k + 1} - Q^2_k ) \bigg)
		\]
		We fix a $\eta > 0$.
		For each $S,M \in \cC$,  we choose $(\mu_\eta(S,M), \lambda_\eta(S,M), \zeta_\eta(S,M), Q_\eta (S,M)$ so that
		\[
		\sP_{S,M}(\mu_\eta(S,M), \lambda_\eta(S,M), \zeta_\eta(S,M), Q_\eta (S,M) ) \leq \max \bigg( -\frac{1}{\eta}, \inf_{\mu,\lambda,\zeta,Q} \sP_{S,M}(\mu, \lambda, \zeta, Q ) + \eta \bigg)\,.
		\]
		We next choose  $\epsilon_{S,M}(\eta)>0$ such that $\epsilon ( |\lambda_{\eta}(S,M)| + |\mu_{\eta}(S,M)| ) \leq \eta$. We can always cover
		\[
		\cC \subset \bigcup_{S,M} \Sigma_{\epsilon_{S,M}(\eta)} (S,M)
		\]
		and the right hand side is an open cover of our compact set $\mathcal C$, so there exists a finite subcover, which we denote by $ \Sigma_{\epsilon_k} (S_k,M_k)$ for $k \leq K$. We have the obvious upper bound,
		\begin{align*}
			\frac{1}{N}\E_{Y} \log \int e^{H^{SK}_N (\bx)} \, d \pP_X^{\otimes N}(\bx) &\leq  \frac{1}{N} \log \sum_{k} \int_{\Sigma_{\epsilon_{k}(\eta)} (S_{k},M_{k})} e^{H^{SK}_N (\bx)} \, d \pP_X^{\otimes N}(\bx) 
			\\&\leq \frac{\log(K)}{N} + \E_{Y} \max_{k} \frac{1}{N} \log \int_{\Sigma_{\epsilon_{k}(\eta)} (S_{k},M_{k})}  e^{H^{SK}_N (\bx)} \, d \pP_X^{\otimes N}(\bx) .
		\end{align*}
		By Gaussian concentration based on Herbst argument and the fact that 
		$$W_{ij}\rightarrow  F_{N}(\ba)(\bx^{0},W):= \frac{1}{N} \log \int_{\Sigma_{\epsilon_{k}(\eta)} (S_{k},M_{k})}  e^{H^{SK}_N (\bx)} \, d \pP_X(\bx)$$
		is differentiable with derivative uniformly bounded by $C N^{-3/2}$ for all $i<j$, and by Talagrand's concentration \cite{talagrand} applied to the independent bounded variables  $x^{0}_{i}$ and the convex function
		$\bx^{0}\rightarrow \E_{W}F_{N}(\ba)(\bx^{0},W)$ such that $\partial_{x^{0}_{i}}\E_{W} F_{N}(\ba)(\bx^{0},W)$ is uniformly  bounded by $C/N$, 
		we deduce that 
		\[
		\frac{1}{N} \E_Y \log \int e^{H^{SK}_N (\bx)} \, d \pP_X^{\otimes N}(\bx) \leq  \frac{\log(K)}{N} +  \max_{k} \frac{1}{N} \E_Y \log \int_{\Sigma_{\epsilon_{k}(\eta)} (S_{k},M_{k})}  e^{H^{SK}_N (\bx)} \, d \pP_X^{\otimes N}(\bx)  + o(N,\eta).
		\]
		We can now apply Proposition~\ref{prop:upbd}  giving us the upper bound
		\[ \limsup_{N \to \infty} F_N(\ba) \leq \max_{k } \varphi(S_{k},M_{k})+\eta\le\sup \varphi(S,M)+\eta\]
		and conclude by letting $\eta$ going to zero.  The quenched result follows as well from concentration inequalities and Borel-Cantelli's lemma.

	\end{proof}
	\subsection{ Lower bound by localizing the free energy}
	The lower bound will be a  clear consequence of Theorem \ref{technicalldp} after localization.
	We prove that we can restrict the free energy to the localized free energy
	
	\begin{lem}[Restriction to the Localized Free Energy]\label{lemdelta}
		We have for any $\bar \beta = (\ba,\bb,\bc) \in \R^3$
		\[
		\liminf_{N \to \infty}\frac{1}{N} \E \log \int e^{H^{\bar \beta}_N(\bx)} \, d \pP^{\otimes N}_X(\bx)  \geq \liminf_{N \to \infty} \frac{1}{N} \E \1_{\cB_\delta} \log \int e^{H^{\bar \beta}_N(\bx)} \, d \pP^{\otimes N}_X(\bx)
		\]
	\end{lem}
	
	\begin{proof}
		We have the following decomposition of the free energies
		\begin{equation}\label{eq:splitFE}
			\frac{1}{N} \E \log \int e^{H^{\bar \beta}_N(\bx)} \, d \pP^{\otimes N}_X(\bx)  = \frac{1}{N} \E \1_{\cB_\delta} \log \int e^{H^{\bar \beta}_N(\bx)} \, d \pP^{\otimes N}_X(\bx) + \frac{1}{N} \E \1_{\cB^c_\delta} \log \int e^{H^{\bar \beta}_N(\bx)} \, d \pP^{\otimes N}_X(\bx).
		\end{equation}
		We will now prove that the second term is negligible in the limit. Given a realization of the Gaussian terms $W$, we have the following universal bound of the Hamiltonian ,
		\begin{equation}
			| H^{\bar \beta}_N(\bx) | \leq \max\bigg( \ba^2, \bb, \bc \bigg) \bigg( \sqrt{N} C^2 \| W \|_{op} + 2 N C^2 \bigg)
		\end{equation}
		since $|\bx^{\mathrm{T}} W \bx| \leq \| W \|_{op} \|\bx\|^2_2 \leq \| W \|_{op} C^2 N$. Because a symmetric random matrix with iid standard Gaussian entries has bounded by $2 \sqrt{N}$ asymptotically almost surely, 
		\[
		| H_N(\bx) | \leq 4  \max\bigg( \ba^2, \bb, \bc \bigg) N
		\]
		asymptotically almost surely. Notice that this upper bound only depends on the fixed model parameters. For any $L \geq 0$, we have the following decomposition of the second term in \eqref{eq:splitFE} 
		\[
		\frac{1}{N} \E \1_{\cB^c_\delta} \1_{ \| W \|_{op} \leq L \sqrt{N}  } \log \int e^{H^{\bar \beta}_N(\bx)} \, d \pP^{\otimes N}_X(\bx) + \frac{1}{N} \E \1_{\cB^c_\delta} \1_{\|W\|_{op} > L \sqrt{N}  } \log \int e^{H^{\bar \beta}_N(\bx)} \, d \pP^{\otimes N}_X(\bx).
		\]
		
		The second term can be made arbitrarily small because of the exponential control over the operator norm of a GOE matrix. To see this, we have
		\[
		\frac{1}{N} \log \int e^{H^{\bar \beta}_N(\bx)} \, d \pP^{\otimes N}_X(\bx) \leq \max\bigg( \ba^2, \bb, \bc \bigg) \bigg( \frac{1}{\sqrt{N}} C^2 \| W \|_{op} + 2 C^2 \bigg).
		\]
		We can use the tail bounds of the operator norm of random matrices with subgaussian tails \eqref{conc2},
		\[
		\pP( \|W\|_{op}  \geq t \sqrt{N}) \leq K e^{- k t N}
		\]
		for some absolute constants $k$, $K$ and all $t > K$. This implies that
		\begin{align*}
			\frac{1}{N} \E \1_{\cB^c_\delta} \1_{\|W\|_{op} > L \sqrt{N}  } \log \int e^{H_N(\bx)} \, d \pP^{\otimes N}_X(\bx) &\leq K' \bigg( \E  \1_{\|W\|_{op} > L \sqrt{N}  } \frac{\|W\|_{op}}{\sqrt{N}} + \pP( \|W\|_{op} > L \sqrt{N} ) \bigg)
			\\&=  K' \bigg( \int_{L}^\infty \pP\biggl(  \frac{ \|W\|_{op} }{\sqrt{N}} \geq t\biggr) \, dt + \pP( \|W\|_{op} > L \sqrt{N} ) \bigg)
			\\&\leq e^{ -O(LN) }.
		\end{align*}
		Repeating this argument using the lower bound
		\[
		\frac{1}{N} \log \int e^{H^{\bar \beta}_N(\bx)} \, d \pP^{\otimes N}_X(\bx) \geq -\max\bigg( \ba^2, \bb, \bc \bigg) \bigg( \frac{1}{\sqrt{N}} C^2 \| W \|_{op} + 2 C^2 \bigg)
		\]
		will imply
		\[
		\frac{1}{N} \E \1_{\cB^c_\delta} \1_{\|W\|_{op} > L \sqrt{N}  } \log \int e^{H_N(\bx)} \, d \pP^{\otimes N}_X(\bx) \geq  -e^{ -O(LN) }.
		\]
		Therefore, if we take $L > K$, we have exponential control of the second term.

		We now fix $L > K$, on the set $\{| H^{\bar \beta}_N(\bx) | \leq L N  \}$ we have the obvious upper bound
		\[
		\bigg| \frac{1}{N} \E \1_{\cB_\delta} \1_{| H^{\bar \beta}_N(\bx) | \leq L N  } \log \int e^{H^{\bar \beta}_N(\bx)} \, d \pP^{\otimes N}_X(\bx) \bigg| \leq L \pP ( \1_{\cB_\delta} )
		\]
		Since $\pP_0$ is compactly supported, Sanov's Theorem for empirical measures \cite[Theorem 2.1.10]{DZ}, it follows that there exists $c_{\delta}>0$ such that
		\[
		\pP ( {\cB_\delta}^{c} ) = O( e^{-N c_\delta} ).
		\]
		This can be made aribtrarily small by taking $N \to \infty$. 
	\end{proof}

	\begin{rem}
		The restriction of the empirical measure is essential to prove the lower bound in general. For example, if $\pP(x^0 = 0) > 0$ then the set $\{ x \mmm |R_{1,0}-M|\le\epsilon \} = \emptyset$ if $\bx^0 = \bm{0}$ and $|M| > \epsilon$. This implies that there is a positive probability with respect to $\pP^{\otimes N}_0$ that the set $\1(|R_{1,0}-M|\le\epsilon) = 0$ which will mean that the lower bound will be $-\infty$. This will imply that the lower bound will always be $-\infty$ which gives us a non-sharp lower bound. 
	\end{rem}
	
	We can finally deduce the lower bounds on the free energy now that it is localized since by Lemma \ref{lemdelta}, we have
	\begin{eqnarray*}
		\liminf_{N \to \infty}F_{N}(\beta)  &\geq& \liminf_{N \to \infty} \frac{1}{N} \E \1_{\cB_\delta} \log \int e^{H_N(\bx)} \, d \pP^{\otimes N}_X(\bx)
		\\
		&\geq& \liminf_{N \to \infty} \frac{1}{N} \E \1_{\cB_\delta} \log \int_{\Sigma_{\epsilon}(S,M)} e^{H_N(\bx)} \, d \pP^{\otimes N}_X(\bx)
		\\
	\end{eqnarray*}
	where $(S,M)$ has finite entropy and $\epsilon$ is a positive real number. The conclusion follows by Theorem \ref{technicalldp}.

	\iffalse
	\appendix

	\section{Cavity Computations}
	\begin{lem}\label{lem:sequenceCavity}
		For any sequence, $(a_N)_{N \geq 1}$, of real numbers and any $n \geq 1$,
		\[
		\liminf_{N \to \infty} \frac{a_N}{N} \geq \liminf_{N \to \infty} \frac{a_{N+n}-a_N }{n} 
		\]
	\end{lem}
	\begin{proof}
		The statement is trivial if the right hand side is negative infinity, so we may assume that 
		\[
		\liminf_{N \to \infty}  \frac{(a_{N+n}-a_N )}{n}  = L \in (-\infty,\infty]
		\]
		For every $\ell < L$, there exists a $N_0(n)$ such that for $N\geq N_0$,
		\[
		a_{N+n}-a_N > n \ell.
		\]
		Now write $n_0=k_0 n+r_0$. Observe that for any $N$ we may write $N= k_N n+r_N$ and $0 \leq r_N < n$, and for
		$N$ sufficiently large, $k_N \gg k_0$. This means that
		\begin{align*}
			\frac{a_N}{N}=\frac{a_{k_N n+r_N}}{k_N n+r_N} &= \frac{a_{k_Nn+r_N}-a_{(k_N-1)n+r_N}+\ldots + a_{ (k_0 + 2) n+r_N} -a_{(k_0 + 1) n+r_N}}{k_N n+r_N} + \frac{a_{(k_0 + 1) n+r_N}}{k_N n+r_N}
			\\&>\frac{1}{N}\left[(k_N -(k_0 + 1)) n \ell +a_{k_0m+r_N}\right].
		\end{align*}
		Taking the limit inferior  on both sides, we see that 
		\[
		\liminf_{N \to \infty} \frac{a_N}{N} > \liminf_{N \to \infty} \frac{k_N n}{k_Nn + r_N} \ell = \ell.
		\]
		Taking $\ell \to L$ finishes the proof.
	\end{proof}
	
	\fi
	
	\bibliographystyle{amsplain}

\end{document}